\documentclass{article}

\usepackage[top=1in, bottom=1.25in, left=1.4in, right=1.25in]{geometry}

\usepackage[utf8]{inputenc}
\usepackage[english]{babel}

\usepackage[breaklinks]{hyperref}
\hypersetup{%
  pdftitle={\@title},            
  pdfauthor={\@author},          
  unicode=true,           
  colorlinks=true,       
  linkcolor=blue,         
  citecolor=blue,      
  filecolor=green,        
  urlcolor=blue           
}
\usepackage{xr-hyper}

\usepackage[usenames,dvipsnames]{xcolor}

\usepackage[square,numbers]{natbib}

\usepackage{amsmath,amsfonts,amssymb,amsthm}
\usepackage{bm}
\usepackage{mathrsfs}
\usepackage{stmaryrd}
\usepackage{bibunits}
\usepackage{titling}
\usepackage{tikz}

\theoremstyle{plain}
\newtheorem{theorem}{Theorem}[section]
\newtheorem{proposition}[theorem]{Proposition}
\newtheorem{lemma}[theorem]{Lemma}

\theoremstyle{remark}
\newtheorem{definition}[theorem]{Definition}

\tikzset{vertex/.style = {shape=circle,draw,minimum size=1.5em, thick}}
\tikzset{edge/.style = {->, thick}}
\usetikzlibrary{arrows.meta,
                decorations.pathreplacing,
                calligraphy,
                positioning}
\usetikzlibrary{fit}

\tikzstyle{boldvertex}=[draw, circle, minimum size=1.3em, inner sep=0pt, line width=1.3pt]
\tikzstyle{dottedvertex}=[draw, circle, minimum size=1.3em, inner sep=0pt, dash pattern=on 2pt off 2pt]
\tikzstyle{doublevertex}=[draw, circle, minimum size=1.3em, inner sep=0pt, double, line width=0.5pt]
\tikzstyle{normalvertex}=[draw, circle, minimum size=1.3em, inner sep=0pt]

\newcommand{\V}{\mathrm{V}}
\newcommand{\G}{G}

\newcommand{\PP}{\mathbb{P}}
\newcommand{\FF}{\mathbb{F}}
\newcommand{\PPn}{\mathbb{P}^{n}_{0}}
\newcommand{\PPa}{\mathbb{P}^{n}_{1}}
\newcommand{\EE}{\mathbb{E}}
\newcommand{\1}{\mathbf{1}}

\DeclareMathOperator*{\argmax}{arg\,max}

\newcommand{\bigO}[1]{\ensuremath{\mathop{}\mathopen{}O\mathopen{}\left(#1\right)}}
\newcommand{\smallO}[1]{\ensuremath{\mathop{}\mathopen{}o\mathopen{}\left(#1\right)}}

\newcommand{\g}{\mathfrak{g}}
\newcommand{\ug}{\mathfrak{u}}
\newcommand{\VS}{\mathsf{V}}
\newcommand{\ES}{\mathsf{E}}
\newcommand{\Par}{\mathsf{P}}
\newcommand{\Chi}{\mathsf{C}}
\newcommand{\Din}{\mathsf{d}^{\mathrm{in}}}
\newcommand{\Dout}{\mathsf{d}^{\mathrm{out}}}
\newcommand{\D}{\mathsf{d}}
\newcommand{\intd}{\mathrm{d}}
\newcommand{\NNInts}{\mathbb{Z}_{+}}
\newcommand{\Nats}{\mathbb{N}}

\newcommand{\PartB}{\tilde{\mathcal{V}}}

\begin{document}

\defaultbibliography{Bibliography}


\begin{bibunit}[plain]

  \title{On the impossibility of detecting a late change-point in the preferential attachment random graph model}

\author{%
  Ibrahim Kaddouri, Zacharie Naulet, {\'Elisabeth Gassiat}\\
  Université Paris-Saclay, CNRS, Laboratoire de mathématiques d’Orsay\\
  91405, Orsay, France}
\date{}

\maketitle

\begin{abstract}
  We consider the problem of late change-point detection under the preferential attachment random graph model with time dependent attachment function. This can be formulated as a hypothesis testing problem where the null hypothesis corresponds to a preferential attachment model with a constant affine attachment parameter $\delta_0$ and the alternative corresponds to a preferential attachment model where the affine attachment parameter changes from $\delta_0$ to $\delta_1$ at a time $\tau_n = n - \Delta_n$ where $0\leq \Delta_n \leq n$ and $n$ is the size of the graph. It was conjectured in \cite{BBCH23} that when observing only the unlabeled graph, detection of the change is not possible for $\Delta_n = o(n^{1/2})$. In this work, we make a step towards proving the conjecture by proving the impossibility of detecting the change when $\Delta_n = o(n^{1/3})$. We also study change-point detection in the case where the labeled graph is observed and show that change-point detection is possible if and only if $\Delta_n \to \infty$, thereby exhibiting a strong difference between the two settings.


\end{abstract}

\section{Introduction}
\label{sec:introduction}

Empirical studies carried out on networks modeling different types of interactions have revealed striking similarities between them. In many situations, these networks are scale-free, i.e. their empirical degree distribution generally follows a power law. This was observed in many networks such as citation networks \cite{BJNRSV02, N01}, internet \cite{FFF99} and the World Wide Web \cite{AH00}. On the other hand, the typical distances between vertices in these networks are small (see the books \cite{Watts06, Watts99}). This is generally referred to as the small-world phenomenon. Motivated by these observations, the preferential attachment random graph model was proposed to mathematically model scale-free networks. It provides a simple and intuitive mechanism for generating networks with a power-law degree distribution. The model helps in understanding how networks evolve over time by showing that vertices with higher degrees tend to attract more links, leading to the rich-get-richer phenomenon. This mirrors many real-world situations where popular entities tend to become even more popular over time. The first preferential attachment model to emerge was the Barabási-Albert model \cite{BA99}. In this model, new vertices are added to the network one at a time, and each new vertex gets attached to existing vertices with a probability proportional to their current degree. In \cite{BBCC03, CWR23, GVCH17}, variants of this model were proposed and they depend mainly on the attachment function which can be linear, nonlinear, constant in time or time-varying. Recently, there has been notable interest in investigating time-varying networks \cite{ZCL19, WYR18, MCG11, HS19}, i.e. networks where the attachment function is not constant over time. These networks usually involve a set of parameters that describe the time evolution of the network. Within this framework, an important question is to understand the effect of abrupt changes in these parameters on the degree distribution and how these changes can be detected and localized. Our work focuses on the situation where the growth dynamics of the network might undergo at most one change at some point of time. To model this, a time-inhomogeneous affine preferential attachment model is used. In this model, a new vertex entering the graph at time $t\in\llbracket 2, n\rrbracket$ connects to an existing vertex with degree $k$ with probability proportional to $f(k)=k+\delta(t)$ where $\delta(t)$ is the parameter likely to change at a given time. In particular, we are interested in late change-point detection, specifically when the change-point is given by $\tau_n = n - \Delta_n$ with $\Delta_n = o(n)$. This scenario is important for detecting changes as quickly as possible. Understanding this context will highlight the fundamental limits of change point detection and provide an estimate of the minimum number of vertices that must be observed between the moment the change took place and the moment it is detected. In \cite{BBCH23}, the authors built a test based on low degree vertices which was shown to detect the change only when $\frac{\Delta_n}{n^{1/2}} \to \infty$. They conjectured that when $\Delta_n = o(n^{1/2})$ and based only on the unlabeled graph, detection of the change is not possible. In light of this framework, this paper has two goals: (i) Prove the conjecture holds at least for $\Delta_n = o(n^{1/3})$, and (ii) Study the problem of change-point detection in the situation where the labeled graph is observed. More precisely, below is an informal statement of our main results.

\begin{theorem}[Informal]
  \label{thm:informal:contig}
    Using the unlabeled preferential attachment random graph, detection of the change-point is not possible when $\Delta_n = o(n^{1/3})$.
\end{theorem}

\begin{theorem}[Informal]
  \label{thm:informal:labeled}
  Using the labeled preferential attachment random graph, detection of the change-point is possible if and only if $\Delta_n \to \infty$.
\end{theorem}

The formal statement of Theorem~\ref{thm:informal:contig} is given later in the paper by Theorem~\ref{thm:main}, while the formal statement of Theorem~\ref{thm:informal:labeled} is given by Theorem~\ref{thm:labeled:change_detection}.

In what follows, Section~\ref{Section_2} introduces the notations and defines the model and the attachment mechanism. Section~\ref{Section_3} presents the main results of the paper. Section~\ref{Section_4} is devoted to the discussions and perspectives while Appendices~\ref{sec:likel-likel-ratio} and~\ref{Section_5} detail the proofs for the unlabeled model. The article comes with a supplementary material \cite{kng:supplement} that contains the proofs and additional results in the case of labeled observations.

\subsection{Related work}

This work is a continuation of \cite{BBCH23} where the problem of late change-point detection ($\tau_n = n - \lfloor cn^{\gamma}\rfloor$) was studied and a test was built for detecting the change when $\gamma\in\left(1/2, 1\right)$. The idea behind this test is that the variations in the number of vertices with minimal degree around its asymptotic value exhibit different magnitudes under the null hypothesis compared to the alternative hypothesis. They also conjectured that no test is capable of detecting the change when $\gamma\in\left(0, 1/2\right)$. In \cite{BJN18, BBC23}, the authors considered the problems of change-point detection and localization, but they focused mainly on the situation of early change-point, that is when changes occur at $\tau_n = \alpha n$ with $\alpha\in\left(0, 1\right)$. Detection of the change was shown to be always possible in this setting and a non-parametric consistent estimator of $\alpha$ was devised, allowing in addition to detection for localization of the change-point. Similarly, a likelihood-based methodology for change-point localization was proposed in \cite{CWZ23}. In \cite{BBC23}, a different regime of early change-detection was studied. It corresponds to the situation where $\tau_n = \lfloor cn^{\gamma}\rfloor$ with $\gamma\in\left(0, 1\right)$ and $c > 0$. Unlike the case of late change, the test used in this regime is based on maximal degrees. This is because, while the asymptotic degree distribution does not depend on the parameter $\gamma$, the distribution of the maximal degree does. A similar phenomenon was noted in \cite{BMR15}, which demonstrated that the influence of the seed graph (the initial subgraph from which the preferential attachment graph originates) persists as the number of vertices increases to infinity.  In the absence of any change-point, the general problem of estimation of general attachment functions was already studied in \cite{GVCH17}. This problem reduces to a simple parametric estimation in the case of affine preferential attachment. The estimation can be done using the MLE as shown in \cite{BBCH23}. Consistency and asymptotic normality of this estimator were proved in the more general setting of random initial degrees in \cite{GV17}. Finally, let us mention that some ideas underlying our proof are reminiscent to the arguments employed by \cite{briend2024estimating} to derive minimax lower bounds for the problem of estimating the order of arrival of the vertices in the unlabeled preferential attachment tree.

\section{Setting, definitions and notations}
\label{Section_2}

\subsection{Labeled versus unlabeled graphs, structure}
\label{sec:label-vers-unlab}

The preferential attachment mechanism introduced in Section~\ref{sec:introduction} (see also next section for more details) defines a sequence of random multigraphs $(G_t)_{t\geq 1}$ on vertex sets $\{0,\dots,t\}$. There is no loss of generality in assuming that these graphs are directed, using the convention that the arrows go from vertices with largest labels to vertices with smallest labels. To be somewhat more precise, in the next, a \textit{labeled graph} refers to the following definition:

\begin{definition}[Labeled graph]
  \label{def:labeled_graph} 
  A labeled (multi)graph $\g$ is a couple $\left(\mathcal{V}, \mathcal{E}\right)$ where $\mathcal{V}$ is the set of vertices and $\mathcal{E}\subset\mathcal{V}^{2}$ is the multiset of directed edges, with no loop allowed. For an edge $(u,v) \in \mathcal{E}$, we use the convention that the arrow goes from $u$ to $v$, and we write for simplicity $u\to_{\g} v$ for $(u,v)\in \mathcal{E}$.
\end{definition}

Given a labeled graph $\g = (\mathcal{V},\mathcal{E})$, we define for convenience $\VS(\g) = \mathcal{V}$ the vertex set of $\g$, and $\ES(\g) = \mathcal{E}$ the edge multiset of $\g$. Note that in a multigraph, two vertices can be connected by more than one edge. We count the multiplicity of edges via the function $\mu_{\g}: \VS(\g)^2 \to \NNInts$, such that $\mu_{\g}(u,v) = k$ means that there are $k$ directed edges $u\to v$ in $\g$ (with possibly $k=0$). The set $\Par_{\g}(u) = \{v\in \VS(\g) \;:\; v \to_{\g} u\}$ are the in-neighbors of vertex $u\in \VS(\g)$ (\textit{aka}. parents) and $\Chi_{\g}(u) = \{v\in \VS(\g) \;:\; u \to_{\g} v\}$ are the out-neighbors of vertex $u\in \VS(\g)$ (\textit{aka}. children). The in-degree of $u \in \VS(\g)$ is written $\Din_{\g}(u) = \sum_{v\in \Par_{\g}(u)}\mu_{\g}(v,u)$, the out-degree is $\Dout_{\g}(u) = \sum_{v\in \Chi_{\g}(u)}\mu_{\g}(u,v)$, and the degree is  $\D_{\g}(u) = \Din_{\g}(u) + \Dout_{\g}(u)$. For a subset $S \subset \V(\g)$ we denote by $\g[S]$ the induced sub(multi)graph of $S$, \textit{ie}. $\g[S] = (S,\mathcal{E}[S])$ where $\mathcal{E}[S]$ is the multiset obtained from $\ES(\g)$ by deleting edges that have an endpoint not in $S$.

In order to define unlabeled graphs, we require the following definition of an isomorphism of multigraphs.
\begin{definition}[Graph isomorphism]
  Let $\g$ and $\g'$ be two labeled graphs. An isomorphism $\phi$ between $\g$ and $\g'$ is a bijective map $\phi : \VS(\g) \mapsto \VS(\g')$ that preserves the set of neighbors of each vertex. More precisely, for vertices $v, w\in \VS(\g)$, and $k\in \NNInts$:
  \begin{equation*}
    \mu_{\g}(u, v) = k%
    \iff%
    \mu_{\g'}(\phi(u), \phi(w)) = k.
  \end{equation*}%
\end{definition}
In the next, $\g \cong \g'$ will denote the fact that $\g$ and $\g'$ are isomorphic, \textit{ie.} there exists an isomorphism between $\g$ and $\g'$. We are now in position to define \textit{unlabeled graphs}.
\begin{definition}[Unlabeled graph]
  \label{def:graph:unlabeled}
  An unlabeled graph $\ug$ is an isomorphism class of labeled graphs (for the relation $\cong$ defined above).
\end{definition}

An important aspect in our work is that we consider the model where only the unlabeled version of the preferential attachment is observed; \textit{ie}. only the \textit{structure} of the graph is available to the statistician:

\begin{definition}[Structure]
  \label{def:structure}
  Let $\g$ be a labeled graph. The unlabeled graph associated to $\g$, which will be denoted $s(\g)$, is the equivalence class of labeled graphs that are isomorphic to $\g$, ie.
  \begin{equation*}
    s(\g) = \{\g^{\prime}\;:\; \g^{\prime} \cong \g\}.
  \end{equation*}
\end{definition}

\subsection{Formal statement of the problem}
\label{sec:PA_mechanism}

Using the vocabulary defined in Section~\ref{sec:label-vers-unlab}, the preferential attachment model produces a sequence $(G_t)_{t\geq 1}$ of random labeled graphs, which we now intend to define rigorously. Let $m\in\mathbb{N} = \{1,2,\dots\}$ and $\delta\colon \Nats \to (-m, +\infty)$. The process $(G_t)_{t\geq 1}$ of interest is better described by introducing the intermediate process $((G_{t,i})_{i=0}^m)_{t\geq 1}$, constructed as follows. For $t=1$ let $G_{1,0}$ be the graph consisting of two isolated vertices labeled $0$ and $1$. Then for $i=1,\dots,m$, $G_{1,i}$ is obtained from $G_{1,i-1}$ by adding an edge between vertices $0$ and $1$. For $t\geq 2$, the sequence $(G_{t,i})_{i=0}^m$ is obtained by letting $G_{t,0}$ be the graph $G_{t-1,m}$ together with an isolated vertex with label $t$; and, for $i=1,\dots,m$, $G_{t,i}$ is obtained from $G_{t,i-1}$ by adding an edge directed from $t$ towards a randomly chosen vertex $V_{t,i}$ in $\{0,\dots,t-1 \}$ sampled according to the probabilities that $V_{t,i} = v$ (conditionally to $G_{t,i-1}$) given by
\begin{equation}
  \label{PA_model}
  \frac{\D_{\G_{t, i-1}}(v) + \delta(t)}{\sum_{v^{\prime}=0}^{t-1}\big(\D_{\G_{t, i-1}}(v^{\prime}) + \delta(t)\big)} = \frac{\D_{\G_{t, i-1}}(v) + \delta(t)}{2m(t-1) + \delta(t) t + (i-1)}.
\end{equation}
Finally, the process $(G_t)_{t\geq 1}$ is obtained from the process $((G_{t,i})_{i=0}^m)_{t\geq 1}$ by setting $G_t = G_{t,m}$ for each $t\geq 1$. Otherwise said, the process $(G_t)_{t\geq 1}$ is obtained from the intermediate process by forgetting the order of arrivals of the $m$ edges added at every time step $t\geq 1$.

The aim of this work is to find evidence in the preferential attachment graph that the value of $\delta$ has changed at a given time or not, using solely the information contained in the unlabeled graph $s(G_n)$ at time $n$. We are interested in the situation where the value of $\delta$ changes at most once. This can be formulated as a simple hypothesis testing problem:
\begin{equation*}
  (H_0)\;:\;\delta(t) = \delta_0,\qquad\qquad (H_1)\;:\;\delta(t) = \delta_0\1_{t\leq \tau_n} + \delta_1\1_{t>\tau_n}
\end{equation*}
where $1 \leq \tau_n\leq n$, $\delta_0 \in (-m,+\infty)$ and $\delta_1 \in (-m,+\infty)$ are known. As in \cite{BBCH23}, we are interested only in the situation of late change-points, that is the situation where $\tau_n = n - \Delta_n$ for $\Delta_n = o(n)$. \cite{BBCH23} constructed a sequence of tests $(\phi_n)_{n\geq 1}$ with vanishing Type I and Type II errors when $\frac{\Delta_n}{n^{1/2}}\to \infty$. They conjectured that using \textit{only the unlabeled random graph}, change point detection becomes impossible when $\Delta_n = o(n^{1/2})$. This work proves the conjecture holds at least for $\Delta_n = o(n^{1/3})$. We prove that even if the model parameters $\tau_n$, $\delta_0$ and $\delta_1$ are known, detection of the change is still not possible (and hence also impossible when they are unknown).

In the sequel, for each $n\geq 1$, $(\Omega_n,\mathcal{F}_n,\PP_0^n)$ (respectively $(\Omega_n,\mathcal{F}_n,\PP_1^n)$) is a probability space that is rich enough to define the beginning of the sequence of intermediate graphs $((G_{t,i})_{i=0}^m)_{t=1}^n$ under the hypothesis $H_0$ (resp. $H_1$). Expectation under $\PP_0^n$ (respectively $\PP_1^n$) is denoted by $\EE_0^n$ (resp. $\EE_1^n$).

\subsection{Further Notations}

Besides the notations and conventions defined in previous sections, we make use of the following. For real numbers $x,y$ we write $x\wedge y = \min(x,y)$ and $x\vee y = \max(x,y)$. For sequences of real numbers, $a_n\sim b_n$ means that $a_n/b_n$ converges to $1$, $a_n =o(b_n)$ means that $a_n/b_n $ converges to $0$,  $a_n = O(b_n)$ means that $a_n/b_n$ is asymptotically bounded, $a_n \lesssim b_n$ is equivalent to $a_n = O(b_n)$ and $a_n \asymp b_n$ means that $\left(\exists \alpha, \beta > 0\right)~ \alpha a_n \leq b_n \leq \beta a_n$. We write $\sigma(X_1,\dots,X_n)$ the $\sigma$-field generated by random variables $(X_1,\dots,X_n)$.

\section{Main results}
\label{Section_3}

\subsection{The observation is the unlabeled graph}

We first consider the situation where only the unlabeled graph is observed. The following theorem establishes the conjecture in some regimes of the parameters $\Delta_n$ and $(\delta_0, \delta_1)$ which are assumed to be known.

\begin{theorem}
  \label{thm:main}
  If $\delta_0 > 0$ and $\Delta_n = o(n^{1/3})$ [or $\delta_0 = 0$ and $\Delta_n = o\big(\frac{n^{1/3}}{\log(n)} \big)$], then for every sequence of events $\left(A_{n}\right)_{n\geq 1}$ with $A_n \in \sigma( s(G_n) )$ for all $n\geq 1$, 
  \begin{equation*}
     \PP^{n}_{0}(A_n) \to 0 \implies \PP^{n}_{1}(A_n) \to 0.
   \end{equation*}
\end{theorem}

In other words, under the assumptions of the theorem, the laws of $(s(G_n))_{n\geq 1}$ under $H_1$ are \textit{contiguous} to those under $H_0$. By Le Cam's first lemma \cite[Section~6.2]{vaart98}, no (eventually randomized) test made on the basis of observing $s(G_n)$ is capable of controlling both Type I and Type II error rates simultaneously: if $(\phi_n)_{n\geq 1}$ is a sequence of $s(G_n)$-measurable tests such that $\EE_0^n(\phi_n) \to 0$ then $\EE_1^n(\phi_n) \to 0$ as well. Note that a consequence of this result is that even if the model parameters are known, detection is still not possible which is a stronger result than if the model parameters are unknown. A sketch of the proof of the theorem is given in the next section.

\subsection{Sketch of proof of Theorem \ref{thm:main}}\label{sec:sketch}

\subsubsection{Difficulties in proving contiguity}
\label{sec:diff-solv-probl}

Let us for simplicity denote $Q_j^{n,s}  = \PP_j^n\circ (s \circ G_n)^{-1}$ the law of $s(G_n)$ under hypothesis $H_{j}$. The statement in Theorem~\ref{thm:main} is equivalent to the contiguity of $(Q_1^{n,s})_{n\geq 1}$ with respect to $(Q_0^{n,s})_{n\geq 1}$. A well-known sufficient condition for establishing contiguity is that the second moment of the likelihood ratio $\frac{\intd Q_1^{n,s}}{\intd Q_0^{n,s}}$ remains bounded as $n\to \infty$. Understanding this likelihood ratio is, however, not a simple task. To see why, observe that for a given unlabeled graph $\ug_n$ on $n+1$ vertices we do have
\begin{equation}
  \label{eq:likelihood:unlabeled}
    \PP^n_{\ell}(s(G_n) = \ug_n) = \sum_{\substack{\g \in \mathfrak{u}_n\\\VS(\g) = \llbracket 0,n\rrbracket}}\PP^{n}_{\ell}(G_n = \g),\qquad \ell=0,1.
\end{equation}
Though $\PP_{\ell}^n(G_n = \g)$ is easy to evaluate when $\PP_{\ell}^n(G_n = \g) > 0$ (see lemmas~\ref{lem:lkl:null} and~\ref{lem:lkl:alt}), it is much more delicate for an arbitrary unlabeled graph $\ug_n$ to understand which of the terms in the summation of \eqref{eq:likelihood:unlabeled} is non-zero. Indeed, if $\PP^{n}_{\ell}(G_n = \g) > 0$ and there is an edge $u \to_{\g} v$, then the graph $\g^{\prime}$ obtained from $\g$ by swapping the labels $u$ and $v$ has the same structure as $\g$ while $\PP^{n}_{\ell}(\G_n = \g^{\prime}) = 0$. This is because in the preferential attachment mechanism, arrows can only go from the largest label to the smallest. So to understand the likelihood of $s(G_n)$, it is required to understand the intersection of $\ug_n$ with the support of the law of $G_n$, which turns out to be rather challenging. Instead, we prefer to reduce the problem to a simpler one, as we explain in the next section.

\subsubsection{Problem reduction}
\label{sec:reduction}

Informally, problem reduction consists of analyzing a simpler problem where the observation is richer than the structure, but where detection is still not possible. The first natural reduction to examine is the situation where the labeled graph $G_n$ is observed. Unfortunately, we will show in Section \ref{sec:labeled} that in this case, change detection is always possible whenever $\Delta_n\xrightarrow{}+\infty$ and $n-\Delta_n\xrightarrow{}+\infty$ and that this reduction is therefore useless for our proof. Consequently, We are bound to look for an intermediate problem where the observation is richer than the structure, but not as informative as the labeled graph. The main idea of the proof is that if we show that change-point detection is impossible in this (easier) problem, then this should imply that it is also impossible in the original problem where only the structure is observed. Such an intermediate problem is offered by certain random permutations of the graph. We use the following definition of a permuted graph.

\begin{definition}[Permutation of a labeled graph]
  Let $\g$ be a labeled graph and $\pi$ a permutation of $\VS(\g)$. We call $\pi(\g)$ the labeled graph obtained by the application of permutation $\pi$ to the vertices of the graph $\g$. In other words $\VS(\pi(\g)) = \VS(\g)$ and for vertices $u,v\in \VS(\g)$ and $k\in \NNInts$
  \begin{equation*}
    \mu_{\g}(u,v) = k \iff \mu_{\pi(\g)}(\pi(u), \pi(v)) = k.
  \end{equation*}
\end{definition}
One might question the reasoning behind reducing the problem to a random permutation of the graph. In fact, for $\mathfrak{g}$ a preferential attachment graph, the likelihood of $s(\mathfrak{g})$ coincides up to a multiplicative term with that of $\pi(\mathfrak{g})$ which is the graph $\mathfrak{g}$ to which a random uniform permutation $\pi$ is applied. Of course reducing to a uniform permutation of the graph is helpless since it is statistically equivalent to observe the unlabeled graph, but this motivates the use of other random, simpler, permutations. In particular, our random permutation will be chosen to be uniform over a distinguished subset of vertices. Permuting part of the vertices is statistically equivalent to hiding their labels, which is the purpose sought through the reduction of the original problem.

From now on, it is assumed that the spaces $(\Omega_n,\mathcal{F}_n,\PP_0^n)$ and $(\Omega_n,\mathcal{F}_n,\PP_1^n)$ are rich enough to define $((G_{t,i})_{i=0}^m)_{t=1}^n$ jointly with a random permutation $\pi_n$ of $\llbracket 0,n\rrbracket$; the details of which are given below. The situations where one observe $G_n$, $\pi_n(G_n)$, or $s(G_n)$, are of increasing difficulty since one observes less and less information related to the labeled graph. Detection of the change should become more and more difficult. The following lemma confirms this insight, provided the conditional distributions of $\pi_n$ given $G_n$ are the same under $\PP_0^n$ and $\PP_1^n$.

\begin{lemma}
  \label{lem:increasing_difficulty}
  Let $\mathcal{G}_n$ denote the set of all labeled graphs on vertex set $\llbracket 0,n\rrbracket$ and $\mathcal{S}_n$ denote the set of all permutations of $\llbracket 0,n\rrbracket$. Suppose there is a Markov Kernel $K_n : \mathcal{G}_n \times 2^{\mathcal{S}_n} \to [0,1]$ such that both $\PP_0^n$ and $\PP_1^n$ admit $K_n$ as conditional distribution of $\pi_n$ given $G_n$. Consider the following propositions:
  \begin{enumerate}
    \item For every sequence $(A_n)_{n\geq 1 }$ of $G_n$-measurable sets, $\PP^{n}_{0}(A_n)\to 0 \implies \PP^{n}_{1}(A_n)\to 0$.
    \item For every sequence $(A_n)_{n\geq 1}$ of $\pi_n(G_n)$-measurable sets, $\PP^{n}_{0}(A_n)\to 0 \implies \PP^n_{1}(A_n) \to 0$.
    \item For every sequence  $(A_n)_{n\geq 1}$ of $s(G_n)$-measurable sets, $\PP^{n}_{0}(A_n) \to 0 \implies \PP^n_{1}(A_n) \to 0$.
  \end{enumerate}
  Then $1\implies 2 \implies 3$.
\end{lemma}

See Appendix~\ref{sec:proof-lemma-increasing-difficulty} for the proof of Lemma~\ref{lem:increasing_difficulty}.
In what follows, we will consider the reduction where one observes $\pi_n(\G_{n})$ in place of $s(G_n)$. Letting $Q_0^{n,p}$ (respectively $Q_1^{n,p})$ denote the law of $\pi_n(G_n)$ under the null hypothesis (resp. the alternative hypothesis), a mere change of variable followed by an application of Cauchy-Schwarz shows that for any events $A_n,B_n \in \sigma(\pi_n(G_n))$
\begin{equation*}
  \PP_1^n(A_n)%
  \leq \PP_1^n(B_n^c) + \PP_0^n(A_n)^{1/2}\EE_0^n\Bigg[\Bigg(\frac{\intd Q_1^{n,p}}{\intd Q_0^{n,p}}(\pi_n(G_n)) \Bigg)^2\1_{B_n} \Bigg]^{1/2}.
\end{equation*}
Hence, if we build a sequence of kernels $(K_n)_{n\geq 1}$ and events $(B_n)_{n\geq 1}$ in $\sigma(\pi_n(G_n))$ such that
\begin{gather*}
  \PP_1^n(B_n^c) \to 0,\qquad \mathrm{and},\qquad
  \limsup_{n\to \infty} \EE_0^n\Bigg[\Bigg(\frac{\intd Q_1^{n,p}}{\intd Q_0^{n,p}}(\pi_n(G_n)) \Bigg)^2\1_{B_n} \Bigg] < +\infty,
\end{gather*}
then $2$ of Lemma~\ref{lem:increasing_difficulty} holds, which by said lemma implies the validity of our theorem. We build $(K_n)_{n\geq 1}$ and $(B_n)_{n\geq 1}$ in the next section.

\subsubsection{Construction of the Markov kernel \texorpdfstring{$K_n$}{} and the event \texorpdfstring{$B_n$}{}}

We first remark that, when building $K_n(\g_n,\cdot)$, it is enough to consider $\g_n$ in $\mathfrak{S}_n = \{\g_n' \;:\; \PP_0^n(G_n = \g_n') \ne 0 \} = \{\g_n' \;:\; \PP_1^n(\G_n = \g_n') \ne 0 \}$. We give a characterization of the set $\mathfrak{S}_n$ in Lemma~\ref{lem:support}. Remark that all graphs in $\mathfrak{S}_n$ have vertex set $\llbracket 0,n\rrbracket$.

To construct $K_n$ and $B_n$, we first define the following set of vertices of a labeled graph $\g_n$, which corresponds to the vertices illustrated in \textit{bold} in Figure~\ref{fig:PA:example_1}:
\begin{equation*}
  \PartB(\g_n)%
  = \Big\{ v\in \llbracket \tau_n'+1,n\rrbracket \;:\; \D_{\g_n}(v) = m,\ %
  \forall w\in \Chi_{\g_n}(v)\ w\leq\tau_n' \text{ and } \Par_{\g_n}(w)\backslash \{v\} \subset \llbracket 0,\tau_n'\rrbracket\Big\}.
\end{equation*}
In the previous definition $(\tau_n')_{n\geq 1}$ is a sequence of integer numbers to be chosen accordingly later, but satisfying $0 \leq \tau_n' < \tau_n$. In other words, $\PartB(G_n)$ contains the late vertices of $G_n$ which have minimal degree and are the unique late parent of their children. We then consider permutations which leave invariant the labels not in $\PartB(G_n)$; \textit{ie.} letting $\mathcal{S}_n$ the set of all permutations of $\llbracket 0,n\rrbracket$ we define
\begin{equation*}
  \Pi_n(\g_n) = \left\{\pi\in\mathcal{S}_{n}\;:\; \forall i \notin \PartB(\g_n),\ \pi(i) = i \right\},
\end{equation*}
and we define $K(\g_n,\cdot)$ as the uniform distribution over $\Pi_n(\g_n)$, for all $\g_n\in \mathfrak{S}_n$. We note that one of the advantages of this permutation scheme is that $\pi(\g_n) \in \mathfrak{S}_n$ for any $\g_n\in \mathfrak{S}_n$ and $\pi \in \Pi_n(\g_n)$ (see Lemma~\ref{lem:perm:invariant-support}), which precludes the issues mentioned in Section~\ref{sec:diff-solv-probl}. Then, we consider the sequence of events
\begin{equation*}
  B_n%
  = \Big\{ |\PartB(G_n)| \geq \Delta_n'\Big(1 - \frac{\alpha_n \Delta_n'}{\tau_n'} \Big),\quad%
  \llbracket \tau_n+1,n\rrbracket \subset \PartB(G_n) \Big\}
\end{equation*}
for a sequence $(\alpha_n)_{n\geq 1}$ diverging slowly to infinity, and where $\Delta_n' = n - \tau_n'$. We note that by construction $\PartB(G_n) = \PartB(\pi_n(G_n))$, so that $B_n \in \sigma(\pi_n(G_n))$ as required (see previous section). On the event $B_n$ all the late vertices of $G_n$ are eventually permuted and are indistinguishable from the earlier vertices in $\PartB(G_n)$. This informally tells why the change-point cannot be detected. More formally, the Theorem~\ref{thm:main} is an immediate consequence of the two following propositions, choosing $\Delta_n' \asymp n^{2/3}$ (implying $\tau_n' \sim n$), $\Delta_n = O\big(\frac{n^{1/3}}{\alpha_n}\big)$, and $\alpha_n\to \infty$ arbitrarily slowly if $\delta_0>0$ or $\frac{\alpha_n}{\log(n)} \to \infty$ arbitrarily slowly if $\delta_0 = 0$.

\begin{proposition}
  \label{prop:contiguity_intermediate_regime}
  There exist constants $c_1,c_2 > 0$ depending only on $\delta_0$, $\delta_1$, and $m$, such that for all $n\geq 4$, if $3\leq \tau_n' < \tau_n$, $\Delta_n' \leq \tau_n$ , $\frac{\alpha_n\Delta_n'}{\tau_n'} \leq \frac{1}{2}$ and $\frac{\Delta_n}{\Delta_n'} \leq \frac{1}{4}$ then
  \begin{equation*}
    \log\EE_0^n\Bigg[\Bigg(\frac{\intd Q_1^{n,p}}{\intd Q_0^{n,p}}(\pi_n(G_n)) \Bigg)^2\1_{B_n} \Bigg]%
    \leq \frac{4\alpha_n\Delta_n\Delta_n'}{\tau_n'}%
    + \frac{22m \Delta_n^2}{\Delta_n'}%
    + \frac{2}{3\Delta_n'}%
    + \sqrt{\frac{c_1 \Delta_n^2}{\Delta_n'}}e^{\frac{c_2 \Delta_n^2}{\Delta_n'}}.
  \end{equation*}

\end{proposition}

See Appendix~\ref{sec:proof-prop-second-moment} for the proof of Proposition~\ref{prop:contiguity_intermediate_regime}. 

  \begin{proposition}
    \label{pro:eventBn}
    There exists a constant $C > 0$ depending only on $\delta_0$, $\delta_1$, and $m$, such that for all $2 \leq \tau_n' \leq n$
    \begin{equation*}
      \PP_1^n(B_n^c)%
      \leq \frac{C}{\alpha_n}\Big(1 + \frac{\alpha_n\Delta_n\Delta_n'}{\tau_n'} \Big)%
      \cdot%
      \begin{cases}
        \log(\tau_n') &\mathrm{if}\ \delta_0 = 0,\\
        1 &\mathrm{if}\ \delta_0 > 0.
      \end{cases}
    \end{equation*}
  \end{proposition}

See Appendix~\ref{sec:proof-prop-eventBn} for the proof of Proposition~\ref{pro:eventBn}. 

We believe the Theorem~\ref{thm:main} can not be improved beyond $\Delta_n = o(n^{1/3})$ using the permutation $\pi_n$ we designed. In fact, we think that the upper bound in Proposition~\ref{prop:contiguity_intermediate_regime} is essentially tight and could only be made bounded by choosing a smaller event $B_n$ if $\Delta_n$ is made larger than $o(n^{1/3})$. However, using a smaller event $B_n$ makes unlikely that a result such as Proposition~\ref{pro:eventBn} hold.

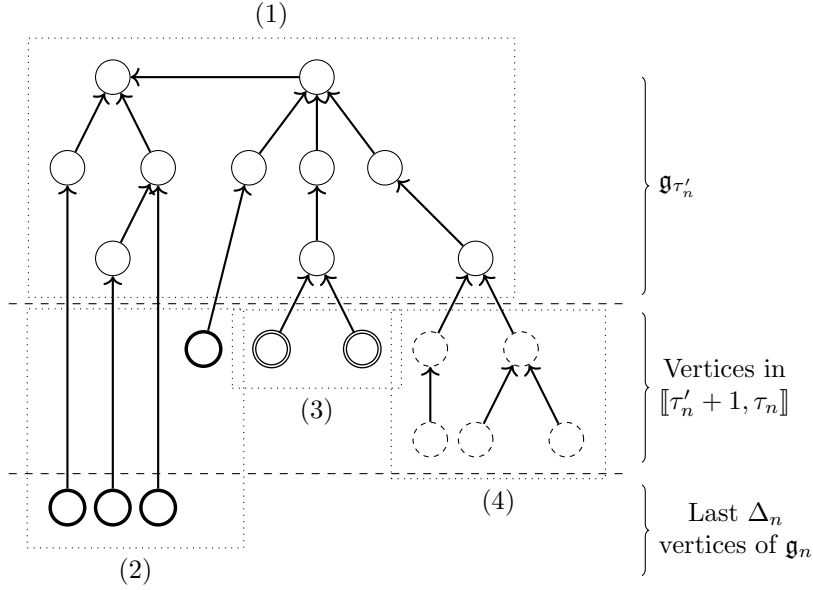
\begin{figure}[!htbp]
    \centering
\begin{tikzpicture}[scale=0.6]

\tikzstyle{boldvertex}=[draw, circle, minimum size=1.3em, inner sep=0pt, line width=1.3pt]
\tikzstyle{dottedvertex}=[draw, circle, minimum size=1.3em, inner sep=0pt, dash pattern=on 2pt off 2pt]
\tikzstyle{doublevertex}=[draw, circle, minimum size=1.3em, inner sep=0pt, double, line width=0.5pt]
\tikzstyle{normalvertex}=[draw, circle, minimum size=1.3em, inner sep=0pt]

    \node[normalvertex] (A) at (-2,3) {};
    \node[normalvertex] (B) at (2.5,3) {};
    \node[normalvertex] (C) at (-3,1) {};
    \node[normalvertex] (D) at (-1,1) {};
    \node[normalvertex] (E) at (1,1) {};
    \node[normalvertex] (F) at (2.5,1) {};
    \node[normalvertex] (G) at (4,1) {};
    \node[normalvertex] (H) at (-2,-1) {};
    \node[normalvertex] (I) at (2.5,-1) {};
    \node[normalvertex] (J) at (6,-1) {};
    \node[boldvertex] (K) at (-3,-6.5) {};
    \node[boldvertex] (L) at (-2,-6.5) {};
    \node[boldvertex] (M) at (-1,-6.5) {};
    \node[boldvertex] (N) at (0,-3) {};
    \node[doublevertex] (O) at (1.5,-3) {};
    \node[doublevertex] (P) at (3.5,-3) {};
    \node[dottedvertex] (Q) at (5,-3) {};
    \node[dottedvertex] (R) at (7,-3) {};
    \node[dottedvertex] (S) at (5,-5) {};
    \node[dottedvertex] (T) at (6,-5) {};
    \node[dottedvertex] (U) at (8,-5) {};

    \node[draw, dotted, fit=(A)(B)(C)(D)(E)(F)(G)(H)(I)(J), inner sep=8pt, label={above:(1)}] {};
    \node[draw, dotted, fit=(K)(L)(M)(N), inner sep=8pt, label={below:(2)}] {};
    \node[draw, dotted, fit=(O)(P), inner sep=8pt, label={below:(3)}] {};
    \node[draw, dotted, fit=(Q)(R)(S)(U), inner sep=8pt, label={below:(4)}] {};
    
    \draw[edge] (B) to (A);
    \draw[edge] (C) to (A);
    \draw[edge] (D) to (A);
    \draw[edge] (E) to (B);
    \draw[edge] (F) to (B);
    \draw[edge] (G) to (B);
    \draw[edge] (H) to (D);
    \draw[edge] (I) to (F);
    \draw[edge] (J) to (G);
    \draw[edge] (R) to (J);
    \draw[edge] (Q) to (J);
    \draw[edge] (P) to (I);
    \draw[edge] (O) to (I);
    \draw[edge] (K) to (C);
    \draw[edge] (L) to (H);
    \draw[edge] (M) to (D);
    \draw[edge] (N) to (E);
    \draw[edge] (S) to (Q);
    \draw[edge] (T) to (R);
    \draw[edge] (U) to (R);
    \node at (-4.5,-2) (left) {};
    \node at (9.5,-2) (right) {};
    \draw (left) -- (right) [dashed];
    \node at (-4.5,-5.75) (left) {};
    \node at (9.5,-5.75) (right) {};
    \draw (left) -- (right) [dashed];
    \draw[decorate,decoration={brace,raise=0.1cm}]
(9.5,3) -- (9.5,-1.8) node[midway, right, xshift=.2cm] {$\g_{\tau_n'}$};
    \draw[decorate,decoration={brace,raise=0.1cm}]
(9.5,-2.2) -- (9.5,-5.5) node[midway,right,xshift=.2cm, align = center] {Vertices in\\ $\llbracket \tau_n'+1,\tau_n\rrbracket$};
    \draw[decorate,decoration={brace,raise=0.1cm}]
(9.5,-6) -- (9.5,-8) node[midway,right,xshift=.2cm, align = center] {Last $\Delta_n$\\ vertices of $\g_n$};
\end{tikzpicture}
\caption{Typical preferential attachment graph $\g_n$ with $m = 1$ when $\Delta_n = o(n^{1/3})$. Four types of vertices emerge: normal vertices (1), bold vertices (2), double circle vertices (3) and dotted vertices (4). Our random permutation $\pi_n$ is built to permute only vertices represented in bold.}
    \label{fig:PA:example_1}
\end{figure}

\subsection{The observation is the labeled graph}\label{sec:labeled}

We consider now the model where the observation is the labeled graph $G_n$. The main purpose of this section is to emphasize the difference between the labeled and the unlabled model, by showing that in the labeled model the change-point can be detected as soon as $\Delta_n \to \infty$; in contrast with the unlabeled model for which $\frac{\Delta_n}{n^{1/2}}\to \infty$ is sufficient by \cite{BBCH23} and $\frac{\Delta_n}{n^{1/3}} \to \infty$ is necessary by our previous result. This also shows that a reduction scheme to a problem where the labeled graph undergoes a transformation is unavoidable to obtain a non trivial lower bound in the unlabeled model.

We assume that the model parameters $(\delta_0, \delta_1)$ and $\tau_n$ are known to be consistent with our Theorem~\ref{thm:main}. We however state additional results in the supplemental~\cite{kng:supplement}, covering the case where $(\delta_0,\delta_1)$ are unknown as well as the localization of the change-point (\textit{ie}. estimating $\tau_n$). In particular these additional results show that not knowing the parameters does not affect the capability of detecting the change-point as soon as $\Delta_n \to \infty$.

\begin{theorem}
  \label{thm:labeled:change_detection}
    Let $Q^{n}_{0} = \PP_0^n(G_n \in \cdot)$ and $Q^{n}_{1} = \PP_1^n(G_n \in \cdot)$. If $\tau_n \to \infty$ and $\Delta_n \to \infty$, detection of the change is possible: the likelihood-ratio test $T_n = \1\big(\frac{\intd Q_1^n}{\intd Q_0^n}(G_n) > 1\big)$ satisfies
    \begin{equation*}
        \EE_0^n(T_n) + \EE_1^n(1-T_n) \to 0.
    \end{equation*}
    When $\limsup_{n\to\infty}\Delta_n < +\infty$ detection of the change is not possible: $(Q_1^n)_{n\geq 1}$ is contiguous with respect to $(Q_0^n)_{n\geq 1}$.
\end{theorem} 

See the supplementary material \cite{kng:supplement} for the proof of Theorem~\ref{thm:labeled:change_detection}. Observe that Theorem~\ref{thm:labeled:change_detection} identifies the exact phase transition for detection when the labeled graph is observed and the model parameters are known.

\section{Discussions and perspectives}\label{Section_4}

While the original conjecture in \cite{BBCH23} had two parts, one concerning the impossibility of detection using the sequence of degrees and the other concerning the impossibility of detection using the unlabeled graph, our work focuses only on the second part, which is more general as it implies the impossibility of detection using the degrees. Although we believe the conjecture to be true, our proof of Theorem \ref{thm:main} does not cover all the regimes of the conjecture in terms of $\Delta_n$ and $(\delta_0,\delta_1)$. As explained in Section~\ref{sec:sketch}, the main step of our proof of Theorem~\ref{thm:main} resides in showing that the second moment of the likelihood-ratio of the permuted graph is bounded by an absolute constant. We were able to exhibit such a bound only in the regime where $\Delta_n = o(n^{1/3})$ and $\delta_0 \geq 0$. To put it simply, our proof works when all the last $\Delta_n$ vertices are in $\PartB(G_n)$ (\textit{ie}. bold in the Figure~\ref{fig:PA:example_1}): the expression of the likelihood-ratio is easier to handle in this case and its second moment can be bounded by an absolute constant. This situation is illustrated in Figure \ref{fig:PA:example_1} by a typical example. However, in the regime $n^{1/3} \lesssim \Delta_n \lesssim n^{1/2}$ and as illustrated in Figure \ref{fig:PA:example_2}, ``double circle'' and ``dotted vertices'' start appearing amongst the last $\Delta_n$ vertices, making it more difficult to choose an appropriate permutation. If we keep the same permutation (the one modifying only the labels of bold vertices) in the regime $n^{1/3} \lesssim \Delta_n \lesssim n^{1/2}$, the labels of the ``dotted'' and ``double circle'' vertices appearing amongst the last $\Delta_n$ vertices will be kept invariant and the second moment of the likelihood-ratio will diverge to infinity. One possible way of generalizing the result to the remaining regime is to construct a permutation that modifies the labels of almost $O(n)$ vertices, including all the last $\Delta_n$ vertices, while at the same time still be able to uniformly bound the second moment of the likelihood-ratio. There is a trade-off between the complexity of the chosen permutation (how many labels are modified and how they are modified) and the ease in bounding the second moment of the likelihood-ratio. For a similar reason, the regime $\delta_0 < 0$ was not covered in the proof. The main shortcoming of our proof is that we choose permutations that modify only labels in $\PartB(G_n)$. The reason behind this choice is that given a preferential attachment random graph, every permutation affecting only $\PartB(G_n)$ results in a labeled graph having positive probability under preferential attachment (Lemma \ref{lem:perm:invariant-support}). This facilitates the explicit writing of the likelihood-ratio. However, if we were to allow the permutations to modify the labels of ``dotted'' and ``double circle'' vertices, then one needs to be much more careful to ensure that after the application of the permutation, the labeled graph still has positive probability under preferential attachment; or find another way to circumvent the issues discussed in Section~\ref{sec:diff-solv-probl}.

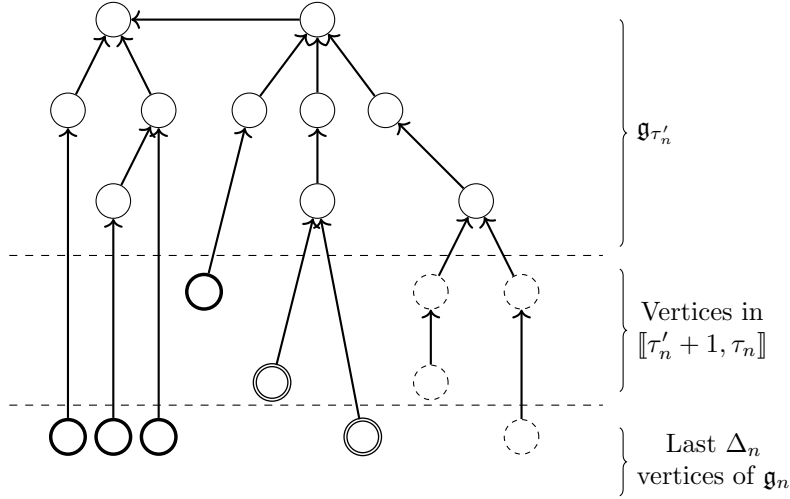
\begin{figure}[!htbp]
    \centering
\begin{tikzpicture}[scale=0.6]
    \node[normalvertex] (A) at (-2,3) {};
    \node[normalvertex] (B) at (2.5,3) {};
    \node[normalvertex] (C) at (-3,1) {};
    \node[normalvertex] (D) at (-1,1) {};
    \node[normalvertex] (E) at (1,1) {};
    \node[normalvertex] (F) at (2.5,1) {};
    \node[normalvertex] (G) at (4,1) {};
    \node[normalvertex] (H) at (-2,-1) {};
    \node[normalvertex] (I) at (2.5,-1) {};
    \node[normalvertex] (J) at (6,-1) {};
    \node[boldvertex] (K) at (-3,-6.2) {};
    \node[boldvertex] (L) at (-2,-6.2) {};
    \node[boldvertex] (M) at (-1,-6.2) {};
    \node[boldvertex] (N) at (0,-3) {};
    \node[doublevertex] (O) at (1.5,-5) {};
    \node[doublevertex] (P) at (3.5,-6.2) {};
    \node[dottedvertex] (Q) at (5,-3) {};
    \node[dottedvertex] (R) at (7,-3) {};
    \node[dottedvertex] (S) at (5,-5) {};
    \node[dottedvertex] (U) at (7,-6.2) {};
    \draw[edge] (B) to (A);
    \draw[edge] (C) to (A);
    \draw[edge] (D) to (A);
    \draw[edge] (E) to (B);
    \draw[edge] (F) to (B);
    \draw[edge] (G) to (B);
    \draw[edge] (H) to (D);
    \draw[edge] (I) to (F);
    \draw[edge] (J) to (G);
    \draw[edge] (R) to (J);
    \draw[edge] (Q) to (J);
    \draw[edge] (P) to (I);
    \draw[edge] (O) to (I);
    \draw[edge] (K) to (C);
    \draw[edge] (L) to (H);
    \draw[edge] (M) to (D);
    \draw[edge] (N) to (E);
    \draw[edge] (S) to (Q);
    \draw[edge] (U) to (R);
    \node at (-4.5,-2.2) (left) {};
    \node at (9,-2.2) (right) {};
    \draw (left) -- (right) [dashed];
    \node at (-4.5,-5.5) (left) {};
    \node at (9,-5.5) (right) {};
    \draw (left) -- (right) [dashed];
    \draw[decorate,decoration={brace,raise=0.1cm}]
(9,3) -- (9,-2) node [midway,right,xshift=.2cm]{$\g_{\tau_n'}$};
    \draw[decorate,decoration={brace,raise=0.1cm}]
(9,-2.5) -- (9,-5.2) node[midway,right,xshift=.2cm, align = center] {Vertices in\\ $\llbracket \tau_n'+1,\tau_n\rrbracket$};
\draw[decorate,decoration={brace,raise=0.1cm}]
(9,-6) -- (9,-7.5) node[midway,right,xshift=.2cm, align = center] {Last $\Delta_n$\\ vertices of $\g_n$};
\end{tikzpicture}
\caption{Typical preferential attachment graph $\g_n$with $m = 1$ when $n^{1/3} \lesssim \Delta_n \lesssim n^{1/2}$.}
    \label{fig:PA:example_2}
\end{figure}

%
\section{Proof elements common to both labeled and unlabeled graphs}
\label{sec:likel-likel-ratio}

\subsection{A result on the support of the general preferential attachment model}
\label{sec:result-supp-gener}

Anticipating that we will need to compute the likelihood under both the null hypothesis and the alternative, we first derive the likelihood in the most general case of a Preferential Attachment Model (PAM) with an arbitrary parameter function $\delta_n : \mathbb{N} \to (-m,+\infty)$, which is allowed to change with $n$. We let $\FF_{\delta_n}^n$ denote the distribution of a partial sequence of random graph $(G_0,G_1,\dots,G_n)$ distributed according to the PAM with parameter $\delta_n$.%

\begin{lemma}
    \label{lem:support}
    For $n\geq 0$, let
    \begin{equation*}
      \mathfrak{S}_n%
      = \Big\{ \g_n\;:\; \VS(\g_n) = \llbracket 0,n\rrbracket,\ \Chi_{\g_n}(0) = \varnothing,\  \forall v \in  \llbracket 1,n \rrbracket\ \Chi_{\g_n}(v) \subset \llbracket 0, v-1\rrbracket\ \mathrm{and}\  \Dout_{\g_n}(v) = m \Big\}.
    \end{equation*}
    Then $\FF_{\delta_n}^n(G_n = \g_n) > 0 \iff \g_n \in \mathfrak{S}_n$. Furthermore, for any $\g_n \in \mathfrak{S}_n$,
    \begin{equation*}
      \FF_{\delta_n}^n(G_n = \g_n)%
      = C(\g_n)%
        \frac{\prod_{j=2}^n\prod_{w\in \Chi_{\g_n}(j)}\prod_{k=1}^{\mu_{\g_n}(j,w)}\big( \D_{\g_n[\llbracket 0,j-1\rrbracket]}(w) + k-1 + \delta_n(j) \big)}{\prod_{j=2}^n\prod_{i=1}^mS_{j,i-1}(\delta_n(j))}
      \end{equation*}
      where $C(\g_n) = \frac{(m!)^{n-1}}{\prod_{j=2}^n\prod_{w\in \Chi_{\g_n}(j)}\mu_{\g_n}(j,w)!}$ and $S_{j,i-1}(\delta) = 2m(j-1) + i-1 + \delta j$ (defined as in \cite{GV17}).
  \end{lemma}
\begin{proof}
    Suppose $n\geq 2$, otherwise the result is trivial. By construction $\FF_{\delta_n}^n((G_0,G_1) = (\g_0,\g_1)) =1$ iff $\g_0$ is the labeled graph with a unique vertex with label zero and no edge, and $\g_1$ is the graph with two vertices zero and one with $m$ edges going from one to zero. Let $1\leq j \leq n$ and suppose that
    \begin{equation}
      \label{eq:inducsupport}
      \forall k\in \llbracket 0,j\rrbracket,\ %
      \g_k \in \mathfrak{S}_k\ \mathrm{and}\ \g_k[ \llbracket 0,k-1\rrbracket] =  \g_{k-1}%
      \iff \FF_{\delta_n}^n\big((G_0,\dots,G_j) = (\g_0,\dots,\g_j)\big) > 0,
    \end{equation}
    which has been shown to be verified for $j=1$. The graph $G_j$ is obtained from $G_{j-1}$ by sampling $m$ edges according to the PA rule. In other word
    \begin{equation*}
      \FF_{\delta_n}^n((G_0,\dots,G_j) = (\g_0,\dots,\g_j)) = \FF_{\delta_n}^n((G_0,\dots,G_{j-1}) = (\g_0,\dots,\g_{j-1}))K_{\delta_n,j}(\g_j\mid\g_{j-1}) 
    \end{equation*}
    for a Markov kernel $K_{\delta_n,j}(\g_j \mid \g_{j-1} )$ that assigns non-zero probability to $\g_j$ iff $\VS(\g_j) = \llbracket 0,j\rrbracket$ and $\g_j[\llbracket 0,j-1\rrbracket] = \g_{j-1}$ and $\Dout_{\g_j}(j) = m$ and $\Chi_{g_j}(j) \subset \llbracket 0,j-1\rrbracket$. By induction \eqref{eq:inducsupport} is then verified for all $1\leq j \leq n$. Observe that \eqref{eq:inducsupport} implies that the law of $(G_0,\dots,G_n)$ is entirely determined by $G_n$ since it must be that $G_k = G_n[\llbracket 0,k\rrbracket]$ $\FF_{\delta_n}^n$-almost-surely for all $k \in \llbracket 0,n\rrbracket$.

    Next, let $(\g_{j-1},\g_{j})\in \mathfrak{S}_{j-1}\times \mathfrak{S}_j$ with $\g_{j-1} = \g_j[\llbracket 0,j-1\rrbracket]$. A rapid computation using equation~\eqref{PA_model} shows that if we enumerate $v_1 < \dots < v_{\ell}$ the elements of $\Chi_{\g_j}(j)$ and denote by $\mu_1,\dots,\mu_{\ell}$ the associated edge multiplicities:
    \begin{align*}
      K_{\delta_n,j}(\g_j\mid\g_{j-1}) %
      &= \sum_{(e_1,\dots,e_m)}\prod_{i=1}^{m}\frac{\D_{\g_{j-1}}(v_{e_i}) + \sum_{1\leq k < i}\1_{e_k = v_{e_i}} + \delta_n(j) }{\sum_{w=0}^{j-1}\big(\D_{\g_{j-1}}(w) + \sum_{1\leq k < i}\1_{e_k = w} + \delta_n(j) \big) }\\
      &= \sum_{(e_1,\dots,e_m)}\frac{\prod_{w\in \Chi_{\g_j}(j)}\prod_{k=1}^{\mu_{\g_j}(j,w)}\big( \D_{\g_{j-1}}(w) + k-1 + \delta_n(j) \big)}{\prod_{i=1}^m\sum_{w=0}^{j-1}\big(\D_{\g_{j-1}}(w) + \sum_{1\leq k < i}\1_{e_k = w} + \delta_n(j) \big) }
    \end{align*}
    where the summation over $(e_1,\dots,e_m)$ is understood over sequences in $\llbracket 1,\ell \rrbracket^m$ with $\mu_k$ elements equal to $k$ for each $k=1,\dots,\ell$ (\textit{ie}. over all the possible ways of assigning the $m$ edges to the $\ell$ children with the multiplicity constraint taken into account). We observe that exactly $m$ edges are added at each step of the construction, so
    \begin{equation*}
      \sum_{w=0}^{j-1}\Big(\D_{\g_{j-1}}(w) + \sum_{1\leq k < i}\1_{e_k = w} + \delta_n(j)\Big)%
      =  2m(j-1) + i-1 + j\delta_n(j)%
      = S_{j,i-1}(\delta_n(j)).
    \end{equation*}
    It follows that
    \begin{equation*}
      K_{\delta_n,j}(\g_j \mid \g_{j-1})%
      = \frac{m!}{\prod_{w\in \Chi_{\g_j}(j)}\mu_{\g_j}(j,w)!}%
      \frac{\prod_{w\in \Chi_{\g_j}(j)}\prod_{k=1}^{\mu_{\g_j}(j,w)}\big( \D_{\g_{j-1}}(w) + k-1 + \delta_n(j) \big)}{\prod_{i=1}^mS_{j,i-1}(\delta_n(j)) }.
    \end{equation*}
    Consequently, for all $\g_n \in \mathfrak{S}_n$, writing abusively $\g_j = \g_n[\llbracket 0,j\rrbracket]$ (which is justified by the above discussion),
    \begin{align*}
      \FF_{\delta_n}^n(G_n = \g_n)%
      &= \prod_{j=2}^n\frac{m!}{\prod_{w\in \Chi_{\g_j}(j)}\mu_{\g_j}(j,w)!}%
        \frac{\prod_{w\in \Chi_{\g_j}(j)}\prod_{k=1}^{\mu_{\g_j}(j,w)}\big( \D_{\g_{j-1}}(w) + k-1 + \delta_n(j) \big)}{\prod_{i=1}^mS_{j,i-1}(\delta_n(j)) }\\
      &= C(\g_n)%
        \frac{\prod_{j=2}^n\prod_{w\in \Chi_{\g_n}(j)}\prod_{k=1}^{\mu_{\g_n}(j,w)}\big( \D_{\g_{j-1}}(w) + k-1 + \delta_n(j) \big)}{\prod_{j=2}^n\prod_{i=1}^mS_{j,i-1}(\delta_n(j))}.
    \end{align*}
    This concludes the proof.
  \end{proof}

\subsection{The likelihood of a labeled graph under the null and the alternative hypotheses}
\label{sec:likel-label-graph}

In this section we compute the likelihood of the labeled graph under the null hypothesis (Lemma~\ref{lem:lkl:null}), under the alternative hypothesis (Lemma~\ref{lem:lkl:alt}), as well as the likelihood-ratio (Lemma~\ref{lem:lkl:ratio}). %

\begin{lemma}
  \label{lem:lkl:null}
  Let $\mathfrak{S}_n$, $C(\g_n)$, and $S_{t,i-1}(\delta)$ as defined in Lemma~\ref{lem:support}. Then for all $\g_n \in \mathfrak{S}_n$
  \begin{equation*}
    \PP_0^n\left(G_n = \g_n\right)
    = 
      C(\g_n)\frac{\prod_{v=0}^{n-1}\prod_{k=0}^{\Din_{\g_n}(v)-1}(m+\delta_0 + k)}{\prod_{t=2}^{n}\prod_{i=1}^{m}S_{t,i-1}(\delta_0)}
    = C(\g_n)\frac{\prod_{k=m}^{nm}(k+\delta_0)^{N_{>k}(\g_n)}}{\prod_{t=2}^{n}\prod_{i=1}^{m}S_{t,i-1}(\delta_0)}
  \end{equation*}
  where $N_{>k}(\g_n)$ is the number of vertices in $\g_n$ which have degree strictly greater than $k$.%
\end{lemma}
\begin{proof}
    The first expression comes from swapping the product over parents and children in the expression given in Lemma~\ref{lem:support} and using that the parameter is constant over time:
    \begin{align*}
      \PP_0^n(G_n = \g_n)%
      &= C(\g_n)\frac{\prod_{j=2}^n\prod_{w\in \Chi_{\g_n}(j)}\prod_{k=1}^{\mu_{\g_n}(j,w)}\big( \D_{\g_n[\llbracket 0,j-1\rrbracket]}(w) + k-1 + \delta_0 \big)}{\prod_{j=2}^n\prod_{i=1}^mS_{j,i-1}(\delta_0)}\\
      &= C(\g_n)\frac{\prod_{t=0}^{n-1}\prod_{s\in \Par_{\g_n}(t)}\prod_{k=1}^{\mu_{\g_n}(s,t)}\big( \D_{\g_n[\llbracket 0,s-1\rrbracket]}(t) + k-1 + \delta_0 \big)}{\prod_{j=2}^n\prod_{i=1}^mS_{j,i-1}(\delta_n(s))}.
    \end{align*}
    Now for each vertex $t$ contributing to the above product, order its parents in increasing time of arrivals and see that the product over $s$ and $k$ is in fact equal to $(m+\delta_0)(m+1+\delta_0)\dots(m+ \Din_{\g_n}(t) - 1 + \delta_0)$. Thus,
    \begin{align*}
      \PP_0^n(G_n = \g_n)%
      &= C(\g_n)\frac{\prod_{v=0}^{n-1}\prod_{k=0}^{\Din_{\g_n}(v)-1}(m+\delta_0 + k)}{\prod_{t=2}^{n}\prod_{i=1}^{m}S_{t,i-1}(\delta_0)}
    \end{align*}
    which is the first expression in the statement of the Lemma. For the second expression, notice that
  \begin{align*}
    \prod_{v=0}^{n-1}\prod_{k=0}^{\Din_{\g_n}(v)-1}(m+\delta_0 + k)
    &= \prod_{v=0}^{n-1}\prod_{k=0}^{(n-1)m}(m+\delta_0+k)\1_{k\leq \Din_{g_n}(v)-1}\\
    &= \prod_{k=0}^{(n-1)m}(m+\delta_0+k)^{\sum_{v=0}^{n-1}\1_{\Din_{\g_n}(v)>k}}\\
    &= \prod_{k=m}^{nm}(k+\delta_0)^{\sum_{v=0}^{n-1}\1_{\Din_{\g_n}(v) + m > k}}\\
    &= \prod_{k=m}^{nm}(k+\delta_0)^{\sum_{v=0}^{n-1}\1_{\D_{\g_n}(v) > k}}.
  \end{align*}
  Hence the result.
\end{proof}

Note that under the null hypothesis, the likelihood of the graph does not depend on the labels of the vertices. It depends only on the structure $s(\g_{n})$ since $N_{>k}(\cdot)$ is constant over $s(\g_n)$.

\begin{lemma}
  \label{lem:lkl:alt}
  Let $\mathfrak{S}_n$, $C(\g_n)$, and $S_{t,i-1}(\delta)$ as defined in Lemma~\ref{lem:support}. Also define  $H_{\g_n}^{\leq\tau_n}(v) = \sum_{u\in \Par_{\g_n}(v)} \mu_{\g_n}(u, v)\1_{u\leq \tau_n}$ and $H_{\g_n}^{>\tau_n}(v) = \sum_{u\in \Par_{\g_n}(v)} \mu_{\g_n}(u, v)\1_{u> \tau_n}$. Then for all $\g_n \in \mathfrak{S}_n$
  \begin{align*}
    \PP_1^n(G_n = \g_n)
    &= C(\g_n) \frac{\prod_{v=0}^{n-1}\big[\prod_{k=0}^{H_{\g_n}^{\leq \tau_n}(v)-1}(m+\delta_0 + k) \prod_{k = H_{\g_n}^{\leq \tau_n}(v))}^{H_{\g_n}^{\leq \tau_n}(v))+ H_{\g_n}^{> \tau_n}(v) -1}(m+\delta_1 + k) \big]}{\prod_{t=2}^{n}\prod_{i=1}^{m}S_{t,i-1}(\delta(t))}\\
    &=C(\g_n) \frac{\prod_{k=m}^{nm}\left(k+\delta_0\right)^{N_{>k}(\g_{\tau_n})}}{\prod_{t=2}^{\tau_n}\prod_{i=1}^{m}S_{t,i-1}(\delta_0)}\frac{\prod_{k=m}^{nm}(k+\delta_1)^{N_{>k}(\g_n)-N_{>k}(\g_{\tau_n})}}{\prod_{t=\tau_n +1}^{n}\prod_{i=1}^{m}S_{t,i-1}(\delta_1)}
  \end{align*}
  with $\g_{\tau_n} = \g_n[\llbracket 0,\tau_n\rrbracket]$ and $\delta(t) = \delta_0\1_{t\leq \tau_n} + \delta_1\1_{t > \tau_n}$.
\end{lemma}
\begin{proof}
  The first expression comes from the Lemma~\ref{lem:support} and using the same arguments as in Lemma~\ref{lem:lkl:null}. For second expression, notice that
  \begin{align*}
    \prod_{v=0}^{n-1}\left[\prod_{k=0}^{H^{\leq \tau_n}(v)-1}(m+\delta_0 + k) \prod_{k = H^{\leq \tau_n}(v)}^{H(v) -1}(m+\delta_1 + k) \right]
    &=
      \prod_{v=0}^{n-1}\prod_{k=0}^{H^{\leq \tau_n}(v)-1}(m+\delta_0 + k) \frac{\prod_{k = 0}^{\Dout_{\g_n}(v)-1}(m+\delta_1 + k)}{\prod_{k=0}^{H^{\leq \tau_n}(v)-1}(m+\delta_1 + k)}\\
    &= \prod_{v=0}^{n-1}\left(\prod_{k=0}^{H^{\leq \tau_n}(v)-1}\frac{m + \delta_0 + k}{m + \delta_1 + k}\prod_{k = 0}^{\Dout_{\g_n}(v)-1}(m+\delta_1 + k)\right)\\
    &= \prod_{k=m}^{nm}\left(\frac{k+\delta_0}{k+\delta_1}\right)^{N_{>k}(\g_{\tau_n})}(k+\delta_1)^{N_{>k}(\g_n)}\\
    &= \prod_{k=m}^{nm}\left(k+\delta_0\right)^{N_{>k}(\g_{\tau_n})}(k+\delta_1)^{N_{>k}(\g_n)-N_{>k}(\g_{\tau_n})}
  \end{align*}
  which concludes the proof.
\end{proof}

\begin{lemma}
  \label{lem:lkl:ratio}
  Let $Q_{\ell}^n = \PP_{\ell}^n(G_n \in \cdot)$ for $\ell=0,1$.  Let $S_{t,i-1}(\delta)$ as defined in Lemma~\ref{lem:support}. Then, for every $\g_n \in \mathfrak{S}_n$
  \begin{equation*}
    \frac{\intd Q_1^n}{\intd Q_0^n}(\g_n)%
    = \prod_{t = \tau_n + 1 }^{n}\prod_{i = 1}^{m}\frac{S_{t,i-1}(\delta_0)}{S_{t,i-1}(\delta_1)}\prod_{k=m}^{nm}\left(\frac{k+\delta_1}{k+\delta_0}\right)^{N_{>k}(\g_n) - N_{>k}(\g_{n}[\llbracket 0,\tau_n\rrbracket])}.
  \end{equation*}
  Furthermore, almost-surely under $\PP_0^n$
  \begin{equation*}
    \frac{\intd Q_1^n}{\intd Q_0^n}(G_n)%
    =
    \prod_{t = \tau_n + 1 }^{n}\prod_{i = 1}^{m}\frac{S_{t,i-1}(\delta_0)}{S_{t,i-1}(\delta_1)}\left(\frac{\D_{G_{t,i-1}}(V_{t,i}) + \delta_1}{\D_{G_{t,i-1}}(V_{t,i}) + \delta_0}\right).
  \end{equation*}
  where $V_{t,i}$ is defined as in Section \ref{sec:PA_mechanism}.
\end{lemma}%
\begin{proof}
  The first expression in the statement of the lemma is an immediate consequence of Lemmas~\ref{lem:lkl:null} and~\ref{lem:lkl:alt}. Regarding the second statement, it suffices to observe that $N_{>k}(G_n)$ depends only on $s(G_n)$, so that [recall $G_t = G_{t,m}$]
  \begin{equation*}
    N_{>k}(G_n) - N_{>k}(G_{\tau_n})%
    = \sum_{t=\tau_n + 1}^{n}\sum_{i=1}^{m}\1_{\D_{G_{t,i-1}}(V_{t,i}) = k}.
  \end{equation*}
  It follows that
\begin{equation*}
  \begin{split}
    \prod_{k=m}^{nm}\left(\frac{k+\delta_1}{k+\delta_0}\right)^{N_{>k}(G_n) - N_{>k}(G_{\tau_n})}%
      &= \prod_{t = \tau_n + 1}^{n}\prod_{i=1}^{m}\prod_{k=m}^{nm}\left(\frac{k+\delta_1}{k+\delta_0}\right)^{\1_{\D_{G_{t,i-1}}(V_{t,i}) = k}}\\
        &= \prod_{t = \tau_n + 1}^{n}\prod_{i=1}^{m}\prod_{k=m}^{nm}\left(\frac{\D_{G_{t,i-1}}(V_{t,i}) + \delta_1}{\D_{G_{t,i-1}}(V_{t,i}) + \delta_0}\right)^{\1_{\D_{G_{t,i-1}}(V_{t,i}) = k}}\\
        &= \prod_{t = \tau_n + 1}^{n}\prod_{i=1}^{m}\left(\frac{\D_{G_{t,i-1}}(V_{t,i}) + \delta_1}{\D_{G_{t,i-1}}(V_{t,i}) + \delta_0}\right).
    \end{split}
  \end{equation*}
  This concludes the proof.
\end{proof}

  The following lemma will also be used several times when analyzing likelihood ratios.

  \begin{lemma}
    \label{lem:boundSti}
    Suppose $\tau_n \geq 3$.  Let $S_{t,i-1}(\delta)$ as defined in Lemma~\ref{lem:support}. Then for every $\delta_0,\delta_1 > -m$
    \begin{equation*}
      e^{-\frac{6m\Delta_n}{\tau_n}} \Big(\frac{2m+\delta_0}{2m+\delta_1}\Big)^{m\Delta_n}\leq%
      \prod_{t=\tau_n+1}^n\prod_{i=1}^m\frac{S_{t,i-1}(\delta_0)}{S_{t,i-1}(\delta_1)}%
      \leq e^{\frac{6m\Delta_n}{\tau_n}}\Big(\frac{2m+\delta_0}{2m+\delta_1}\Big)^{m\Delta_n}
    \end{equation*}
  \end{lemma}
  \begin{proof}
    By definition of $S_{t,i-1}$
    \begin{equation*}
      \prod_{t=\tau_n+1}^n\prod_{i=1}^m\frac{S_{t,i-1}(\delta_0)}{S_{t,i-1}(\delta_1)}%
      = \Big(\frac{2m+\delta_0}{2m+\delta_1}\Big)^{m\Delta_n}%
      \prod_{t=\tau_n+1}^n\prod_{i=1}^m \frac{1 + \frac{-2m+i-1}{t(2m+\delta_0)}}{1 + \frac{-2m+i-1}{t(2m+\delta_1)}}
    \end{equation*}
    But for $j=0,1$, $\tau_n+1\leq t\leq n$,~ $1\leq i \leq m$ and $\delta_j > -m$
    \begin{equation*}
      1 - \frac{2}{\tau_n}%
      \leq 1 + \frac{-2m+i-1}{t(2m+\delta_j)} \leq 1.
    \end{equation*}
    Thus,
    \begin{equation*}
      \Big(1 - \frac{2}{\tau_n}\Big)^{m\Delta_n}%
      \leq%
      \prod_{t=\tau_n+1}^n\prod_{i=1}^m \frac{1 + \frac{-2m+i-1}{t(2m+\delta_0)}}{1 + \frac{-2m+i-1}{t(2m+\delta_1)}}
      \leq \Big(\frac{1}{1-2/\tau_n}\Big)^{m\Delta_n}.
    \end{equation*}
    The conclusion follows because $\log(1-2/\tau_n) \geq -\frac{2}{\tau_n-2} \geq -\frac{6}{\tau_n}$ when $\tau_n \geq 3$.
  \end{proof}

\section{Proofs when the observation is the unlabeled graph}
\label{Section_5}

\subsection{Proof of Lemma \ref{lem:increasing_difficulty}}
\label{sec:proof-lemma-increasing-difficulty}

$1\implies 2$. Let $(A_n)_{n\geq 1}$ a sequence of $\pi_n(G_n)$-measurable sets such that $\PP_0^n(A_n) \to 0$ and let $\varepsilon > 0$ arbitrary. Because $\PP_0^n(A_n) = \EE_0^n[ \EE_0^n(\1_{A_n} \mid G_n) ]  = \EE_0^n[K_n(G_n,A_n)]$, it must be that $\PP_0^n(K_n(G_n,A_n) > \varepsilon) \to 0$. But $\{\omega \in \Omega_n\;:\, K_n(G_n(\omega),A_n) > \varepsilon\} \in \sigma(G_n)$, so by 1 $\PP_1^n(K_n(G_n,A_n) > \varepsilon) \to 0$. Since $\PP_1^n(A_n) = \EE_1^n[K_n(G_n,A_n)] \leq \varepsilon + \PP_1^n(K_n(G_n,A_n) > \varepsilon)$, and since $\varepsilon$ is arbitrary, the result follows.

$2\implies 3$. Let $(E_n)_{n\geq 1}$ be a sequence such that $\PP_0^n(s(G_n) \in E_n) \to 0$. Remark that $\PP_0^n(s(G_n) \in E_n) = \PP_0^n(s(\pi_n(G_n)) \in E_n) = \PP_0^n(\pi_n(G_n) \in s^{-1}(E_n))$. So $\PP_1^n(s(G_n) \in E_n) = \PP_1^n(s(\pi_n(G_n)) \in E_n) =  \PP_1^n(\pi_n(G_n) \in s^{-1}(E_n))$ goes to zero by 2.

\subsection{Proof of Proposition~\ref{prop:contiguity_intermediate_regime}}
\label{sec:proof-prop-second-moment}

\subsubsection{Derivation of the expression of the likelihood-ratio}

In this section we determine the expression of the likelihood ratio $\frac{\intd Q_1^{n,p}}{\intd Q_0^{n,p}}$.

  \begin{lemma}
    \label{lem:lkl:perm}
    Let $S_{t,i-1}$ as defined in Lemma~\ref{lem:support}. $\PP_0^n$-almost-surely:
    \begin{equation*}
      Y_n = \frac{\intd Q_1^{n,p}}{\intd Q_0^{n,p}}(\pi_n(G_n))%
      = \frac{1}{|\Pi_{n}(G_n)|}\sum_{\Bar{\pi}\in \Pi_{n}(G_n)}\prod_{t=\tau_n+1}^n\prod_{i = 1}^{m}\frac{S_{t,i-1}(\delta_0)}{S_{t,i-1}(\delta_1)} %
  \frac{\D_{G_{\Bar{\pi}(t),i-1}}(V_{\Bar{\pi}(t),i})+\delta_1}{\D_{G_{\Bar{\pi}(t),i-1}}(V_{\Bar{\pi}(t),i})+\delta_0}.
    \end{equation*}
  \end{lemma}
  \begin{proof}
    Let $\g_n\in \mathfrak{S}_n$ and $\pi_0\in \Pi_{n}(\g_n)$. Then, for $j=0,1$
    \begin{align*}
      \PP^{n}_{j}\big(\pi_n(G_n) = \pi_0(\g_n)\big)%
      &= \EE_j^n\left[\EE_0^n\Bigg(\sum_{\Bar{\pi}\in \Pi_{n}(\G_n)}\1_{\Bar{\pi}(\G_n) = \pi_{0}(\g_n), \pi_n = \Bar{\pi}} \Big\vert G_n\Bigg)\right]\\
      &= \EE_j^n\left[\sum_{\Bar{\pi}\in \Pi_{n}(\G_n)}\1_{\Bar{\pi}(\G_n) = \pi_{0}(\g_n)}\PP_0^n(\pi_n= \Bar{\pi} \mid G_n)\right]\\
      &=\EE_j^n\left[\frac{1}{|\Pi_{n}(\G_n)|}\sum_{\Bar{\pi}\in \Pi_{n}(\G_n)}\1_{\Bar{\pi}(\G_n) = \pi_{0}(\g_n)}\right]
    \end{align*}
    Now remark that $\Bar{\pi} \in \Pi_n(G_n)$ leaves invariant $\PartB(G_n)$ and $\pi_0 \in \Pi_n(\g_n)$ leaves invariant $\PartB(\g_n)$, thus
    \begin{align*}
        \Bar{\pi}(\G_n) = \pi_0(\g_n)&\implies \PartB(\Bar{\pi}(\G_n)) = \PartB(\pi_0(\g_n))\\
        &\implies\PartB(\G_n) = \PartB(\g_n)\\
        &\implies \Pi_{n}(\G_n) = \Pi_{n}(\g_n).
    \end{align*}
    It follows
    \begin{align*}
        \PP^{n}_{j}\big(\pi_n(G_n) = \pi_0(\g_n)\big)%
      &= \EE_j^n\left[\frac{1}{|\Pi_{n}(\g_n)|}\sum_{\Bar{\pi}\in \Pi_{n}(\g_n)}\1_{\Bar{\pi}(\G_n) = \pi_{0}(\g_n)}\right]\\
        &= \frac{1}{|\Pi_{n}(\g_n)|}\sum_{\Bar{\pi}\in \Pi_{n}(\g_n)} \PP_j^n\left(\Bar{\pi} (\G_n) = \pi_{0}(\g_n)\right)\\
        &= \frac{1}{|\Pi_{n}(\g_n)|}\sum_{\Bar{\pi}\in \Pi_{n}(\g_n)} \PP^{n}_j\left(\Bar{\pi} (\G_n) = \g_n\right)\\
        &= \frac{1}{|\Pi_{n}(\g_n)|}\sum_{\Bar{\pi}\in \Pi_{n}(\g_n)} \PP_j^n\big( \G_n = \Bar{\pi}^{-1}(\g_n)\big)
\end{align*}
To simplify the notations in what follows, we denote respectively $\bar{\pi}^{-1}(G_n)$ and $\bar{\pi}^{-1}(\g_n)$ by $\bar{G}_n$ and $\bar{\g}_n$. Note that the advantage of permuting only bold vertices is that the set $\Pi_{n}(\G_n)$ is a group, which makes the expression of the likelihood ratio easier to handle. As shown in Lemma~\ref{lem:lkl:null}, the likelihood of the labeled graph $\bar{\g}_n$ under the null hypothesis does not depend on the permutation $\Bar{\pi}$ when $\Bar{\pi}\in\Pi_{n}(\g_n)$. It follows that $\PP^{n}_{0}(\pi_n(G_n) = \pi_0(\g_n)) = \PP^n_{0}\big(\G_n = \bar{\g}_n\big)$ for every $\Bar{\pi}\in\Pi_{n}(\g_n)$. Furthermore, the Lemma~\ref{lem:perm:invariant-support} below guarantees that $\bar{\g}_n\in\mathfrak{S}_n$ whenever $\g_n \in \mathfrak{S}_n$. Then by Lemma~\ref{lem:lkl:ratio}
\begin{align*}
  Y_n
  &= \frac{1}{|\Pi_{n}(G_n)|}\sum_{\Bar{\pi}\in \Pi_{n}(G_n)} \frac{\intd Q_1^n}{\intd Q_0^n}(\Bar{G}_n) \\
  &=
  \frac{1}{|\Pi_{n}(G_n)|}\sum_{\Bar{\pi}\in \Pi_{n}(G_n)}\prod_{t=\tau_n+1}^n\prod_{i = 1}^{m}\frac{S_{t,i-1}(\delta_0)}{S_{t,i-1}(\delta_1)} \prod_{k=m}^{nm}\left(\frac{k+\delta_1}{k+\delta_0}\right)^{N_{>k}(\Bar{G}_n) - N_{>k}((\Bar{G}_n)_{\tau_n})}%
\end{align*}
with $(\bar{G}_n)_t \equiv \bar{G}_n[ \llbracket 0,t \rrbracket]$ for all $t\in \llbracket 1,n\rrbracket$. Let $\Bar{\pi}\in \Pi_n(\g_n)$ be arbitrary. Then,
\begin{align*}
  N_{>k}(\bar{G}_n) - N_{>k}(\left(\bar{G}_n\right)_{\tau_n})%
  &= \sum_{t=\tau_n+1}^n\sum_{s\in \Chi_{(\bar{G}_n)_t}(t)}\1\Big(k + 1 - \mu_{(\bar{G}_n)_t}(t,s) \leq \D_{(\bar{G}_n)_{t-1}}(s) \leq k \Big).
\end{align*}
Remark that for $s\in \Chi_{(\bar{G}_n)_t}(t)$ is must be that $\D_{G_n}(s) > m$ and hence $\bar{\pi}(s) = s$. In particular $\Chi_{(\bar{G}_n)_t}(t) = \Chi_{G_{\Bar{\pi}(t)}}(\Bar{\pi}(t))$ and $\mu_{(\bar{G}_n)_t}(t,s) = \mu_{G_{\Bar{\pi}(t)}}(\Bar{\pi}(t),s)$. In addition, using the Lemma~\ref{lem:labels:consistency_general}, we deduce that
\begin{align*}
  N_{>k}(\bar{G}_n) - N_{>k}(\left(\bar{G}_n\right)_{\tau_n})%
  &= \sum_{t=\tau_n+1}^n\sum_{s\in \Chi_{G_{\Bar{\pi}(t)}}(\bar{\pi}(t))}\1\Big(k + 1 - \mu_{G_{\Bar{\pi}(t)}}(\Bar{\pi}(t),s) \leq \D_{G_{\Bar{\pi}(t)-1}}(s) \leq k \Big)\\
  &= \sum_{t=\tau_n+1}^n\sum_{i=1}^m\1\Big( \D_{G_{\Bar{\pi}(t),i-1}}(V_{\Bar{\pi}(t),i}) = k \Big)
\end{align*}
where the last line follows because the fact that vertex $\Bar{\pi}(t)$ has a child whose degree is $\leq k$ at instant $\Bar{\pi}(t)-1$ but $>k$ at instant $\Bar{\pi}(t)$ is equivalent to the fact that vertex $\Bar{\pi}(t)$ choose a vertex $V_{\Bar{\pi}(t),i}$ of degree $k$ in $G_{\Bar{\pi}(t),i-1}$ for some $i=1,\dots,m$. Consequently
\begin{equation*}
  Y_n%
  =
  \frac{1}{|\Pi_{n}(G_n)|}\sum_{\Bar{\pi}\in \Pi_{n}(G_n)}\prod_{t=\tau_n+1}^n\prod_{i = 1}^{m}\frac{S_{t,i-1}(\delta_0)}{S_{t,i-1}(\delta_1)} %
  \frac{\D_{G_{\Bar{\pi}(t),i-1}}(V_{\Bar{\pi}(t),i})+\delta_1}{\D_{G_{\Bar{\pi}(t),i-1}}(V_{\Bar{\pi}(t),i})+\delta_0}.%
  \qedhere
\end{equation*}
  \end{proof}

The following lemma shows that the permutations appearing in the proof of Lemma \ref{lem:lkl:perm} leave invariant the support of preferential attachment graphs. 
  \begin{lemma}
    \label{lem:perm:invariant-support}
    Let $\g_n \in \mathfrak{S}_n$ and let $\bar{\pi}\in \Pi_n(\g_n)$. Then $\bar{\pi}^{-1}(\g_n) \in \mathfrak{S}_n$.
  \end{lemma}
  \begin{proof}
    In view of Lemma~\ref{lem:support}, the set $\mathfrak{S}_n$ is the set of directed labeled graphs on vertex set $\llbracket 0,n\rrbracket$ where each non-zero vertex has out-degree exactly $m$ and arrows are all directed from largest to smallest label. Since $\bar{\pi} \in \Pi_n(\g_n)$, it permutes only the labels of vertices in $\PartB(\g_n)$. But any $v \in \PartB(\g_n)$ must satisfy $v > \tau_n'$ and have all of its children $c_1,\dots,c_k$ in $\llbracket 0,\tau_n'\rrbracket$. So $\Bar{\pi}^{-1}(v) > \tau_n'$ as well and $\Bar{\pi}^{-1}(c_j) = c_j$ for all its children. In other words, the out-degree of any vertex in $\bar{\pi}^{-1}(\g_n)$ is also $m$ and the arrows are all directed from largest to smallest label, as required.
  \end{proof}

  The following lemma shows that the degrees of the children of the $n-\tau_n$ late vertices remains invariant after the application of permutation $\bar{\pi}$.
\begin{lemma}
  \label{lem:labels:consistency_general}
  Let $\g_n \in \mathfrak{S}_n$, $\Bar{\pi} \in \Pi_n(\g_n)$ and let $\Bar{\pi}^{-1}(\g_n)_{t} = \Bar{\pi}^{-1}(\g_n)[\llbracket 0,t\rrbracket]$ and $\g_t = \g_n[\llbracket 0,t\rrbracket]$ for all $t\in \llbracket 1,n\rrbracket$. Then for all $t \in \llbracket \tau_n+1,n\rrbracket$ and all $s\in \Chi_{\Bar{\pi}^{-1}(\g_n)_t}(t)$:
  \begin{equation*}
    \D_{\Bar{\pi}^{-1}(\g_n)_{t-1}}(s) = \D_{\g_{\bar{\pi}(t)-1}}(s).
  \end{equation*}
\end{lemma}
\begin{proof}
  Let $\Bar{\pi}\in \Pi_n(\g_n)$, $t \in \llbracket \tau_n+1,n\rrbracket$ and $s \in \Chi_{\bar{\pi}^{-1}(\g_n)_t}(t)$. Observe that since $s\in \Chi_{\bar{\pi}^{-1}(\g_n)_t}(t)$ it is necessary that $\D_{\g_n}(s) > m$ and then $\bar{\pi}(s) = s$.

    Suppose first that for all $t' \in \Par_{\Bar{\pi}^{-1}(\g_n)}(s)$ we have $\Bar{\pi}(t') = t'$. Then $s \in \Chi_{\g_t}(t)$ and $\D_{\Bar{\pi}^{-1}(\g_n)_{t-1}}(s) = \D_{\g_{t-1}}(s) = \D_{\g_{\Bar{\pi}(t)-1}}(s)$.

    Second, suppose there exists $t' \in \Par_{\Bar{\pi}^{-1}(\g_n)}(s)$ such that $\bar{\pi}(t') \ne t'$. It is necessary that $\Bar{\pi}(t') > \tau_n'$ since $\bar{\pi}$ permute only the labels in $\PartB(\g_n) \subset \llbracket \tau_n'+1,n\rrbracket$. Furthermore $\Bar{\pi}(t') \in \PartB(\g_n)$ so it must be that $\Par_{\g_n}(s)\backslash \{\bar{\pi}(t')\} \subset \llbracket 0,\tau_n'\rrbracket$. Let enumerate $v_1<\dots <v_r$ the elements of $\Par_{\g_n}(s) \backslash \{\Bar{\pi}(t')\}$. Hence the elements of $\Par_{g_n}(s)$ are $v_1 < \dots < v_r < \bar{\pi}(t')$. Since $v_1<\dots < v_r \leq \tau_n'$ they are not in $\PartB(\g_n)$ and thus $\Bar{\pi}(v_j) = v_j$ for all $j=1,\dots,r$. It follows that the elements of $\Par_{\Bar{\pi}^{-1}(\g_n)}(s)$ are $v_1,\dots,v_r,t'$ and satisfy
    \begin{equation*}
      v_1 < \dots < v_r \leq \tau_n' < t'
    \end{equation*}
    because $\Bar{\pi}(t') > \tau_n' \implies t' > \tau_n'$. Therefore $t'=t$ and $\D_{\Bar{\pi}^{-1}(\g_n)_{t-1}}(s) = \D_{\g_{\tau_n'}}(s) = \D_{\g_{\Bar{\pi}(t)-1}}(s)$.
\end{proof}

\subsubsection{Bound on the second moment of the likelihood ratio}

As in Lemma~\ref{lem:lkl:perm} we let $Y_n \equiv \frac{\intd Q_1^{n,p}}{\intd Q_0^{n,p}}(\pi_n(G_n))$ for simplicity. Then by Lemmas~\ref{lem:lkl:perm} and~\ref{lem:boundSti}, since $\tau_n > \tau_n' \geq 3$
  \begin{multline*}
    \EE_0^n\big(Y_n^2\1_{B_n}\big)%
    \leq e^{12m\Delta_n/\tau_n}\left(\frac{2m+\delta_0}{2m+\delta_1}\right)^{2m\Delta_n}\\
    \times \EE_0^n\Biggl[\frac{\sum_{\pi, \Bar{\pi}\in \Pi_{n}(\G_n)}}{|\Pi_{n}(\G_n)|^{2}}\prod_{t= \tau_n +1}^{n}\prod_{i=1}^{m}\frac{\D_{\G_{\Bar{\pi}(t), i-1}}(\V_{\Bar{\pi}(t), i}) + \delta_1}{\D_{\G_{\Bar{\pi}(t), i-1}}(\V_{\Bar{\pi}(t), i}) + \delta_0}\frac{\D_{\G_{\pi(t), i-1}}(\V_{\pi(t), i}) + \delta_1}{\D_{\G_{\pi(t), i-1}}(\V_{\pi(t), i}) + \delta_0}\1_{\mathrm{B}_n}\Biggr].
  \end{multline*}
Observe that $|\Pi_{n}(\G_n)| = \big|\PartB(G_n)\big|!$. Moreover, on the event $B_n$ we have forced that $\llbracket \tau_n+1,n \rrbracket \subset \PartB(G_n)$. This implies that on $B_n$ we have $\Bar{\pi}(t) \in \PartB(G_n)$ for all $t\in \llbracket \tau_n+1,n\rrbracket$. Consequently on $B_n$,
\begin{multline*}
  \frac{\sum_{\pi, \Bar{\pi}\in \Pi_{n}(\G_n)}}{\left|\Pi_{n}(\G_n)\right|^{2}}\prod_{t= \tau_n +1}^{n}\prod_{i=1}^{m}\frac{\D_{\G_{\Bar{\pi}(t),i-1}}(\V_{\Bar{\pi}(t),i}) + \delta_1}{\D_{\G_{\Bar{\pi}(t),i-1}}(\V_{\Bar{\pi}(t),i}) + \delta_0}\frac{\D_{G_{\pi(t), i-1}}(\V_{\pi(t), i}) + \delta_1}{\D_{\G_{\pi(t), i-1}}(\V_{\pi(t), i}) + \delta_0}\\
  \leq \frac{1}{\left|\PartB(G_n)\right|!^{2}}\sum_{\substack{k^{\prime}_{\tau_n + 1}\neq \dots \neq k^{\prime}_n \in \PartB(G_n)\\ k_{\tau_n + 1}\neq ..\neq k_n \in \PartB(G_n)}}\sum_{\substack{\pi, \Bar{\pi}\in\Pi_n (\G_n)\\ (\pi(\tau_{n+1}),\dots,\pi(n)) = (k_{\tau_n + 1},\dots, k_n)\\
      (\Bar{\pi}(\tau_n + 1),\dots,\Bar{\pi}(n)) = (k^{\prime}_{\tau_n + 1}, \dots, k^{\prime}_n)}}\\
  \times \prod_{t= \tau_n +1}^{n}\prod_{i=1}^{m}\frac{\D_{\G_{k^{\prime}_t, i-1}}(\V_{k^{\prime}_t, i}) + \delta_1}{\D_{\G_{k^{\prime}_t, i-1}}(\V_{k^{\prime}_t, i}) + \delta_0}\frac{\D_{\G_{k_t, i-1}}(\V_{k_t, i}) + \delta_1}{\D_{\G_{k_t, i-1}}(\V_{k_t, i}) + \delta_0}
\end{multline*}
which can be further bounded above by
\begin{align*}
  &\leq \frac{\left(\left|\PartB(G_n)\right| - \Delta_n\right)!^{2}}{\left|\PartB(G_n)\right|!^{2}}\sum_{\substack{k^{\prime}_{\tau_n + 1}\neq \dots\neq k^{\prime}_n \in \PartB(G_n)\\ k_{\tau_n + 1}\neq \dots \neq k_n \in \PartB(G_n)}}\prod_{t= \tau_n +1}^{n}\prod_{i=1}^{m}\frac{\D_{\G_{k^{\prime}_t, i-1}}(\V_{k^{\prime}_t, i}) + \delta_1}{\D_{\G_{k^{\prime}_t, i-1}}(\V_{k^{\prime}_t, i}) + \delta_0}\frac{\D_{\G_{k_t, i-1}}(\V_{k_t, i}) + \delta_1}{\D_{\G_{k_t, i-1}}(\V_{k_t, i}) + \delta_0}\\
  &\leq \frac{\left(\left|\PartB(G_n)\right| - \Delta_n\right)!^{2}}{\left|\PartB(G_n)\right|!^{2}}\sum_{\substack{k^{\prime}_{\tau_n + 1},\dots, k^{\prime}_n \in \PartB(G_n)\\ k_{\tau_n + 1}, \dots,  k_n \in \PartB(G_n)}}\prod_{t= \tau_n +1}^{n}\prod_{i=1}^{m}\frac{\D_{\G_{k^{\prime}_t, i-1}}(\V_{k^{\prime}_t, i}) + \delta_1}{\D_{\G_{k^{\prime}_t, i-1}}(\V_{k^{\prime}_t, i}) + \delta_0}\frac{\D_{\G_{k_t, i-1}}(\V_{k_t, i}) + \delta_1}{\D_{\G_{k_t, i-1}}(\V_{k_t, i}) + \delta_0}\\
  &= \frac{\left(\left|\PartB(G_n)\right| - \Delta_n\right)!^{2}}{\left|\PartB(G_n)\right|!^{2}}\Bigg( \sum_{k\in \PartB(G_n)}\prod_{i=1}^{m}\frac{\D_{\G_{k, i-1}}(\V_{k, i}) + \delta_1}{\D_{\G_{k, i-1}}(\V_{k, i}) + \delta_0} \Bigg)^{2\Delta_n}\\
  &\leq \frac{\left(\left|\PartB(G_n)\right| - \Delta_n\right)!^{2}}{\left|\PartB(G_n)\right|!^{2}}\Bigg( \sum_{k = \tau_n'+1}^n\prod_{i=1}^{m}\frac{\D_{\G_{k, i-1}}(\V_{k, i}) + \delta_1}{\D_{\G_{k, i-1}}(\V_{k, i}) + \delta_0} \Bigg)^{2\Delta_n}
\end{align*}

Next, we use that for any non-negative integer $\sqrt{2\pi n}(n/e)^n < n! < \sqrt{2\pi n}(n/e)^ne^{1/(12n)}$ (see for instance \cite[Section~3.6]{Tem96}) which entails that for any $\nu> k \geq 1$
    \begin{align*}
      \frac{(\nu-k)!}{\nu!}%
      \leq \frac{\sqrt{\nu-k}\big(\frac{\nu-k}{e} \big)^{\nu-k}e^{\frac{1}{12(\nu-k)}}}{\sqrt{\nu}\big( \frac{\nu}{e}\big)^{\nu}}
      = \Big(1 - \frac{k}{\nu}\Big)^{\nu-k+\frac{1}{2}}\Big(\frac{e}{\nu}\Big)^ke^{\frac{1}{12(\nu-k)}}
      \leq \nu^{-k} e^{k^2/\nu + \frac{1}{12(\nu-k)}}.
    \end{align*}
    Since on the event $B_n$ it holds that $|\PartB(G_n)| \geq \Delta_n'\big(1 - \frac{\alpha_n\Delta_n'}{\tau_n'}\big)$, we deduce that%
    \begin{align*}
      \frac{\left(\left|\PartB(G_n)\right| - \Delta_n\right)!}{\left|\PartB(G_n)\right|!}
      &\leq\frac{1}{|\PartB(G_n)|^{\Delta_n}}\exp\Big(\frac{\Delta_n^2}{|\PartB(G_n)|} + \frac{1}{12(|\PartB(G_n)| - \Delta_n )} \Big)\\
      &\leq \frac{1}{(\Delta_n')^{\Delta_n}}\exp\Big(-\Delta_n\log\Big(1 - \frac{\alpha_n\Delta_n'}{\tau_n'} \Big) + \frac{\Delta_n^2}{|\PartB(G_n)|} + \frac{1}{12(|\PartB(G_n)| - \Delta_n )} \Big)\\
      &\leq \frac{1}{(\Delta_n')^{\Delta_n}}\exp\Big( \frac{\alpha_n\Delta_n\Delta_n'}{\tau_n' - \alpha_n\Delta_n'} + \frac{\Delta_n^2}{|\PartB(G_n)|} + \frac{1}{12(|\PartB(G_n)| - \Delta_n )} \Big)\\
      &\leq \frac{1}{(\Delta_n')^{\Delta_n}}\exp\Big(\frac{2\alpha_n\Delta_n\Delta_n'}{\tau_n'} + \frac{2\Delta_n^2}{\Delta_n'} + \frac{1}{3\Delta_n'} \Big)
    \end{align*}
    where in the last line we have used the assumptions that $\frac{\alpha_n\Delta_n'}{\tau_n'} \leq \frac{1}{2}$ and $\Delta_n \leq \frac{1}{4}\Delta_n'$, which imply that $|\PartB(G_n)|\geq \frac{\Delta_n'}{2}$ and $|\PartB(G_n)| - \Delta_n \geq \frac{\Delta_n'}{4}$. Hence one obtains the bound [here we use that $\frac{6m\Delta_n}{\tau_n} + \frac{4\Delta_n^2}{\Delta_n'} \leq \frac{10m\Delta_n^2}{\Delta_n'}$]
    \begin{equation*}
      \EE_0^n\big(Y_n^2\1_{B_n} \big)%
      \leq e^{\frac{4\alpha_n\Delta_n\Delta_n'}{\tau_n'} + \frac{10m\Delta_n^2}{\Delta_n'} + \frac{2}{3\Delta_n'}}\Big(\frac{2m+\delta_0}{2m+\delta_1} \Big)^{2m\Delta_n}%
      \EE_0^n\Bigg(\Bigg(\frac{1}{\Delta_n'}\sum_{k = \tau_n'+1}^n\prod_{i=1}^{m}\frac{\D_{\G_{k, i-1}}(\V_{k, i}) + \delta_1}{\D_{\G_{k, i-1}}(\V_{k, i}) + \delta_0} \Bigg)^{2\Delta_n}\Bigg).
    \end{equation*}
    Letting $Z_n$ and $m_n$ as in Lemma~\ref{lem:lr:martingale}, we deduce from said lemma that
  \begin{align*}
   \EE_0^n\Bigg(\Bigg(\frac{1}{\Delta_n'}\sum_{k = \tau_n'+1}^n\prod_{i=1}^{m}\frac{\D_{\G_{k, i-1}}(\V_{k, i}) + \delta_1}{\D_{\G_{k, i-1}}(\V_{k, i}) + \delta_0} \Bigg)^{2\Delta_n}\Bigg)%
  &\leq  \ m_n^{2\Delta_n}\EE_0^n\Bigg(\Big(1+ \frac{Z_n - m_n}{m_n} \Big)^{2\Delta_n} \Bigg)\\
  &= m_n^{2\Delta_n} \int_0^{\infty}\PP_0^n\Bigg( \Big(1+ \frac{Z_n - m_n}{m_n} \Big)^{2\Delta_n} > x \Bigg)\intd x\\
  &= m_n^{2\Delta_n} \int_0^{\infty}\PP_0^n\Big(Z_n - m_n > m_n\big(x^{\frac{1}{2\Delta_n}} -1\big) \Big)\intd x\\
  &\leq m_n^{2\Delta_n}\Bigg(1%
    + \int_1^{\infty}\PP_0^n\Big(Z_n - m_n > \frac{m_n\log(x)}{2\Delta_n} \Big)\intd x\Bigg)\\
  &\leq m_n^{2\Delta_n}\Bigg(1%
    + \int_1^{\infty}\exp\Big(-\frac{c\Delta_n' m_n^2}{4\Delta_n^2}\log(x)^2 \Big)\intd x\Bigg).
  \end{align*}
  Using Lemma~\ref{lem:boundintegral} to upper bound the last integral, together with Lemma~\ref{lem:boundmn} implying that $m_n \geq e^{-3m}\big(\frac{2m+\delta_1}{2m+\delta_0} \big)^m$ since $\tau_n' \geq 3$, it is found that there are constants $c_1,c_2 > 0$ depending only on $\delta_0$, $\delta_1$, and $m$, such that
  \begin{equation*}
    \EE_0^n\Bigg(\Bigg(\frac{1}{\Delta_n'}\sum_{k = \tau_n'+1}^n\prod_{i=1}^{m}\frac{\D_{\G_{k, i-1}}(\V_{k, i}) + \delta_1}{\D_{\G_{k, i-1}}(\V_{k, i}) + \delta_0} \Bigg)^{2\Delta_n}\Bigg)%
    \leq m_n^{2\Delta_n}\Bigg(1 + \sqrt{\frac{c_1 \Delta_n^2}{\Delta_n'}}e^{\frac{c_2 \Delta_n^2}{\Delta_n'}}\Bigg).
  \end{equation*}
  Finally, summarizing everything and using Lemma~\ref{lem:boundmn} to get an upper bound on $m_n$, we find that [here we use that $\frac{12m\Delta_n}{\tau_n'} + \frac{10m\Delta_n^2}{\Delta_n'} \leq \frac{22m\Delta_n^2}{\Delta_n'}$]
  \begin{equation*}
    \log \EE_0^n\big(Y_n^2\1_{B_n}\big)%
    \leq%
    \frac{4\alpha_n\Delta_n\Delta_n'}{\tau_n'}%
    + \frac{22m \Delta_n^2}{\Delta_n'}%
    + \frac{2}{3\Delta_n'}%
    + \sqrt{\frac{c_1 \Delta_n^2}{\Delta_n'}}e^{\frac{c_2 \Delta_n^2}{\Delta_n'}}.
  \end{equation*}

\subsubsection{Auxiliary results used to prove the Proposition~\ref{prop:contiguity_intermediate_regime}}

\begin{lemma}
    \label{lem:lr:martingale}
    Let
    \begin{equation*}
      Z_n = \frac{1}{\Delta_n'}\sum_{k = \tau_n' + 1}^{n}\prod_{i=1}^{m}\frac{\D_{\G_{k, i-1}}(\V_{k, i}) + \delta_1}{\D_{\G_{k, i-1}}(\V_{k, i}) + \delta_0},\qquad%
       m_n = \frac{1}{\Delta_n'}\sum_{k=\tau_n'+1}^n\prod_{i=1}^m\frac{S_{k,i-1}(\delta_1)}{S_{k,i-1}(\delta_0)}.
 \end{equation*}
 Then there exists a constant $c>0$ depending only on $\delta_0$, $\delta_1$, and $m$, such that for all $x\geq 0$
 \begin{equation*}
   \PP_0^n\big(Z_n - m_n \geq x \big) \leq e^{-c\Delta_n' x^2}.
 \end{equation*}
\end{lemma}
\begin{proof}
  In the proof we let $\mathcal{F}_t = \sigma(G_1,\dots,G_t)$ and $\mathcal{F}_{t,i} = \sigma(G_1,\dots,G_{t-1},G_{t,1},\dots,G_{t,i})$ for $t=1,\dots,n$ and $i=1,\dots,m$. Let $W_k = \prod_{i=1}^m\frac{\D_{G_{k,i-1}}(V_{k,i}) + \delta_1}{\D_{G_{k,i-1}}(V_{k,i}) + \delta_0}$ for $k=\tau_n'+1,\dots, n$. Clearly $\EE_0^n(W_k \mid \mathcal{F}_t) = W_k$ for all $t \geq k$. Also,
  \begin{align*}
    \EE_0^n\big( W_k \mid \mathcal{F}_{k-1} \big)%
    &=\EE_0^n\big( \EE_0^n(W_k \mid \mathcal{F}_{k,m-1})  \mid \mathcal{F}_{k-1} \big)\\%
    &=\EE_0^n\Bigg(\prod_{i=1}^{m-1}\frac{\D_{G_{k,i-1}}(V_{k,i}) + \delta_1}{\D_{G_{k,i-1}}(V_{k,i}) + \delta_0} \EE_0^n\Big(\frac{\D_{G_{k,m-1}}(V_{k,m}) + \delta_1}{\D_{G_{k,m-1}}(V_{k,m}) + \delta_0} \mid \mathcal{F}_{k,m-1}\Big)  \mid \mathcal{F}_{k-1} \Bigg)\\
    &= \EE_0^n\Bigg(\prod_{i=1}^{m-1}\frac{\D_{G_{k,i-1}}(V_{k,i}) + \delta_1}{\D_{G_{k,i-1}}(V_{k,i}) + \delta_0} \sum_{u=0}^{k-1}\frac{\D_{G_{k,m-1}}(u) + \delta_1}{\D_{G_{k,m-1}}(u) + \delta_0}\frac{\D_{G_{k,m-1}}(u) + \delta_0}{S_{k,m-1}(\delta_0)}  \mid \mathcal{F}_{k-1} \Bigg)\\
    &= \EE_0^n\Bigg(\prod_{i=1}^{m-1}\frac{\D_{G_{k,i-1}}(V_{k,i}) + \delta_1}{\D_{G_{k,i-1}}(V_{k,i}) + \delta_0} \frac{S_{k,m-1}(\delta_1)}{S_{k,m-1}(\delta_0)}  \mid \mathcal{F}_{k-1} \Bigg).
  \end{align*}
  Continuing inductively, it is found that
  \begin{align*}
    \EE_0^n\big(W_k \mid \mathcal{F}_{k-1})%
    = \prod_{i=1}^m\frac{S_{k,i-1}(\delta_1)}{S_{k,i-1}(\delta_0)}
  \end{align*}
  and then $\EE_0^n\big(W_k \mid \mathcal{F}_{\ell})%
  = \prod_{i=1}^m\frac{S_{k,i-1}(\delta_1)}{S_{k,i-1}(\delta_0)}$ for all $\ell < k$. Deduce that for all $k=\tau_{n}'+1,\dots,n$ and all $\ell=\tau_n'+1,\dots,n$
  \begin{equation}
    \label{eq:estim-increments}
    \EE_0^n\big(W_k \mid \mathcal{F}_{\ell}) -
    \EE_0^n\big(W_k \mid \mathcal{F}_{\ell-1})%
    =%
    \begin{cases}
      0 &\mathrm{if}\ \ell \ne k,\\
      \prod_{i=1}^m\frac{\D_{G_{k,i-1}}(V_{k,i}) + \delta_1}{\D_{G_{k,i-1}}(V_{k,i}) + \delta_0}%
      - \prod_{i=1}^m\frac{S_{k,i-1}(\delta_1)}{S_{k,i-1}(\delta_0)} &\mathrm{if}\ \ell = k.
    \end{cases}
  \end{equation}
  Build the Doob martingale $M_j =  \Delta_n'\EE(Z_n \mid \mathcal{F}_j)$ and observe that
  \begin{equation*}
    \sum_{j=\tau_n'+1}^{n}(M_{j} - M_{j-1})%
    = \Delta_n'\Big(Z_n - \EE_0^n(Z_n \mid \mathcal{F}_{\tau_n'})\Big).
  \end{equation*}
  Furthermore for every $j=\tau_n'+1,\dots,n$, by equation~\eqref{eq:estim-increments}
  \begin{align*}
    |M_j - M_{j-1}|%
    &= \Bigg|\sum_{k=\tau_n'+1}^n\Big(\EE_0^n(W_k \mid \mathcal{F}_j) - \EE_0^n(W_k \mid \mathcal{F}_{j-1}) \Big) \Bigg|\\
    &=| W_j - \EE_0^n(W_j \mid \mathcal{F}_{j-1})|\\
    &\leq \max\Big(1,\, \frac{m+\delta_1}{m+\delta_0}\Big)^m
  \end{align*}
  because
  \begin{equation*}
    W_j%
    =\prod_{i=1}^m\Big(1 + \frac{\delta_1 - \delta_0}{\D_{G_{k,i-1}}(V_{k,i}) + \delta_0}\Big)%
    \leq \max\Big(1,\, \frac{m+\delta_1}{m+\delta_0}\Big)^m. 
  \end{equation*}
  By Hoeffding-Azuma's inequality, for all $x\geq 0$, almost-surely
  \begin{align*}
    \PP_0^n\Big(Z_n - \EE_0^n(Z_n\mid \mathcal{F}_{\tau_n'}) \geq \frac{x}{\Delta_n'} \mid \mathcal{F}_{\tau_n'} \Big)
    &=\PP_0^n\big( M_n - M_{\tau_n'} \geq x \mid \mathcal{F}_{\tau_n'})\\%
    &\leq \exp\Bigg(- \frac{x^2}{2\Delta_n'\max\big(1,\, \frac{m+\delta_1}{m+\delta_0}\big)^m} \Bigg).
  \end{align*}
  Then the result follows by taking the expectation both sides of the last display and by noticing that $\EE_0^n(Z_n \mid \mathcal{F}_{\tau_n'}) = m_n$ almost-surely.
\end{proof}

  \begin{lemma}
    \label{lem:boundintegral}
    For every $\beta > 0$
    \begin{equation*}
      0 \leq \int_1^{\infty}e^{-\beta \log(x)^2}\intd x%
      \leq \sqrt{ \frac{\pi e^{1/(2\beta)}}{\beta} }.
    \end{equation*}
  \end{lemma}
  \begin{proof}
    It is found after a straightforward change of variable that
    \begin{equation*}
      \int_1^{\infty}e^{-\beta \log(x)^2}\intd x%
      = \frac{1}{\sqrt{2\beta}}\int_0^{\infty}e^{-\frac{1}{2}y^2}e^{\frac{1}{\sqrt{2\beta}}y}\intd y
      = \frac{e^{\frac{1}{4\beta}}}{\sqrt{2\beta}}\int_0^{\infty}e^{-\frac{1}{2}(y - \frac{1}{\sqrt{2c}})^2}\intd y%
      \leq \sqrt{ \frac{\pi e^{1/(2\beta)}}{\beta} }.%
      \qedhere
    \end{equation*}
  \end{proof}

  \begin{lemma}
    \label{lem:boundmn}
    For every $\tau_n' \geq 3$
    \begin{equation*}
      e^{-\frac{6m}{\tau_n'}}\Big(\frac{2m+\delta_1}{2m+\delta_0} \Big)^m%
      \leq m_n \leq e^{\frac{6m}{\tau_n'}}\Big(\frac{2m+\delta_1}{2m+\delta_0} \Big)^m.
    \end{equation*}
  \end{lemma}
  \begin{proof}
    As in Lemma~\ref{lem:boundSti}, we have whenever $k > \tau_n'$ that
    \begin{equation*}
      \big(1-2/\tau_n')^m\Big(\frac{2m+\delta_1}{2m+\delta_0}\Big)^m%
      \leq \prod_{i=1}^m\frac{S_{k,i-1}(\delta_1)}{S_{k,i-1}(\delta_0)}%
      \leq \Big(\frac{2m+\delta_1}{2m+\delta_0}\Big)^m\frac{1}{(1-2/\tau_n')^m}.
    \end{equation*}
    Hence the result follows since $\log(1-2/\tau_n') \geq -\frac{2}{\tau_n'-2} \geq -\frac{6}{\tau_n'}$ for $\tau_n' \geq 3$.
  \end{proof}

\subsection{Proof of Proposition~\ref{pro:eventBn}}
\label{sec:proof-prop-eventBn}

\subsubsection{Upper bound on the probabilities}

  By Markov's inequality
  \begin{align*}
    \PP_1^n\Big(|\PartB(G_n)| < \Delta_n'\Big(1 - \frac{\alpha_n\Delta_n'}{\tau_n'} \Big) \Big)%
    &=\PP_1^n\Big(|\llbracket \tau_n'+1,n \rrbracket \backslash \PartB(G_n)| >  \frac{\alpha_n(\Delta_n')^2}{\tau_n'} \Big)\\
    &\leq \frac{\tau_n'}{\alpha_n(\Delta_n')^2}\EE_1^n\big(|\llbracket \tau_n'+1,n \rrbracket \backslash \PartB(G_n)| \big)\\
    &=\frac{\tau_n'}{\alpha_n(\Delta_n')^2}\Big(\Delta_n' - \EE_1^n\big(| \PartB(G_n)|\big) \Big).
  \end{align*}
  Hence by Lemma~\ref{lem:expectPartB} below
  \begin{equation*}
    \PP_1^n\Big(|\PartB(G_n)| < \Delta_n'\Big(1 - \frac{\alpha_n\Delta_n'}{\tau_n'} \Big) \Big)%
    \leq \frac{C}{\alpha_n}%
    \begin{cases}
      \log(\tau_n') &\mathrm{if}\ \delta_0 = 0,\\
      1 &\mathrm{if}\ \delta_0 > 0.
    \end{cases}
  \end{equation*}
  Similarly,
  \begin{align*}
    \PP_1^n\Big( \llbracket \tau_n + 1,n\rrbracket \not\subset \PartB(G_n) \Big)%
    &= \PP_1^n\Big( |\llbracket \tau_n+1,n\rrbracket \backslash \PartB(G_n)| \geq  1 \Big)\\
    &\leq \EE_1^n\big(|\llbracket \tau_n+1,n\rrbracket \backslash \PartB(G_n)| \big)\\
    &= \Delta_n - \EE_1^n\big(|\PartB(G_n) \cap\llbracket \tau_n+1,n\rrbracket  | \big).
  \end{align*}
  Hence by Lemma~\ref{lem:expeclastnodesbold} below
  \begin{equation*}
    \PP_1^n\Big( \llbracket \tau_n + 1,n\rrbracket \not\subset \PartB(G_n) \Big)%
    \leq \frac{C\Delta_n\Delta_n'}{\tau_n'}%
    \begin{cases}
      \log(\tau_n') &\mathrm{if}\ \delta_0 = 0,\\
      1 &\mathrm{if}\ \delta_0 > 0.
    \end{cases}
  \end{equation*}

\subsubsection{Computation of expectations of \texorpdfstring{$|\PartB(G_n)|$}{} and \texorpdfstring{$|\PartB(G_n)\cap\llbracket \tau_n+1,n\rrbracket|$}{}}

In this section we derive estimates on the expectations of $|\PartB(G_n)|$ and $|\PartB(G_n)\cap\llbracket \tau_n+1,n\rrbracket |$ which are crucial elements in bounding the probability $\PP_1^n(B_n^c)$.

\begin{lemma}
  \label{lem:expectPartB}
  There exists a constant $B > 0$ depending only on $m$, $\delta_0$ and $\delta_1$, such that for all $2\leq \tau_n' \leq n$
  \begin{equation*}
    \Delta_n' \geq%
    \EE_1^n\big(|\PartB(G_n)| \big)%
    \geq \Delta_n'%
    - \frac{B(\Delta_n')^2}{\tau_n'}%
    \begin{cases}
      (\tau_n')^{-\delta_0/(2m+\delta_0)} &\mathrm{if}\ \delta_0 < 0,\\
      \log(\tau_n') &\mathrm{if}\ \delta_0 = 0,\\
      1 &\mathrm{if}\ \delta_0 > 0.
    \end{cases}
  \end{equation*}
\end{lemma}
\begin{proof}
    First observe that the upper bound is trivial since $\PartB(G_n) \subset \llbracket \tau_n'+1,n\rrbracket$ almost-surely. We now focus on the lower bound.

  Let us write $X_n = |\PartB(G_n)|$ for simplicity. In the whole proof we use the convention that an empty product equals one. Note that
\begin{align*}
  X_n%
  &=
    \sum_{j=\tau_n'+1}^n\1\Big(\D_{G_n}(j) = m,\ \forall k \in \Chi_{G_n}(j),\ k\leq \tau_n'\ \mathrm{and}\ \forall \ell \in \Par_{G_n}(k)\backslash \{j\},\ \ell \leq \tau_n' \Big)\\%
  &=\sum_{\tau_n'< j \leq n}\sum_{\ell=1}^m\sum_{0\leq x_1 < \dots < x_{\ell} \leq \tau_n'}\sum_{\substack{y_1,\dots,y_{\ell}\geq 1\\y_1+\dots+y_{\ell}=m}}%
    \Bigg(\prod_{j< k \leq n}\1(k\not\to_{G_k} j)\Bigg)%
    \Bigg(\prod_{i=1}^{\ell}\1(\mu_{G_j}(j,x_i) = y_i  ) \Bigg)\Bigg(\prod_{\substack{\tau_n' < k \leq n\\k\neq j}}\prod_{i=1}^\ell\1(k \not\to_{G_k} x_i  ) \Bigg).
\end{align*}
Indeed, the previous can be rewritten more conveniently as [here $\bm{x}_{\ell} = (x_1,\dots,x_{\ell})$ and $\bm{y}_{\ell} =(y_1,\dots,y_{\ell})$]
\begin{equation*}
  X_n%
  =\sum_{\tau_n'< j \leq n}\sum_{\ell=1}^m\sum_{0\leq x_1 < \dots < x_{\ell} \leq \tau_n'}%
  \sum_{\substack{y_1,\dots,y_{\ell}\geq 1\\y_1+\dots+y_{\ell}=m}}Y_n^{\bm{x}_{\ell},\bm{y}_{\ell},j}
\end{equation*}
with
\begin{align*}
  Y_n^{\bm{x}_{\ell},\bm{y}_{\ell},j}%
  &=%
    \prod_{\tau_n'< k <
    j}\1( k \not\to_{G_k} \{x_1,\dots,x_{\ell} \})
    \prod_{i=1}^{\ell}\1(\mu_{G_j}(j,x_i) = y_{i})%
    \prod_{j <k \leq n}\1(k\not\to_{G_k} \{x_1,\dots,x_{\ell},j\}).
\end{align*}
Letting $\mathcal{F}_{\ell} = \sigma(G_1,\dots, G_{\ell})$ and $\delta(j) = \delta_0\1(j\leq \tau_n) + \delta_1\1(j>\tau_n)$, it is seen that [assuming $j < n$, otherwise the result is trivial]
\begin{align*}
  \EE_1^n(Y_n^{\bm{x}_{\ell},\bm{y}_{\ell},j} \mid \mathcal{F}_{n-1})%
  &= \prod_{\tau_n'< k <
    j}\1( k \not\to_{G_k} \{x_1,\dots,x_{\ell} \})
    \prod_{i=1}^{\ell}\1(\mu_{G_j}(j,x_i) =y_{\ell} )%
    \prod_{j <k \leq n-1}\1(k\not\to_{G_k} \{x_1,\dots,x_{\ell},j\})\\
  &\qquad%
    \times %
    \prod_{i=1}^m\Bigg(1 - \frac{[\D_{G_{n-1}}(x_1)+\delta(n)]+\dots+[\D_{G_{n-1}}(x_{\ell}) + \delta(n)] + [m+\delta(n)]}{S_{n,i-1}(\delta(n))} \Bigg)\\
  &= \prod_{\tau_n'< k <
    j}\1( k \not\to_{G_k} \{x_1,\dots,x_{\ell} \})
    \prod_{i=1}^{\ell}\1(\mu_{G_j}(j,x_i) = y_{\ell})%
    \prod_{j <k \leq n-1}\1(k\not\to_{G_k} \{x_1,\dots,x_{\ell},j\})\\
  &\qquad%
    \times %
    \prod_{i=1}^m\Bigg(1 - \frac{[\D_{G_{\tau_n'}}(x_1)+y_1+\delta(n)]+\dots+[\D_{G_{\tau_n'}}(x_{\ell})+y_{\ell} + \delta(n)] + [m+\delta(n)]}{S_{n,i-1}(\delta(n))} \Bigg)
\end{align*}
where the second line follows because if the product of indicators is non-zero, then at instant $n-1$ no vertex other than $j$ has connected to one of the $x_1,\dots,x_{\ell}$ on the time interval $\llbracket \tau_n'+1,n-1\rrbracket$, and $j$ has edge multiplicity $y_{\ell}$ with $x_{\ell}$. Defining for simplicity $D_{\tau_n'}^{\bm{x}_{\ell}} = \sum_{i=1}^{\ell}\D_{\G_{\tau_n'}}(x_i) $, and taking conditional expectation of the previous inductively with respect to $\mathcal{F}_{n-2},\dots,\mathcal{F}_{\tau_n'}$, it is found that [here the combinatorial factor comes from enumerating all the possibilities of connecting $j$ to $x_1,\dots,x_{\ell}$ with edges multiplicities $y_1,\dots,y_{\ell}$]
\begin{multline*}
  \EE_1^n(Y_n^{\bm{x}_{\ell},\bm{y}_{\ell},j}\mid \mathcal{F}_{\tau_n'})%
  = \frac{m!}{\prod_{i=1}^{\ell}y_i!}\prod_{\tau_n' < k < j}\prod_{i=1}^m\Bigg(1 - \frac{D_{\tau_n'}^{\bm{x}_{\ell}} + \ell\delta(k)}{S_{k,i-1}(\delta(k))} \Bigg)\\%
  \times%
  \frac{\prod_{i=1}^{\ell}\prod_{i'=1}^{y_i}\big(\D_{G_{\tau_n'}}(x_i)+ i'-1 + \delta(j)\big)}{\prod_{i=1}^mS_{j,i-1}(\delta(j))}%
  \prod_{j < k \leq n}\prod_{i=1}^m\Bigg(1 - \frac{D_{\tau_n'}^{\bm{x}_{\ell}} + 2m + (\ell+1)\delta(k)}{S_{k,i-1}(\delta(k))} \Bigg).
\end{multline*}
Hence,
\begin{align*}
  \EE_1^n(Y_n^{\bm{x}_{\ell},\bm{y}_{\ell},j}\mid \mathcal{F}_{\tau_n'})
  &\geq \frac{m!}{\prod_{i=1}^{\ell}y_i!}%
    \frac{\prod_{i=1}^{\ell}\prod_{i'=1}^{y_i}\big(\D_{G_{\tau_n'}}(x_i)+ i'-1 + \delta(j)\big)}{\prod_{i=1}^mS_{j,i-1}(\delta(j))}%
    \prod_{\tau_n' < k \leq n}\prod_{i=1}^m\Bigg(1 - \frac{D_{\tau_n'}^{\bm{x}_{\ell}} + 2m + (\ell + 1) \delta(k)}{S_{k,i-1}(\delta(k))}\Bigg)\\
  &\geq \frac{m!}{\prod_{i=1}^{\ell}y_i!}%
    \frac{\prod_{i=1}^{\ell}\prod_{i'=1}^{y_i}\big(\D_{G_{\tau_n'}}(x_i)+ i'-1 + \delta(j)\big)}{\prod_{i=1}^mS_{j,i-1}(\delta(j))}%
    \Bigg(1 - \sum_{\tau_n' < k \leq n}\sum_{i=1}^m\frac{D_{\tau_n'}^{\bm{x}_{\ell}} + 2m + (\ell + 1) \delta(k)}{S_{k,i-1}(\delta(k))}\Bigg)\\
  &\geq \frac{m!}{\prod_{i=1}^{\ell}y_i!}%
    \frac{\prod_{i=1}^{\ell}\prod_{i'=1}^{y_i}\big(\D_{G_{\tau_n'}}(x_i)+ i'-1 + \delta(j)\big)}{\prod_{i=1}^mS_{j,i-1}(\delta(j))}%
    \Bigg(1 - m\sum_{\tau_n' < k \leq n}\frac{D_{\tau_n'}^{\bm{x}_{\ell}} + 2m + (m + 1) (\delta(k)\vee 0)}{S_{k,0}(\delta(k))} \Bigg).
\end{align*}
We now define two random variables
\begin{equation*}
  \tilde{X}_n%
  = \sum_{\tau_n'< j \leq n}\sum_{\ell=1}^m\sum_{0\leq x_1 < \dots < x_{\ell} \leq \tau_n'}%
    \sum_{\substack{y_1,\dots,y_{\ell}\geq 1\\y_1+\dots+y_{\ell}=m}}%
  \frac{m!}{\prod_{i=1}^{\ell}y_i!}%
    \frac{\prod_{i=1}^{\ell}\prod_{i'=1}^{y_i}\big(\D_{G_{\tau_n'}}(x_i)+ i'-1 + \delta(j)\big)}{\prod_{i=1}^mS_{j,i-1}(\delta(j))}
\end{equation*}
and
\begin{multline*}
  R_n%
  = m \sum_{\tau_n'< j,k \leq n}\sum_{\ell=1}^m\sum_{0\leq x_1 < \dots < x_{\ell} \leq \tau_n'}%
  \sum_{\substack{y_1,\dots,y_{\ell}\geq 1\\y_1+\dots+y_{\ell}=m}}\frac{m!}{\prod_{i=1}^{\ell}y_i!}\\%
  \times
  \sum_{\substack{y_1,\dots,y_{\ell}\geq 1\\y_1+\dots+y_{\ell}=m}}\frac{m!}{\prod_{i=1}^{\ell}y_i!}%
    \frac{\prod_{i=1}^{\ell}\prod_{i'=1}^{y_i}\big(\D_{G_{\tau_n'}}(x_i)+ i'-1 + \delta(j)\big)}{\prod_{i=1}^mS_{j,i-1}(\delta(j))}\frac{D_{\tau_n'}^{\bm{x}_{\ell}} + 2m + (m+1)\left(\delta(k)\vee 0\right)}{S_{k,0}(\delta(k))}
\end{multline*}
so that $\EE_1^n(X_n \mid \mathcal{F}_{\tau_n'}) \geq \tilde{X}_n - R_n$ almost-surely. To compute the expectations of $\tilde{X}_n$ and $R_n$, we use the following trick. For a fixed $j>\tau_n'$ we define on the same probability space a sequence of random graphs $((\tilde{G}_{t,i}^j)_{i=0}^m)_{t\geq 1}$ such that $\tilde{G}_{t,i}^j = G_{t,i}$ for $1\leq t \leq \tau_n'$ and $0\leq i \leq m$, and then $(\tilde{G}_{t,i}^j)_{i=0}^m$ evolves independently of $(G_{t,i})_{i=0}^m$ according to the preferential attachment rule with parameter $\delta(t) = \delta(j)$ for all $t> \tau_n'$. Then, we see that
\begin{align*}
  1%
  &= \EE_1^n\Bigg(\sum_{\ell=1}^m\sum_{0\leq x_1 < \dots < x_{\ell}\leq \tau_n'}\sum_{\substack{y_1,\dots,y_{\ell}\geq 1\\y_1+\dots+y_{\ell}=m}}\1\Big(\forall i=1,\dots,\ell,\ \mu_{\tilde{G}_{\tau_n'+1}^j}(\tau_n'+1,x_i) = y_i \Big) \ \Big\vert\ \mathcal{F}_{\tau_n'} \Bigg)\\
  &= \sum_{\ell=1}^m\sum_{0\leq x_1 < \dots < x_{\ell}\leq \tau_n'}\sum_{\substack{y_1,\dots,y_{\ell}\geq 1\\y_1+\dots+y_{\ell}=m}} \frac{m!}{\prod_{i=1}^{\ell}y_i!}\frac{\prod_{i=1}^{\ell}\prod_{i'=1}^{y_i}\big(\D_{G_{\tau_n'}}(x_i) + i'-1 + \delta(j) \big)}{\prod_{i=1}^mS_{\tau_n'+1,i-1}(\delta(j))}.
\end{align*}
So indeed,
\begin{align*}
  \EE_1^n(\tilde{X}_n)%
  &=\sum_{\tau_n' < j \leq n}\prod_{i=1}^m\frac{S_{\tau_n'+1,i-1}(\delta(j))}{S_{j,i-1}(\delta(j))}\\
  &=\sum_{\tau_n' < j \leq n}\prod_{i=1}^m\frac{(2m+\delta(j))(\tau_n'+1)-2m+i-1}{(2m+\delta(j))j-2m+i-1}\\
  &=\sum_{\tau_n' < j \leq n}\prod_{i=1}^m\Big(1 - \frac{(2m+\delta(j))(j-\tau_n'-1)}{(2m+\delta(j))j-2m+i-1}\Big)\\
  &\geq \sum_{\tau_n' < j \leq n}\Big(1 - \frac{(2m+\delta(j))(j-\tau_n'-1)}{(2m+\delta(j))j-2m}\Big)^m\\
  &\geq \Delta_n'\Big(1 - \frac{m\Delta_n'}{n-2} \Big).
\end{align*}
Similarly,
\begin{multline*}
  \EE_1^n\Bigg(\sum_{v\in \Chi_{\tilde{G}_{\tau_n'+1}^j}(\tau_n'+1)}\D_{\tilde{G}_{\tau_n'}^j}(v) \ \Big\vert\  \mathcal{F}_{\tau_n'} \Bigg)\\
  = \sum_{\ell=1}^m\sum_{0\leq x_1 < \dots < x_{\ell}\leq \tau_n'}\sum_{\substack{y_1,\dots,y_{\ell}\geq 1\\y_1+\dots+y_{\ell}=m}} \frac{m!}{\prod_{i=1}^{\ell}y_i!}\frac{\prod_{i=1}^{\ell}\prod_{i'=1}^{y_i}\big(\D_{G_{\tau_n'}}(x_i) + i'-1 + \delta(j) \big)}{\prod_{i=1}^mS_{\tau_n'+1,i-1}(\delta(j))}D_{\tau_n'}^{\bm{x}_{\ell}}
\end{multline*}
from which we deduce that
\begin{align*}
  \EE_1^n(R_n)%
  &= m\sum_{\tau_n'< j,k\leq n }\frac{1}{S_{k,0}(\delta(k))}\Bigg(
    \EE_1^n\Bigg(\sum_{v\in \Chi_{\tilde{G}_{\tau_n'+1}^j}(\tau_n'+1)}\D_{\tilde{G}_{\tau_n'}^j}(v) \Bigg) + 2m + (m+1) \delta(k)\vee 0 \Bigg)\prod_{i=1}^m\frac{S_{\tau_n'+1,i-1}(\delta(j))}{S_{j,i-1}(\delta(j))}.
\end{align*}
But
\begin{align*}
  \EE_1^n\Bigg(\sum_{v\in \Chi_{\tilde{G}_{\tau_n'+1}^j}(\tau_n'+1)}\D_{\tilde{G}_{\tau_n'}^j}(v) \Bigg)%
  &\leq \sum_{i=1}^m\EE_1^n\big(\D_{\tilde{G}_{\tau_n'+1,i-1}^j}(\tilde{V}_{t,i}) \big)\\%
  &=\sum_{i=1}^m\EE_1^n\Bigg(\sum_{v=0}^{\tau_n'}\D_{\tilde{G}_{\tau_n'+1,i-1}^j}(v)%
    \frac{\D_{\tilde{G}_{\tau_n'+1,i-1}^j}(v) + \delta(j)}{S_{\tau_n'+1,i-1}(\delta(j))} \Bigg)\\
  &= \sum_{i=1}^m\sum_{v=0}^{\tau_n'}\frac{\EE_1^n\big(\D_{\tilde{G}_{\tau_n'+1,i-1}^j}(v)^2 \big) }{S_{\tau_n'+1,i-1}(\delta(j))}%
    +\delta(j) \sum_{i=1}^m\frac{2m\tau_n' + i-1}{S_{\tau_n'+1,i-1}(\delta(j))}\\
  &\leq 2m\sum_{v=0}^{\tau_n'}\frac{\EE_1^n\big((\D_{G_{\tau_n'}}(v) + \delta_0)^2 \big) }{S_{\tau_n'+1,0}(\delta(j))}%
    +\frac{2m(m\delta(j)\vee 0 + (m-\delta_0)^2)(\tau_n'+1)}{S_{\tau_n'+1,0}(\delta(j))}
\end{align*}
where we have used that only $m$ edges can be added between $\tau_n'$ and $\tau_{n}'+1$, so the difference between the degree of $v$ in $\tilde{G}_{\tau_n'+1}^j$ and its degree in $G_{\tau_n'}$ cannot exceed $m$. Remarking that in time interval $\llbracket 0,\tau_n'\rrbracket$ the process $(\tilde{G}_t)_{t\geq 1}$ evolves according to the preferential attachment rule with parameter $\delta_0$, and remarking that $\frac{\tau_n'+1}{S_{\tau_n'+1,0}(\delta(j))}$ is bounded by a constant, it follows letting $\underline{\delta} = \delta_0\wedge \delta_1$
\begin{align*}
  \EE_1^n(R_n)%
  &\leq \frac{\EE_1^n(\tilde{X}_n)\Delta_n'}{S_{\tau_n'+1,0}(\underline{\delta})}%
    \Bigg(C + \frac{2m^2}{S_{\tau_n'+1,0}(\underline{\delta})}\sum_{v=0}^{\tau_n'}\EE_0^n\big((\D_{G_{\tau_n'}}(v) + \delta_0)^2 \big) \Bigg)
\end{align*}
for a constant $C> 0$ depending only on $\delta_0$, $\delta_1$, and $m$. By Lemmas~\ref{lem:expec2momentdeg} and~\ref{lem:estimatescoeffdegree}, there are constants $C,C' > 0$ depending solely on $m$ and $\delta_0$ such that
\begin{align*}
  \sum_{v=0}^{\tau_n'}\EE_0^n\big((\D_{G_{\tau_n'}}(v) + \delta_0)^2 \big)%
  &\leq C'\sum_{v=0}^{\tau_n'}\Big(\frac{\tau_n'}{1\vee v} \Big)^{2m/(2m+\delta_0)}\\
  &\leq C'(\tau_n')^{2m/(2m+\delta_0)}\Bigg(4%
    + \int_1^{\tau_n'} \frac{1}{x^{2m/(2m+\delta_0)}}\intd x\Bigg)\\
  &\leq C'' %
    \begin{cases}
      (\tau_n')^{2m/(2m+\delta_0)} &\mathrm{if}\ \delta_0 < 0,\\
      \tau_n'\log(\tau_n') &\mathrm{if}\ \delta_0 = 0,\\
      \tau_n' &\mathrm{if}\ \delta_0 > 0.
    \end{cases}
\end{align*}
The conclusion follows because $S_{\tau_n'+1,0}(\underline{\delta}) = (2m+\underline{\delta})(\tau_n'+1) - 2m \geq m\tau_n' + \underline{\delta}$.
\end{proof}

  \begin{lemma}
    \label{lem:expeclastnodesbold}
    There exists a constant $B > 0$ depending only on $m$, $\delta_0$ and $\delta_1$, such that for all $2\leq \tau_n' < \tau_n \leq n$
  \begin{equation*}
    \Delta_n \geq%
    \EE_1^n\big(|\PartB(G_n)\cap \llbracket \tau_n+1,n\rrbracket| \big)%
    \geq \Delta_n%
    - \frac{B\Delta_n\Delta_n'}{\tau_n'}%
    \begin{cases}
      (\tau_n')^{-\delta_0/(2m+\delta_0)} &\mathrm{if}\ \delta_0 < 0,\\
      \log(\tau_n') &\mathrm{if}\ \delta_0 = 0,\\
      1 &\mathrm{if}\ \delta_0 > 0.
    \end{cases}
  \end{equation*}
\end{lemma}
\begin{proof}
  The lemma follows by remarking that $|\PartB(G_n)\cap \llbracket \tau_n+1,n\rrbracket|$ can be rewritten as
  \begin{equation*}
     \sum_{j=\tau_n+1}^n\1\Big(\D_{G_n}(j) = m,\ \forall k \in \Chi_{G_n}(j),\ k\leq \tau_n'\ \mathrm{and}\ \forall \ell \in \Par_{G_n}(k)\backslash \{j\},\ \ell \leq \tau_n' \Big).
  \end{equation*}
  Then the rest of the proof is identical to Lemma~\ref{lem:expectPartB} \textit{mutatis mutandis}.
\end{proof}

\subsubsection{Auxiliary results used to prove the Proposition~\ref{pro:eventBn}}

  \begin{lemma}
    \label{lem:expec:degree}
    Let $\gamma_t = \prod_{i=1}^m\big(1 + \frac{1}{S_{t,i-1}(\delta_0)} \big)$. For every $0\leq u < t \leq n$
    \begin{align*}
      \EE_0^n(\D_{G_t}(u) + \delta_0)%
      &= \gamma_t\EE_0^n( \D_{G_{t-1}}(u) + \delta_0 ).
  \end{align*}
\end{lemma}
\begin{proof}
  These are standard computations, see for instance \cite[Chapter~8]{hofstad_2016}.
\end{proof}

  \begin{lemma}
    \label{lem:expec2momentdeg}
  For every $2 \leq t \leq n$ and $0 \leq u \leq  t$
  \begin{equation*}
    \EE_0^n[(\D_{G_t}(u) + \delta_0)^2 ]%
    = \xi_{1\vee u}^t(m + \delta_0)^2%
    + \kappa_{1\vee u}^t(m+\delta_0)
  \end{equation*}
  where for all $r = 1,\dots,t$:
  \begin{equation*}
    \xi_r^t = \prod_{r+1\leq j \leq t} \prod_{i=1}^m\Big(1 + \frac{2}{S_{j,i-1}(\delta_0)}\Big)
  \end{equation*}
  and
  \begin{multline*}
    \kappa_r^t%
    =  \sum_{r+1\leq j \leq t}\Bigg(\prod_{j+1\leq p \leq t}\prod_{i=1}^m\Big(1 + \frac{2}{S_{p,i-1}(\delta_0)}\Big) \Bigg)\Bigg(\prod_{r+1\leq p \leq j-1}\prod_{i=1}^m\Big(1 + \frac{1}{S_{p,i-1}(\delta_0)}\Big)\Bigg)\\
    \times%
    \Bigg(\sum_{k=1}^m\frac{1}{S_{j,k-1}(\delta_0)}\prod_{1\leq i \leq k-1}\Big(1 + \frac{1}{S_{j,i-1}(\delta_0)}\Big)\prod_{k+1\leq i \leq m}\Big(1 + \frac{2}{S_{j,i-1}(\delta_0)}\Big) \Bigg).
  \end{multline*}
\end{lemma}
\begin{proof}
  Let $\mathcal{F}_t = \sigma(G_1,\dots,G_t)$.  We first compute $\EE_0^n[(\D_{G_{t}}(u) + \delta_0)^2 \mid \mathcal{F}_t] = \EE_0^n[(\D_{G_{t,m}}(u) + \delta_0)^2 \mid \mathcal{F}_{t-1}]$. We define the coefficients $(\alpha_{t,i})_{i=1}^m$ and $(\beta_{t,i})_{i=1}^{m}$ such that $\alpha_{t,m} = 1$ and $\beta_{t,m} = 0$, and satisfying the recurrence for $i=m,\dots,1$
  \begin{equation*}
    \alpha_{t,i-1}
    = \alpha_{t,i}\Big(1 + \frac{2}{S_{t,i-1}(\delta_0)}\Big),\qquad\qquad%
    \beta_{t,i-1},%
    = \beta_{t,i}\Big(1 + \frac{1}{S_{t,i-1}(\delta_0)}\Big)%
      + \frac{\alpha_{t,i}}{S_{t,i-1}(\delta_0)}.
  \end{equation*}
  It is seen that for every $r=1,\dots,m$ (using the convention that empty product equals one and empty sum equals zero):
  \begin{align*}
    \alpha_{t,r}%
    &= \prod_{r+1\leq j\leq m}\Big(1 + \frac{2}{S_{t,j-1}(\delta_0)}\Big)\\
    \beta_{t,r}%
    \notag
    &= \sum_{r+1\leq k\leq m}\frac{\alpha_{t,k}}{S_{t,k-1}(\delta_0)}\prod_{r+1\leq j \leq k-1}\Big(1 + \frac{1}{S_{t,j-1}(\delta_0)}\Big)\\%
    &= \sum_{r+1\leq k\leq m}\frac{1}{S_{t,k-1}(\delta_0)}\prod_{r+1\leq j \leq k-1}\Big(1 + \frac{1}{S_{t,j-1}(\delta_0)}\Big)\prod_{k+1\leq j \leq m}\Big(1 + \frac{2}{S_{t,j-1}(\delta_0)}\Big).
  \end{align*}
  Then we consider the random variable
  \begin{equation*}
    M_{t,i} = \alpha_{t,i}\big(\D_{G_{t,i}}(u) + \delta \big)^2 + \beta_{t,i}\big(\D_{G_{t,i}}(u) + \delta \big).
  \end{equation*}
  We claim that $(M_{t,i})_{i=1}^m$ is a martingale with respect to $(\mathcal{F}_{t,i})_{i=1}^m$, where $\mathcal{F}_{t,i} = \sigma(G_{t,0},\dots,G_{t,i})$; \textit{ie}. we claim that $\EE(M_{t,i} \mid \mathcal{F}_{t,i-1}) = M_{t,i-1}$ for $i=1,\dots,m$. Indeed for $i=1,\dots,m$
  \begin{align*}
    \EE_0^n\big( (\D_{G_{t,i}}(u) + \delta_0)^2 \mid \mathcal{F}_{t,i-1}  \big)%
    &= (\D_{G_{t,i-1}}(u) + 1 + \delta_0)^2 \frac{\D_{G_{t,i-1}}(u) + \delta_0}{S_{t,i-1}(\delta_0)}%
      + (\D_{G_{t,i-1}}(u) + \delta_0)^2\Big(1 - \frac{\D_{G_{t,i-1}}(u) + \delta_0}{S_{t,i-1}(\delta_0)}\Big)\\
    &= (\D_{G_{t,i-1}}(u) + \delta_0)^2\Big(1 + \frac{2}{S_{t,i-1}(\delta_0)}\Big)%
      + \frac{\D_{G_{t,i-1}}(u) + \delta_0 }{S_{t,i-1}(\delta_0)}
  \end{align*}
  and
  \begin{align*}
    \EE_0^n\big( \D_{G_{t,i}}(u) + \delta_0 \mid \mathcal{F}_{t,i-1}  \big)%
    &= (\D_{G_{t,i-1}}(u) + 1 + \delta_0) \frac{\D_{G_{t,i-1}}(u) + \delta_0}{S_{t,i-1}(\delta_0)}%
      + (\D_{G_{t,i-1}}(u) + \delta_0)\Big(1 - \frac{\D_{G_{t,i-1}}(u) + \delta_0}{S_{t,i-1}(\delta_0)}\Big)\\
    &= (\D_{G_{t,i-1}}(u) + \delta_0 )\Big(1 + \frac{1}{S_{t,i-1}(\delta_0)}\Big)
  \end{align*}
  so that
  \begin{align*}
    \EE_0^n(M_{t,i}\mid \mathcal{F}_{t,i-1})%
    &= \alpha_{t,i}\EE_0^n\big( (\D_{G_{t,i}}(u) + \delta_0)^2 \mid \mathcal{F}_{t,i-1}  \big)%
      + \beta_{t,i}\EE_0^n\big( (\D_{G_{t,i}}(u) + \delta_0) \mathcal{F}_{t,i-1} \big)\\
    &= \alpha_{t,i}\Big(1 + \frac{2}{S_{t,i-1}(\delta_0)}\Big)(\D_{G_{t,i-1}}(u) + \delta_0)^2%
      + \Bigg(\frac{\alpha_{t,i}}{S_{t,i-1}(\delta_0)} + \beta_{t,i}\Big(1 + \frac{1}{S_{t,i-1}(\delta_0)}\Big) \Bigg)(\D_{G_{t,i-1}}(u) + \delta_0)\\
    &= \alpha_{t,i-1}(\D_{G_{t,i-1}}(u) + \delta_0)^2%
      + \beta_{t,i-1}(\D_{G_{t,i-1}}(u) + \delta_0)\\
    &= M_{t,i-1}.
  \end{align*}
  Then,
  \begin{align*}
    \EE_0^n[(\D_{G_{t,m}}(u)+\delta)^2 \mid \mathcal{F}_{t-1} ]%
    &= \EE_1^n[M_{t,m} \mid \mathcal{F}_{t-1} ]\\%
    &= M_{t,0}\\
    &= \alpha_{t,0}\big(\D_{G_{t},0}(u) + \delta \big)^2%
      + \beta_{t,0}\big(\D_{G_{t,0}}(u) + \delta \big)\\
    &=\alpha_{t,0}\big(\D_{G_{t-1}}(u) + \delta \big)^2%
      + \beta_{t,0}\big(\D_{G_{t-1}}(u) + \delta \big).
  \end{align*}
  Next, let $(\xi_j^t)_{j=1}^{t}$ and $(\kappa_j^t)_{j=1}^t$ as in the statement of the lemma, and $\gamma_{j} = \prod_{i=1}^m\big(1 + \frac{1}{S_{j,i-1}(\delta_0)}\big)$ (as in Lemma~\ref{lem:expec:degree}). It is straightforward to show that $(\xi_j^t)_{j=1}^t$ and $(\kappa_j^t)_{j=1}^t$ satisfy $\xi_t^t=1$ and $\kappa_t^t = 0$ and the recurrence 
  \begin{equation*}
    \xi_{j-1}^t = \xi_{j}^t\alpha_{j,0},\qquad\qquad%
    \kappa_{j-1}^t = \xi_{j}^t\beta_{j,0} + \kappa_{j}^t\gamma_{j}.
  \end{equation*}
  Indeed, for $r=1,\dots,t$
  \begin{align*}
    \xi_r^t%
    &= \prod_{r+1\leq j\leq t}\alpha_{j,0}%
      = \prod_{r+1\leq j \leq t} \prod_{i=1}^m\Big(1 + \frac{2}{S_{j,i-1}(\delta_0)}\Big)\\
    \kappa_r^t%
    &= \sum_{r+1\leq j \leq t}\xi_j^t\beta_{j,0}\prod_{r+1\leq k \leq j-1}\gamma_k%
  \end{align*}
  which are equal to the expression given in the statement of the lemma. Let now define for $j=1\vee u,\dots,t$
  \begin{equation*}
    M_j' = \xi_j^t\big( \D_{G_j}(u) + \delta_0\big)^2%
    + \kappa_j^t\big(\D_{G_j}(u) + \delta_0\big).
  \end{equation*}
  The claim is that $(M_j')_{j\geq 1}$ is a martingale with respect to $(\mathcal{F}_j)_{j=1\vee u}^t$. Indeed, using Lemma~\ref{lem:expec:degree} and the above computations
  \begin{align*}
    \EE_0^n(M_j' \mid \mathcal{F}_{j-1})%
    &= \xi_j^t \EE_0^n\big( (\D_{G_j}(u) + \delta_0)^2 \mid \mathcal{F}_{j-1}\big)%
      + \kappa_j^t \EE_0^n\big(\D_{G_j}(u) + \delta_0 \mid \mathcal{F}_{j-1}\big)\\
    &= \xi_j^t\Big(\alpha_{j,0}(\D_{G_{j-1}}(u) + \delta_0)^2 + \beta_{j,0}(\D_{G_{j-1}}(u) + \delta_0)  \Big)%
      + \kappa_j^t \gamma_j (\D_{G_{j-1}}(u) + \delta_0)\\
    &= \xi_j^t\alpha_{j,0}(\D_{G_{j-1}}(u) + \delta_0)^2%
      + \big(\xi_j^t\beta_{j,0} + \kappa_j^t\gamma_j \big)(\D_{G_{j-1}}(u) + \delta_0)\\
    &= \xi_{j-1}^t(\D_{G_{j-1}}(u) + \delta_0)^2%
      + \kappa_{j-1}^t(\D_{G_{j-1}}(u) + \delta_0)\\
    &= M_{j-1}'.
  \end{align*}
  This implies that (because $\D_{G_{1\vee u}}(u) = m$ almost-surely)
  \begin{equation*}
    \EE_0^n\big((\D_{G_t}(u) + \delta)^2 \big)%
    = \EE_0^n(M_t')
    = \EE_0^n(M_{1\vee u}')
    = \xi_{1\vee u}(m + \delta)^2 + \kappa_{1\vee u}(m + \delta).\qedhere
  \end{equation*}
\end{proof}

  \begin{lemma}
    \label{lem:estimatescoeffdegree}
    Let $\xi_{1\vee u}^t$ and $\kappa_{1\vee u}^t$ as in the statement of Lemma~\ref{lem:expec2momentdeg}. There exists a constant $B > 0$ depending only on $m$ and $\delta_0$ such that for all $0\leq u < t$ such that
    \begin{equation*}
      \max(\xi_{1\vee u}^t,\, \kappa_{1\vee u}^t )%
      \leq B\Big(\frac{t}{1\vee u} \Big)^{2m/(2m+\delta_0)}.
    \end{equation*}
  \end{lemma}
  \begin{proof}
    Let define the function $g : \mathbb{N} \to \mathbb{R}_+$ as $g(1) = 0$ and $g(n) = \sum_{j=2}^n\sum_{i=1}^m\frac{1}{S_{j,i-1}(\delta_0)}$ for $n\geq 2$. Observe that for any $n\geq 2$
    \begin{align*}
      g(n)%
      &= \sum_{j=2}^n\sum_{i=1}^m\frac{1}{(2m+\delta_0)j - 2m+i-1}\\
      &= \frac{1}{2m+\delta_0}\sum_{j=2}^n\sum_{i=1}^m\frac{1}{j + \frac{-2m + i-1}{2m+\delta_{0}}}\\
      &= \frac{m}{2m + \delta_0}\sum_{j=2}^n\frac{1}{j}%
        + \frac{1}{2m+\delta_0}\sum_{j=1}^n\sum_{i=1}^m\Big(\frac{1}{j + \frac{-2m + i-1}{2m+\delta_{0}}} - \frac{1}{j} \Big)\\
      &= \frac{m}{2m + \delta_0}\sum_{j=2}^n\frac{1}{j}%
        + \frac{1}{(2m+\delta_0)^2}\sum_{j=1}^n\sum_{i=1}^m\frac{-2m + i-1}{j\big(j+ \frac{-2m + i-1}{2m+\delta_{0}}\big)}.
    \end{align*}
    Thus letting $H(n) = \sum_{j=1}^n\frac{1}{j}$ denote the $j$-th harmonic number, we deduce that there is a constant $C > 0$ depending only on $m$ and $\delta_0$ such that for all $n\geq 2$:
    \begin{equation}
      \label{eq:harmonix}
      \frac{m}{2m+\delta_0}\big(\gamma + \log(n)\big) - C
      \leq 
      \frac{m}{2m+\delta_0}H(n) - C
      \leq g(n) \leq \frac{m}{2m+\delta_0}\big(H(n) -1\big)
      \leq \frac{m}{2m+\delta_0}\big(\gamma + \log(n)\big)
    \end{equation}
    with $\gamma$ the Euler constant, using well known bounds on the harmonic numbers.  It follows from \eqref{eq:harmonix} that
    \begin{align*}
      \xi_{1\vee u}^t%
      &\leq \exp\big(2g(t) - 2g(1\vee u) \big)\\
      &\leq \exp\Big(\frac{2m}{2m+\delta_0}\log\Big(\frac{t}{1\vee u} \Big) + C \Big).
    \end{align*}
    Next, since $\max_{2\leq j \leq t}\max_{1\leq \leq m}\prod_{1\leq i \leq k-1}\big(1 + \frac{1}{S_{j,i-1}(\delta_0)}\big)\prod_{k+1\leq i \leq m}\big(1 + \frac{2}{S_{j,i-1}(\delta_0)}\big)$ is finite, we find that for some constants $C',C'',C''' > 0$ depending only on $m$ and $\delta_0$
    \begin{align*}
      \kappa_{1\vee u}^t%
      &\leq C'\sum_{1\vee u+1\leq j \leq t}\frac{1}{S_{j,0}(\delta_0)}e^{2g(t) - 2g(j)}e^{g(j-1) - g(1\vee u)}\\
      &\leq C'' \sum_{1\vee u +1 \leq j \leq t}\frac{1}{j}\Big(\frac{t}{j} \Big)^{2m/(2m+\delta_0)}\Big(\frac{j}{1\vee u}\Big)^{m/(2m+\delta_0)}\\
      &\leq C'' \frac{t^{2m/(2m+\delta_0)}}{(1\vee u)^{m/(2m+\delta_0)}}\int_{1\vee u}^t \frac{1}{x^{1+m/(2m+\delta_0)}}\intd x\\
      &\leq C'''\Big(\frac{t}{1\vee u} \Big)^{2m/(2m+\delta_0)}.
    \end{align*}
    This concludes the proof.
  \end{proof}


\section*{Acknowledgements}

{\'E}lisabeth Gassiat is supported by \textit{Institut Universitaire de France}. All the authors are supported by the \textit{Agence Nationale de la Recherche} under projects ANR-21-CE23-0035-02 and ANR-23-CE40-0018-02.

  \putbib%
\end{bibunit}

\newpage

\setcounter{page}{1}
\setcounter{section}{0}
\setcounter{table}{0}
\setcounter{figure}{0}
\setcounter{theorem}{0}
\setcounter{equation}{0}

\renewcommand{\thepage}{S\arabic{page}}
\renewcommand{\thesection}{S\arabic{section}}
\renewcommand{\thetable}{S\arabic{table}}
\renewcommand{\thefigure}{S\arabic{figure}}
\renewcommand{\thetheorem}{S\arabic{section}.\arabic{theorem}}
\renewcommand{\theproposition}{S\arabic{section}.\arabic{proposition}}
\renewcommand{\thelemma}{S\arabic{section}.\arabic{lemma}}
\renewcommand{\theequation}{S\arabic{section}.\arabic{equation}}


\begin{bibunit}[plain]

\title{On the impossibility of detecting a late change-point in the preferential attachment random graph model: supplementary material}


\maketitle

\section{Organization}
\label{sec:supp:introduction}

This document is the supplementary material for the article \textit{On the impossibility of detecting a late change-point in the preferential attachment random graph model} \cite{kng:main}. It contains the statements of additional results in the case where the observation is the labeled graph (Section~\ref{sec:labeled:additional}), as well as their proofs.
We refer to the main document for all the definitions and notations. Supplementary notations used only in this document are given in Section~\ref{sec:sm:notations}.

Every section, subsection, theorem, etc. of the supplemental has label prefixed
by S and is cited with prefix. References to the main document are cited
with no prefix.

\section{Additional results when the observation is the labeled graph}
\label{sec:labeled:additional}

\subsection{Maximum Likelihood Estimation of \texorpdfstring{$(\delta_0,\delta_1)$}{}}
\label{sec:maxim-likel-estim}

In \cite{GV17}, the estimation of $\delta_0$ was done under the null hypothesis. We now investigate the estimation of $\delta_0$ and $\delta_1$ in the model where there is a change-point from $\delta_0$ to $\delta_1$ at instant $\tau_n = n -\Delta_n$. Here $\tau_n$ is assumed to be known. As shown in the expression of the likelihood in Lemma 
A.3, the likelihood factorizes in two parts each of those involving only $\delta_0$ or $\delta_1$; \textit{ie}. letting (here $G_{\tau_n} = G_n[\llbracket 0,\tau_n\rrbracket]$)
\begin{equation*}
    \begin{split}
        \ell_{1:\tau_n}(\delta_0) &= \log\left(\frac{\prod_{k=m}^{\tau_n m}\left(k+\delta\right)^{N_{>k}(G_{\tau_n})}}{\prod_{t=2}^{\tau_n}\prod_{i=1}^{m}\left[\left(2m+\delta_0\right)t-2m + i -1\right]}\right)\\
        \ell_{\tau_n + 1:n}(\delta_1) &= \log\left(\frac{\prod_{k=m}^{nm}\left(k+\delta_1\right)^{N_{>k}(G_{n})-N_{>k}(G_{\tau_n})}}{\prod_{t=\tau_n +1}^{n}\prod_{i=1}^{m}\left[\left(2m+\delta_1\right)t-2m + i-1\right]}\right)
    \end{split}
\end{equation*}
the log-likelihood of $(\delta_0,\delta_1)$ writes as $\ell_{1:\tau_n}(\delta_0) + \ell_{\tau_n + 1:n}(\delta_1)$. Then, building on the work of \cite{GV17} in the no change-point model, we obtain in the next theorem the asymptotic normality of the MLE in the model with a change-point.

As it will be useful in the next, we recall the expression of the limiting degree distribution of the affine preferential attachment model with parameter $\delta$ (see \cite[Sections~8.6.1 and~8.6.2]{hofstad_2016} for details):
\begin{equation}
  \label{eq:limitdegreeseq}
  p_k(\delta)%
  =%
  (2+\delta/m)\frac{\Gamma(k +\delta)\Gamma(m+2+\delta+\delta/m)}{\Gamma(m+\delta)\Gamma(k+3+\delta+\delta/m)}.
\end{equation}

\begin{theorem}
  \label{thm:mle:normality}
    For all $(\delta_0,\delta_1) \in (-m,\infty)^2$, if $\tau_n \to \infty$ and $\Delta_n \to \infty$, then $(\delta,\delta') \mapsto \ell_{1:\tau_n}(\delta) + \ell_{\tau_n + 1:n}(\delta')$ has a unique maximizer $(\hat\delta_{0,n},\hat\delta_{1,n})$ with probability going to one under $(\PP_1^n)_{n\geq 1}$, and
    \begin{equation*}
      \begin{pmatrix}
        \sqrt{\tau_n} & 0\\
        0 &\sqrt{\Delta_n}\\
      \end{pmatrix}
      \begin{pmatrix}
        \hat\delta_{0,n} - \delta_0\\
        \hat\delta_{1,n} - \delta_1\\
      \end{pmatrix}
      \overset{\PP_1^n}{\rightsquigarrow} \mathcal{N}\Bigg(0,%
      \begin{pmatrix}
        \nu_0 & 0\\
        0 & \nu_1
      \end{pmatrix}^{-1}
      \Bigg)
    \end{equation*}
    where $\rightsquigarrow$ stands for convergence in distribution under $(\PP_1^n)_{n\geq 1}$ and where for $j=0,1$
    \begin{equation*}
    \nu_j = \frac{m}{2m+\delta_j}\Bigg(\sum_{k=m}^{\infty}\frac{p_k(\delta_0)}{k+\delta_j} - \frac{1}{2m+\delta_j} \Bigg).
\end{equation*}
\end{theorem}

See Section~\ref{sec:proof-theor-mle-normality} for the proof of Theorem~\ref{thm:mle:normality}. Remark that in Theorem~\ref{thm:mle:normality} we do not require that the MLE is restricted to a compact set as in \cite{GV17}. This condition was imposed by \cite{GV17} to avoid issues in controlling the score function near the boundary $-m$. Here we circumvent this issue by showing that the score cannot have a zero close to the boundary (and hence the likelihood a maximum).

\subsection{Change-point detection when only \texorpdfstring{$\tau_n$}{} is known}

We now consider the situation where the change-point $\tau_n$ is known but the model parameters $\delta_0$ and $\delta_1$ are unknown. Theorem 3.6 suggests that change-point detection is possible when $\Delta_n$ diverges to $+\infty$ and that the likelihood ratio test guarantees that type I and type II error rates decay to $0$. However, it requires the knowledge of the parameters $\delta_0$ and $\delta_1$ and $\tau_n$. If these two parameters are unknown in advance, we can always try to estimate them and then consider the likelihood-ratio test with plugin estimates of $\delta_0$ and $\delta_1$. One can use the Maximum Likelihood Estimator (MLE) of $(\delta_0, \delta_1)$ derived in the previous section.

In the next we let $Q_{(\tau_n, \delta_0, \delta_1)}^n$ be the distribution of the preferential attachment graph $G_n$ when $\delta(t) = \delta_0\1_{t\leq \tau_n} + \delta_1\1_{t > \tau_n}$. The following theorem shows that when the model parameters $(\delta_0, \delta_1)$ are unknown, the test
\begin{equation*}
  T_n' = \1\Bigg(\frac{\intd Q^n_{(\tau_n, \hat{\delta}_{0,n}, \hat{\delta}_{1,n})}}{\intd Q^n_{(\tau_n, \hat{\delta}_{0,n}, \hat{\delta}_{0,n})}} (G_{n}) > 1 \Bigg)
\end{equation*}
using the MLE $(\hat{\delta}_{0,n}, \hat{\delta}_{1,n})$ is ensured to have vanishing error rates.  In other words plug-in estimates of the parameters $\delta_0$ and $\delta_1$ allow to mimic the asymptotic behavior of the likelihood ratio test.

\begin{theorem}
  \label{thm:labeled:change_detection:unknown_params}
    For every increasing sequence $\tau_n$ such that $\tau_n\to \infty$ and $n-\tau_n = \Delta_n \to \infty$, detection of the change is possible using the test $T_n'$:
    \begin{equation*}
        \EE^{n}_{0}(T_n) + \EE^{n}_{1}(1-T_n) \to 0.
    \end{equation*}
  \end{theorem}

See Section~\ref{sec:proof-theor-changepointunknown} for the proof of Theorem~\ref{thm:labeled:change_detection:unknown_params}.

\subsection{Localization of \texorpdfstring{$\tau_n$}{}}

Finally, we consider the situation where parameters $\delta_0$ and $\delta_1$ are known while $\tau_n$ is unknown. The purpose is to localize the parameter $\tau_n$. The following proposition shows that $\tau_n$ can be localized with an error of order $O\big(\log(n)^{3}\big)$ using the maximum likelihood estimator
\begin{equation*}
  \hat\tau_n%
  \in \argmax_{\tau \in \llbracket 0,n\rrbracket}Q_{(\tau_n,\delta_0,\delta_1)}(\{G_n\}).
\end{equation*}

\begin{proposition}\label{prop:localization_tau}
    For $C> 0$a large enough constant,
    \begin{equation*}
        \PP^{n}_{1}\left(\left|\hat{\tau}_n - \tau_n\right| \leq C \log(n)^{3}\right) \to 1.
      \end{equation*}
    \end{proposition}

See Section~\ref{sec:proof-prop-localiz} for the proof of Proposition~\ref{prop:localization_tau}.

Finally, let us mention that when writing this paper, we were essentially interested in the model where the unlabeled graph is observed. The case where the labeled graph is observed has only been studied to justify the choice of the reduction of the original problem. We were not interested in obtaining the sharpest results in the labeled graph model. For example, we believe that the result of Proposition \ref{prop:localization_tau} can be generalized to include the simultaneous localisation of all model parameters $\delta_0$, $\delta_1$ and $\tau_n$, and eventually reduced from $\log(n)^3$ to constant, but at the cost of some tedious calculations that are outside the scope of this paper.

\section{Supplementary notations}
\label{sec:sm:notations}

We use the same conventions as in the main paper. We furthermore make use of the
following supplementary notations in the subsequent proofs. We write $a_n\lesssim b_n$ to denote $a_n = O(b_n)$. We say that $a_n\asymp b_n$ if there exist constants $c_1, c_2>0$ such that $c_1a_n \leq b_n \leq c_2 a_n$. For sequence of real-valued random variables $(X_n)_{n\geq 1}$ with respective distributions $(P_n)_{n\geq 1}$ and real numbers $c$ we write $X_n \xrightarrow{P_n} c$, $j=0,1$, to say that $\lim_n\PP_j^n( |X_n - c| >\varepsilon ) = 0$ for all $\varepsilon >0$ and we abusively say that $(X_n)_{n\geq 1}$ converges in probability to $c$, even though the random variables $X_n$ may not be necessarily defined on the same probability space. The notation $X_n \overset{P_n}{\rightsquigarrow} X$ stands for convergence in distribution of $(X_n)_{n\geq 1}$ to a random variable $X$.

\section{Proof of Theorem 3.6}\label{Section_5bis}

\subsection{Bounding the sum of the two errors of the likelihood-ratio test}
\label{sec:bounding-sum-two}

Let $Q_0^n$ (respectively $Q_1^n$) denote the law of $G_n$ under $\PP_0^n$ (resp. $\PP_1^n$). The limiting behaviour of $\frac{\intd Q_1^n}{\intd Q_0^n}$ under $\PP_0^n$ is characterized below in Proposition~\ref{prop:asymptotic_likelihood:null}, while Proposition~\ref{prop:asymptotic_likelihood:alt} characterizes its behaviour under $\PP_1^n$. Using Proposition~\ref{prop:asymptotic_likelihood:null}, it is found that
\begin{align*}
  \limsup_{n\to\infty}\PPn\left(\frac{\intd Q_1^n}{\intd Q_0^n}(G_n) > 1\right)%
  &=\PPn\left( \log\left(\frac{\intd Q_1^n}{\intd Q_0^n}(G_n)\right) > 0\right) \\
  &= \limsup_{n\to\infty}\PPn\left(\frac{1}{\Delta_n}\log\left(\frac{\intd Q_1^n}{\intd Q_0^n}(G_n)\right) + \ell^{0}_\infty >  \ell^{0}_{\infty}\right)\\
  &= 0.
\end{align*}
Using Proposition \ref{prop:asymptotic_likelihood:alt}, we prove similarly that for any $K > 0$
\begin{equation*}
  \limsup_{n\to\infty}\PPa\left(\frac{\intd Q_1^n}{\intd Q_0^n}(G_n) \leq  1\right)%
  = 0.
\end{equation*}

\subsection{Regime of contiguity}
\label{sec:regime-contiguity}

Using the Lemma A.4, it is clear that $\Big(\frac{\intd Q_1^n}{\intd Q_0^n}\Big)_{n\geq 1}$ is uniformly bounded below and above when $\limsup_n\Delta_n <+\infty$, and thus $(Q_1^n)_{n\geq 1}$ is contiguous to $(Q_0^n)_{n\geq 1}$.

\subsection{Estimates on the behaviour of the likelihood-ratio under the null and alternative hypothesis}
\label{sec:estim-behav-likel}

The following propositions are used for the proof of Theorem 3.6. Recall that $Q_j^n$ denote the law of $G_n$ under $\PP_j^n$, for $j=0,1$. We recall that $p(\delta)$ is the limiting distribution of the degree distribution of the affine preferential attachment graph with parameter $\delta$ (see also equation~\ref{eq:limitdegreeseq}).

\begin{proposition}\label{prop:asymptotic_likelihood:null}
    Let $\delta_0,\delta_1>-m$ with $\delta_0 \neq \delta_1$. For every increasing sequence $(\tau_n)_{n\geq 1}$ of integer numbers satisfying $0 \leq \tau_n < n$ and $\Delta_n = n-\tau_n \to \infty$, one has
\begin{equation*}
    \frac{1}{\Delta_n}\log\left(\frac{\intd Q_1^n}{\intd Q_0^n}(G_n)\right)\xrightarrow{\PP^{n}_0}-\ell_{\infty}^{0}
\end{equation*}
    where [letting $X\sim p(\delta_0)$]
\begin{equation*}
  \ell_{\infty}^{0} = \frac{m}{2m+\delta_0}\left((2m+\delta_0)\log\left(1 + \frac{\delta_1 - \delta_0}{2m+\delta_0}\right) - \EE\left[(X+\delta_0)\log\left(1 + \frac{\delta_1 - \delta_0}{X+\delta_0}\right)\right]\right) > 0.
\end{equation*}
\end{proposition}
\begin{proof}
  In what follows, we introduce the random variables $D_{t,i} = \D_{\G_{t, i-1}}(\V_{t, i})$ and the filtrations $\mathcal{F}_t = \sigma(G_0,\dots,G_t)$ and $\mathcal{F}_{t,i-1}=\sigma(G_{t,0},\dots,G_{t,i-1})$ as in \cite{GV17} to simplify the notations. We recall that the expression of the likelihood-ratio has been established in Lemma A.4. Normalizing the log-likelihood ratio by $n-\tau_n$, one obtains:
\begin{subequations}
\begin{align}
  \frac{\log\left(\frac{\intd Q_1^n}{\intd Q_0^n}(G_n)\right)}{n-\tau_n}%
  &= \frac{1}{n-\tau_n}\sum_{t=\tau_n + 1}^{n}\sum_{i=1}^{m}\left(\log\left(1 + \frac{\delta_1 - \delta_0}{D_{t,i} + \delta_0}\right) - \frac{\delta_1 - \delta_0}{D_{t,i} + \delta_0}\right)\label{eqn:null:1}\\
    &\quad+\frac{(\delta_1 - \delta_0)}{n-\tau_n}\sum_{t = \tau_n + 1}^{n}\sum_{i=1}^{m}\left(\frac{1}{D_{t,i} + \delta_0} - \frac{t}{S_{t,i-1}(\delta_0)}\right)\label{eqn:null:2}\\
    &\quad-\frac{1}{n-\tau_n}\sum_{t=\tau_n + 1}^{n}\sum_{i=1}^{m}\left(\log\left(1 + \frac{t(\delta_1 - \delta_0)}{S_{t,i-1}(\delta_0)}\right) - \frac{t(\delta_1 - \delta_0)}{S_{t,i-1}(\delta_0)}\right).\label{eqn:null:3}
\end{align}
\end{subequations}
We will control each of the three terms involved in the previous display separately.

\textit{First term \eqref{eqn:null:1}.} This term can be written as:
\begin{equation*}
    \sum_{k=m}^{\infty}\sum_{\substack{\tau_n < t \leq n\\ 1\leq i \leq m}}\left(\log\left(1 + \frac{\delta_1 - \delta_0}{k + \delta_0}\right) - \frac{\delta_1 - \delta_0}{k + \delta_0}\right)\frac{\1_{D_{t,i} = k}}{n-\tau_n} = \sum_{k=m}^{\infty}\left(\log\left(\frac{k + \delta_1}{k+\delta_0}\right) - \frac{\delta_1 - \delta_0}{k+\delta_0}\right)\frac{\sum_{\substack{\tau_n < t \leq n\\ 1\leq i \leq m}}\1_{D_{t,i} = k}}{n-\tau_n}
\end{equation*}
and
\begin{align*}
  \frac{\sum_{\substack{\tau_n < t \leq n\\ 1\leq i \leq m}}\1_{D_{t,i} = k}}{n-\tau_n}%
  &= \frac{\sum_{\substack{\tau_n < t \leq n\\ 1\leq i \leq m}}\left(\1_{D_{t,i} = k} - \EE_0^n\left[\1_{\D_{V_{t, i}} = k}\mid \mathcal{F}_{t, i-1}\right]\right)}{n-\tau_n} + \frac{\sum_{\substack{\tau_n < t \leq n\\ 1\leq i \leq m}}\PP_0^n(D_{t,i} = k\mid \mathcal{F}_{t, i-1})}{n-\tau_n}\\
        &= \frac{\sum_{\substack{\tau_n < t \leq n\\ 1\leq i \leq m}}\left(\1_{D_{t,i} = k} - \EE_0^n\left[\1_{D_{t,i} = k}\mid \mathcal{F}_{t, i-1}\right]\right)}{n-\tau_n} + \frac{k+\delta_0}{n-\tau_n}\sum_{\substack{\tau_n < t \leq n\\ 1\leq i \leq m}}\frac{N_{k}(\G_{t, i-1})}{S_{t,i-1}(\delta_0)}.
\end{align*}
where $N_{k}(\G_{t, i-1})$ is the number of vertices of degree $k$ in the graph after attaching the $(i-1)$-th edge to the vertex $t$ of the graph. On the one hand, by Hoeffding-Azuma inequality, the first term of the equality above converges to $0$ in probability. On the other hand, we have that for all $(t, i)$:
\begin{equation*}
    |N_{k}(\G_{t, i}) - N_{k}(\G_n)| \leq (n-\tau_n)(m+1)    
\end{equation*}
It follows that:
\begin{equation*}
    m \frac{N_{k}(\G_n) - (n-\tau_n)(m+1)}{S_{n,m}(\delta_0)}\leq \frac{1}{n-\tau_n}\sum_{\substack{\tau_n < t \leq n\\ 1\leq i \leq m}}\frac{N_{k}(\G_{t, i-1})}{S_{t,i-1}(\delta_0)}\leq m\frac{N_{k}(\G_n) + (n-\tau_n)(m+1)}{S_{\tau_n+1,0}(\delta_0)}
\end{equation*}
where both sides converge in probability to $\frac{m}{2m+\delta_0}p_k(\delta_0) = \frac{p_{>k}(\delta_0)}{k+\delta_0}$. Thanks to the dominated convergence theorem (Note that the dominated convergence theorem holds also when convergence takes place only in probability):
\begin{equation*}
    \eqref{eqn:null:1}\xrightarrow{\PP_0^n}\sum_{k=m}^{\infty}\left(\log\left(1 + \frac{\delta_1 - \delta_0}{k + \delta_0}\right) - \frac{\delta_1 - \delta_0}{k+\delta_0}\right)p_{>k}(\delta_0).
  \end{equation*}

\textit{Second term \eqref{eqn:null:2}.} First, note that:
    \begin{equation*}
            \EE_0^n\left[\frac{1}{D_{t,i} + \delta_0}\mid \mathcal{F}_{t, i-1}\right] = \sum_{k=m}^{\infty}\frac{1}{k+\delta_0}\frac{(k+\delta_0)N_k(G_{t,i-1})}{S_{t,i-1}(\delta_0)}
            = \frac{t}{S_{t,i-1}(\delta_0)}.
    \end{equation*}
    Given that $0 < \frac{1}{D_{t,i} + \delta_0} \leq \frac{1}{m + \delta_0}$, one can apply the Hoeffding-Azuma inequality and obtain:
    \begin{equation*}
        \eqref{eqn:null:2}\xrightarrow{\PP_0^n} 0.
    \end{equation*}

\textit{Third term \eqref{eqn:null:3}}. Assume $\delta_1 > \delta_0$. Since the function $t\mapsto \log\big(1 + \frac{t(\delta_1 - \delta_0)}{S_{t,i-1}(\delta_0)}\big)$ is non-decreasing for every choice of $t\in\llbracket 1, n\rrbracket$ and $i\in\llbracket 1, m\rrbracket$, one has:
    \begin{equation*}
        \log\left(1 + \frac{(\tau_n + 1)(\delta_1 - \delta_0)}{S_{\tau_n+1,0}(\delta_0)}\right)\leq \log\left(1 + \frac{t(\delta_1 - \delta_0)}{S_{t,i-1}(\delta_0)}\right)\leq \log\left(1 + \frac{n(\delta_1 - \delta_0)}{S_{n,0}(\delta_0)}\right).
    \end{equation*}
    It follows that:
    \begin{equation*}
        \frac{1}{n-\tau_n}\sum_{t=\tau_n + 1}^{n}\sum_{i=1}^{m}\log\left(1 + \frac{t(\delta_1 - \delta_0)}{S_{t,i-1}(\delta_0)}\right) \to m\log\left(1 + \frac{\delta_1 - \delta_0}{2m+\delta_0}\right).
    \end{equation*}
    This convergence holds also when $\delta_1 < \delta_0$. A similar argument yields:
    \begin{equation*}
        \frac{1}{n-\tau_n}\sum_{t=\tau_n + 1}^{n}\sum_{i=1}^{m}\frac{t(\delta_1 - \delta_0)}{S_{t,i-1}(\delta_0)} \to \frac{m(\delta_1 - \delta_0)}{2m+\delta_0}.
    \end{equation*}
    To sum up, one has:
     \begin{equation*}
         \eqref{eqn:null:3} \to -m\left(\log\left(1 + \frac{\delta_1 - \delta_0}{2m+\delta_0}\right) - \frac{\delta_1 - \delta_0}{2m + \delta_0}\right).
       \end{equation*}

Gathering all of the above estimates, it follows that :
\begin{equation*}
    \frac{\log\left(\frac{\intd Q_1^n}{\intd Q_0^n}(\G_n)\right)}{n-\tau_n} \xrightarrow{\PP_0^n} -\ell_{\infty}^{0}
\end{equation*}
where 
\begin{equation*}
    \begin{split}
        \ell_{\infty}^{0} &= m\log\left(1 + \frac{\delta_1 - \delta_0}{2m + \delta_0}\right) - \sum_{k=m}^{\infty}p_{>k}(\delta_0)\log\left(1 + \frac{\delta_1 - \delta_0}{k+\delta_0}\right)\\
        &= \frac{m}{2m+\delta_0}\left((2m+\delta_0)\log\left(1 + \frac{\delta_1 - \delta_0}{2m + \delta_0}\right) - \sum_{k=m}^{\infty}(k+\delta_0)p_{k}(\delta_0)\log\left(1 + \frac{\delta_1 - \delta_0}{k+\delta_0}\right)\right)\\
        &= \frac{m}{2m+\delta_0}\left((2m+\delta_0)\log\left(1 + \frac{\delta_1 - \delta_0}{2m+\delta_0}\right) - \EE\left[(X+\delta_0)\log\left(1 + \frac{\delta_1 - \delta_0}{X+\delta_0}\right)\right]\right)
    \end{split}
\end{equation*}
where $X \sim p(\delta_0) = (p_{k}(\delta_0))_k$. Since $\EE(X) = 2m$ (see for instance \cite[Exercise~8.16]{hofstad_2016}) and when $\delta_0 \neq \delta_1$ the map $x\mapsto x\log\big(1 + \frac{\delta_1 - \delta_0}{x}\big)$ is concave and non-affine on $\mathbb{R}^{+}$ and $p(\delta_0)$ is not a Dirac distribution, it follows that $\ell_{\infty}^{0} > 0$.
\end{proof}

\begin{proposition}\label{prop:asymptotic_likelihood:alt}
    Let $\delta_0,\delta_1>-m$ with $\delta_0 \neq \delta_1$. For every increasing sequence $(\tau_n)_{n\geq 1}$ of integer numbers satisfying $0 \leq \tau_n < n$ and $\Delta_n = n-\tau_n \to \infty$, one has
    \begin{equation*}
        \frac{1}{n-\tau_n}\log\left(\frac{\intd Q_1^n}{\intd Q_0^n}(G_n)\right)\xrightarrow{\PP^{n}_1} \ell_{\infty}^{1}
    \end{equation*}
    where [letting $X\sim p(\delta_0)$]
    \begin{equation*}
      \ell_{\infty}^{1}%
      = -\frac{m}{2m+\delta_1}\left(\EE\left[(X+\delta_1)\log\left(1 + \frac{\delta_0 - \delta_1}{X+\delta_1}\right)\right] - (2m+\delta_1)\log\left(1 + \frac{\delta_0 - \delta_1}{2m+\delta_1}\right)\right) < 0.
    \end{equation*}
\end{proposition}
\begin{proof}
In what follows, we introduce the random variables $D_{t,i} = \D_{\G_{t, i-1}}(\V_{t, i})$ and the filtrations $\mathcal{F}_t = \sigma(G_0,\dots,G_t)$ and $\mathcal{F}_{t,i-1}=\sigma(G_{t,0},\dots,G_{t,i-1})$ as in \cite{GV17} to simplify the notations. Recall the expression of the likelihood-ratio has been established in Lemma A.4. Normalizing the log-likelihood ratio by $n-\tau_n$, one obtains:
\begin{subequations}
\begin{align}
  \frac{1}{n-\tau_n}\log\left(\frac{\intd Q_1^n}{\intd Q_0^n}(\G_n)\right)%
  &= \frac{1}{n-\tau_n}\sum_{t=\tau_n + 1}^{n}\sum_{i=1}^{m}\left(\log\left(1 + \frac{\delta_1 - \delta_0}{D_{t,i} + \delta_0}\right) - \frac{\delta_1 - \delta_0}{D_{t,i} + \delta_0}\right)\label{eqn:alt:1}\\
    &+  \frac{(\delta_1 - \delta_0)}{n-\tau_n}\sum_{t=\tau_n + 1}^{n}\sum_{i=1}^{m}\left(\frac{1}{D_{t,i} + \delta_0} - \frac{t}{S_{t,i-1}(\delta_0)}\right)\label{eqn:alt:2}\\
    &- \frac{1}{n-\tau_n}\sum_{t=\tau_n + 1}^{n}\sum_{i=1}^{m}\left(\log\left(1 + \frac{t(\delta_1 - \delta_0)}{S_{t,i-1}(\delta_0)}\right) - \frac{t(\delta_1 - \delta_0)}{S_{t,i-1}(\delta_0)}\right).\label{eqn:alt:3}
\end{align}
\end{subequations}
We will control each of the three terms involved in the previous display separately.

\textit{First term \eqref{eqn:alt:1}.}
    \begin{equation*}
            \frac{1}{n-\tau_n}\sum_{\substack{\tau_n < t \leq n\\ 1\leq i \leq m}}\left(\log\left(1 + \frac{\delta_1 - \delta_0}{D_{t,i} + \delta_0}\right) - \frac{\delta_1 - \delta_0}{D_{t,i} + \delta_0}\right) = \sum_{k=m}^{+\infty}\left(\log\left(\frac{k + \delta_1}{k+\delta_0}\right) - \frac{\delta_1 - \delta_0}{k+\delta_0}\right)\frac{\sum_{\substack{\tau_n < t \leq n\\ 1\leq i \leq m}}\1_{D_{t,i} = k}}{n-\tau_n}.
    \end{equation*}
    On the one hand with $\mathcal{F}_{t,i-1} = \sigma(G_{t,0},\dots,G_{t,i-1})$
    \begin{equation*}
      \frac{\sum_{\substack{\tau_n < t \leq n\\ 1\leq i \leq m}}\1_{D_{t,i} = k}}{n-\tau_n}%
      = \sum_{\substack{\tau_n < t \leq n\\ 1\leq i \leq m}}\frac{\left(\1_{D_{t,i}= k} - \EE_1^n\left[\1_{D_{t,i} = k}\mid \mathcal{F}_{t, i-1}\right]\right)}{n-\tau_n} + \frac{1}{n-\tau_n}\sum_{\substack{\tau_n < t \leq n\\ 1\leq i \leq m}}\frac{(k+\delta_1)N_{k}(\G_{t, i-1})}{S_{t,i-1}(\delta_1)}.
    \end{equation*}
    The first term converges to $0$ in probability $\PP_1^n$ using Hoeffding-Azuma inequality. On the other hand, we have that for all $(t, i)$:
\begin{equation*}
   |N_{k}(\G_{t, i}) - N_{k}(\G_n)| \leq (n-\tau_n)(m+1).
\end{equation*}
It follows that:
\begin{equation*}
    m \frac{N_{k}(\G_n) - (n-\tau_n)(m+1)}{S_{n,m}(\delta_0)}\leq \frac{1}{n-\tau_n}\sum_{t=\tau_n +1}^{n}\sum_{i=1}^{m}\frac{N_{k}(\G_{t, i-1})}{S_{t,i-1}(\delta_0)}\leq m\frac{N_{k}(\G_n) + (n-\tau_n)(m+1)}{S_{\tau_n+1,0}(\delta_0)}.
\end{equation*}
Thus:
    \begin{equation*}
        \frac{\sum_{t=\tau_n +1}^{n}\sum_{i=1}^{m}\1_{D_{t,i} = k}}{n-\tau_n}\xrightarrow{\PP_1^n} \frac{(k+\delta_1)m}{2m+\delta_1}p_{k}(\delta_0).
    \end{equation*}
    Using dominated convergence theorem, one has:
\begin{equation*}
    \eqref{eqn:alt:1}\xrightarrow{\PP_1^n} \frac{m}{2m+\delta_1}\sum_{k=m}^{\infty}(k+\delta_1)p_{k}(\delta_0)\left(\log\left(1 + \frac{\delta_1 - \delta_0}{k+\delta_0}\right) - \frac{\delta_1 - \delta_0}{k+\delta_0}\right).
\end{equation*}

\textit{Second term \eqref{eqn:alt:2}.} The term 
    \begin{equation*}
        \frac{1}{n-\tau_n}\sum_{t=\tau_n + 1}^{n}\sum_{i=1}^{m}\frac{t(\delta_1 - \delta_0)}{S_{t,i-1}(\delta_0)}
    \end{equation*}
    converges clearly to $\frac{m(\delta_1 - \delta_0)}{2m+\delta_0}$. On the other hand:
    \begin{equation*}
        \begin{split}
            \frac{1}{n-\tau_n}\sum_{t=\tau_n + 1}^{n}\sum_{i=1}^{m}\frac{1}{D_{t,i} + \delta_0} &= \frac{1}{n-\tau_n}\sum_{t=\tau_n + 1}^{n}\sum_{i=1}^{m}\left(\frac{1}{D_{t,i} + \delta_0} - \EE_1^n\left[\frac{1}{D_{t,i} + \delta_0}\mid \mathcal{F}_{t, i-1} \right]\right)\\
            &\quad+ \frac{1}{n-\tau_n}\sum_{t=\tau_n + 1}^{n}\sum_{i=1}^{m}\EE_1^n\left[\frac{1}{D_{t,i} + \delta_0}\mid \mathcal{F}_{t, i-1} \right]\\
            &= \frac{1}{n-\tau_n}\sum_{t=\tau_n + 1}^{n}\sum_{i=1}^{m}\left(\frac{1}{D_{t,i} + \delta_0} - \EE_1^n\left[\frac{1}{D_{t,i} + \delta_0}\mid \mathcal{F}_{t, i-1} \right]\right)\\
            &\quad+\sum_{k=m}^{nm}\frac{k+\delta_1}{k+\delta_0}\frac{1}{n-\tau_n}\sum_{t=\tau_n +1}^{n}\sum_{i=1}^{m}\frac{N_{k}(\G_{t, i-1})}{S_{t,i-1}(\delta_1)}.
        \end{split}
    \end{equation*}
The first term converges in probability to $0$. We will show that:
\begin{equation*}
    \sum_{k=m}^{nm}\frac{k+\delta_1}{k+\delta_0}\frac{1}{n-\tau_n}\sum_{t=\tau_n +1}^{n}\sum_{i=1}^{m}\frac{N_{k}(\G_{t, i-1})}{S_{t,i-1}(\delta_1)}\xrightarrow{\PP_1^n}\frac{m}{2m+\delta_1}\sum_{k=m}^{+\infty}\frac{k+\delta_1}{k+\delta_0}p_{k}(\delta_0).
\end{equation*}
For positive $K$, we have:
\begin{equation*}
    \begin{split}
        &\left|\sum_{k=m}^{nm}\frac{k+\delta_1}{k+\delta_0}\frac{1}{n-\tau_n}\sum_{t=\tau_n +1}^{n}\sum_{i=1}^{m}\frac{N_{k}(\G_{t, i-1})}{S_{t,i-1}(\delta_1)} - \frac{m}{2m+\delta_1}\sum_{k=m}^{+\infty}\frac{k+\delta_1}{k+\delta_0}p_{k}(\delta_0)\right|\\
        &\qquad\qquad\leq \sum_{k=m}^{K}\frac{k+\delta_1}{k+\delta_0}\left|\frac{1}{n-\tau_n}\sum_{t=\tau_n +1}^{n}\sum_{i=1}^{m}\frac{N_{k}(\G_{t, i-1})}{S_{t,i-1}(\delta_1)} -\frac{m}{2m+\delta_1}p_{k}(\delta_0)\right|\\
        &\qquad\qquad\quad+ \frac{m}{2m+\delta_1}\sum_{k=K+1}^{\infty}\frac{k+\delta_1}{k+\delta_0}p_{k}(\delta_0) + \frac{C(\delta_0, \delta_1)}{n-\tau_n}\sum_{t=\tau_n+1}^{n}\sum_{i=1}^{m}\frac{\sum_{k=K+1}^{nm}N_{k}(\G_{t, i-1})}{S_{t,i-1}(\delta_1)}\\
        &\qquad\qquad\leq \sum_{k=m}^{K}\frac{k+\delta_1}{k+\delta_0}\left|\frac{1}{n-\tau_n}\sum_{t=\tau_n +1}^{n}\sum_{i=1}^{m}\frac{N_{k}(\G_{t, i-1})}{S_{t,i-1}(\delta_1)} -\frac{m}{2m+\delta_1}p_{k}(\delta_0)\right|\\
        &\qquad\qquad\quad+ \frac{m}{2m+\delta_1}\sum_{k=K+1}^{\infty}\frac{k+\delta_1}{k+\delta_0}p_{k}(\delta_0) + \frac{C(\delta_0, \delta_1)}{n-\tau_n}\sum_{t=\tau_n +1}^{n}\sum_{i=1}^{m}\frac{N_{>K}(\G_{t, i-1})}{S_{t,i-1}(\delta_1)}\\
        &\qquad\qquad\leq  \sum_{k=m}^{K}\frac{k+\delta_1}{k+\delta_0}\left|\frac{1}{n-\tau_n}\sum_{t=\tau_n +1}^{n}\sum_{i=1}^{m}\frac{N_{k}(\G_{t, i-1})}{S_{t,i-1}(\delta_1)} -\frac{m}{2m+\delta_1}p_{k}(\delta_0)\right|\\
        &\qquad\qquad\quad+ \frac{m}{2m+\delta_1}\sum_{k=K+1}^{\infty}\frac{k+\delta_1}{k+\delta_0}p_{k}(\delta_0) + m C(\delta_0, \delta_1)\frac{N_{>K}(\G_n) + (n-\tau_n)(m+1)}{S_{\tau_n+1,0}(\delta_1)}\\
        &\qquad\qquad\leq \sum_{k=m}^{K}\frac{k+\delta_1}{k+\delta_0}\left|\frac{1}{n-\tau_n}\sum_{t=\tau_n +1}^{n}\sum_{i=1}^{m}\frac{N_{k}(\G_{t, i-1})}{S_{t,i-1}(\delta_1)} -\frac{m}{2m+\delta_1}p_{k}(\delta_0)\right|\\
        &\qquad\qquad\quad+ \frac{m}{2m+\delta_1}\sum_{k=K+1}^{\infty}\frac{k+\delta_1}{k+\delta_0}p_{k}(\delta_0) + m C(\delta_0, \delta_1)\frac{\frac{2mn}{K} + (n-\tau_n)(m+1)}{S_{\tau_{n+1},0}(\delta_1)}
    \end{split}
\end{equation*} 
where $C(\delta_0, \delta_1)$ is a constant depending solely on $\delta_0$ and $\delta_1$. The upper-bound converges in probability to:
\begin{equation*}
    \frac{m}{2m+\delta_1}\sum_{k=K+1}^{\infty}\frac{k+\delta_1}{k+\delta_0}p_{k}(\delta_0) + \frac{2m^{2}C(\delta_0, \delta_1)}{K(2m+\delta_1)}
\end{equation*}
 which can be made arbitrarily small for large values of $K$. We deduce that:
\begin{equation*}
    \eqref{eqn:alt:2}\xrightarrow{\PP_1^n}\frac{m(\delta_1 - \delta_0)}{2m+\delta_1}\sum_{k=m}^{\infty}\frac{k+\delta_1}{k+\delta_0}p_{k}(\delta_0) - \frac{m(\delta_1 - \delta_0)}{2m+\delta_0}.
\end{equation*}

\textit{Third term \eqref{eqn:alt:1}.} Finally, the last term is shown to converge to:
    \begin{equation*}
        \eqref{eqn:alt:3}\to m\left(\frac{\delta_1 - \delta_0}{2m+\delta_0} - \log\left(1 + \frac{\delta_1 - \delta_0}{2m+\delta_0}\right)\right).
    \end{equation*}
It follows that:
\begin{equation*}
    \frac{1}{n-\tau_n}\log\left(\frac{\intd Q_1^n}{\intd Q_0^n}(G_n)\right)\xrightarrow{\PP_1^n} \ell_{\infty}^{1}
\end{equation*}
where 
\begin{equation*}
    \begin{split}
        \ell_{\infty}^{1} &= -\frac{m}{2m+\delta_1}\left(\sum_{k=m}^{\infty}(k+\delta_1)p_{k}(\delta_0)\log\left(1 + \frac{\delta_0 - \delta_1}{k+\delta_1}\right) - (2m+\delta_1)\log\left(1 + \frac{\delta_0 - \delta_1}{2m+\delta_1}\right)\right)\\
        &= -\frac{m}{2m+\delta_1}\left(\EE\left[(X+\delta_1)\log\left(1 + \frac{\delta_0 - \delta_1}{X+\delta_1}\right)\right] - (2m+\delta_1)\log\left(1 + \frac{\delta_0 - \delta_1}{2m+\delta_1}\right)\right)
    \end{split}
\end{equation*}
which can be shown to be positive by a similar argument to that used in Proposition~\ref{prop:asymptotic_likelihood:null}.
\end{proof}

\section{Remaining proofs}
\label{sec:remaining-proofs}

\subsection{Proof of Theorem~\ref{thm:mle:normality}}
\label{sec:proof-theor-mle-normality}

We first remark that the fact that $\hat\delta_{0,n}$ and $\hat\delta_{1,n}$ are asymptotically independent is an immediate consequence of the fact that the likelihood factorizes as the product of a function depending solely on $\delta_0$ and another function depending solely on $\delta_1$. Hence, it is enough to consider separately $\ell_{1:\tau_n}$ and $\ell_{\tau_{n+1}:n}$. But remark that $\ell_{1:\tau_n}$ is the log-likelihood of the model without change-point at size $\tau_n$. Hence, the existence, uniqueness, and asymptotic normality of $(\hat\delta_{0,n})_{n\geq 1}$ follows immediately from the results in \cite{GV17}\footnote{We note however that for practical reasons \cite{GV17} restricts the MLE to some compact set $[-a,b]$ but this is in fact not required, by the argument we develop in Proposition~\ref{pro:nozeroatboundary}.}.

Thus in the next we focus on the analysis of the sequence of maximizers of $\ell_{\tau_n+1,n}$. The proof is standard and mimicks the steps in \cite{GV17}. First define the score as
\begin{equation*}
  \dot\ell_{\tau_n +1:n}(\delta^{\prime})%
  = \sum_{k=m}^{n m}\frac{N_{>k}(G_n) - N_{>k}(G_{\tau_n})}{k+\delta^{\prime}} - \sum_{t=\tau_n +1}^{n}\sum_{i=1}^{m}\frac{t}{(2m + \delta^{\prime})t-2m + i-1}.
\end{equation*}
The Proposition~\ref{prop:likelihood:uniform:cv} below establishes that $\dot\ell_{\tau_n +1:n}(\cdot)$ converges uniformly over $(-m+\varepsilon,+\infty)$ in probability to a function $\iota_1^{\prime}$ (whose expression is given in said proposition) that is monotone decreasing with a unique zero. The Proposition~\ref{pro:nozeroatboundary} shows that with high probability $\dot\ell_{\tau_n+1:n}$ has no zero in $(-m,\varepsilon)$. These facts are exploited hereafter in Proposition~\ref{prop:consistency} to establish the existence and uniqueness (with high probability) and the consistency of $(\delta_{1,n})_{n\geq 1}$. Finally, given the consistency of $(\delta_{1,n})_{n\geq 1}$, we deduce in Section the asymptotic normality using the standard machinery.

\begin{proposition}\label{prop:likelihood:uniform:cv}
  For every $\varepsilon > 0$
\begin{equation*}
    \sup_{\delta\geq -m+\varepsilon}\left|\frac{\dot\ell_{\tau_n +1:n }(\delta)}{n-\tau_n} - \iota^{\prime}_{1}(\delta)\right|\xrightarrow{\PP_1^n}0
\end{equation*}
where [with $X\sim p(\delta_0)$]
\begin{equation*}
    \iota^{\prime}_{1}(\delta) = \frac{2m+\delta_0}{2m+\delta_1}\sum_{k=m}^{\infty}\frac{k+\delta_1}{k+\delta_0}\frac{p_{>k}(\delta_0)}{k+\delta} - \frac{m}{2m+\delta}
    =\frac{m}{2m+\delta_1}\left(\EE\left[\frac{X+\delta_1}{X+\delta}\right] - \frac{\EE\left[X\right] + \delta_1}{\EE\left[X\right] + \delta}\right).
\end{equation*}
\end{proposition}

The proof of Proposition~\ref{prop:likelihood:uniform:cv} is delayed to Section~\ref{sec:proof-prop-uniformcv}.

\begin{proposition}
  \label{pro:nozeroatboundary}
  There exists $\varepsilon_0 > 0$ such that for all $0 < \varepsilon \leq \varepsilon_0$ it holds
  \begin{equation*}
    \PP_1^n\Big( \inf_{\delta' \in (-m,\varepsilon)}\dot\ell_{\tau_n+1:n}(\delta') \geq \Delta_n \Big) \to 1.
  \end{equation*}
\end{proposition}

The proof or Proposition~\ref{pro:nozeroatboundary} is delayed to Section~\ref{sec:proof-proposition-nozeroatboundary}

\begin{proposition}
  \label{prop:consistency}
  For every $(\delta_0,\delta_1) \in (-m,+\infty)$, if $\Delta_n \to \infty$:
  \begin{equation*}
        \hat{\delta}_{1,n} \xrightarrow{\PP_1^n}\delta_1.
    \end{equation*}
\end{proposition}    

The proof of Proposition~\ref{prop:likelihood:uniform:cv} is delayed to Section~\ref{sec:proof-prop-consistency}.

\subsubsection{Asmptotic normality of \texorpdfstring{$(\hat\delta_{1,n})_{n\geq 1}$}{}}
\label{sec:asmpt-norm-hatd}

In what follows, we introduce the random variables $D_{t,i} = \D_{\G_{t, i-1}}(\V_{t, i})$ and the filtrations $\mathcal{F}_t = \sigma(G_0,\dots,G_t)$ and $\mathcal{F}_{t,i-1}=\sigma(G_{t,0},\dots,G_{t,i-1})$ as in \cite{GV17} to simplify the notations. By definition of $\hat{\delta}_{1,n}$, one has:
\begin{equation*}
    \sum_{k=m}^{n m}\frac{N_{>k}(\G_n) - N_{>k}(\G_{\tau_n})}{k+\hat{\delta}_{1,n}} = \sum_{t=\tau_n +1}^{n}\sum_{i=1}^{m}\frac{t}{(2m + \hat{\delta}_{1,n})t-2m + i-1}.
\end{equation*}
It follows that:
\begin{align*}
        &\sum_{k=m}^{nm}\frac{\sum_{t=\tau_n +1}^{n}\sum_{i=1}^{m}\left(\1_{D_{t,i} = k} - \EE_1^n\left[\1_{D_{t,i} = k} \mid \mathcal{F}_{t, i-1}\right]\right)}{k+\hat{\delta}_{1,n}} \\
        &\qquad= \sum_{t=\tau_n + 1}^{n}\sum_{i=1}^{m}\left(\frac{t}{S_{t,i-1}(\hat\delta_{1,n})} - \sum_{k=m}^{nm}\frac{k+\delta_1}{k+\hat{\delta}_{1,n}}\frac{N_{k}(\G_{t, i-1})}{S_{t,i-1}(\delta_1)}\right)\\
        &\qquad= \sum_{t=\tau_n + 1}^{n}\sum_{i=1}^{m}\left(\frac{t}{S_{t,i-1}(\hat\delta_{1,n})} - \frac{t}{S_{t,i-1}(\delta_1)}\right)
        + \sum_{t=\tau_n + 1}^{n}\sum_{i=1}^{m}\left(\frac{t}{S_{t,i-1}(\delta_1)} - \sum_{k=m}^{nm}\frac{k+\delta_1}{k+\hat{\delta}_{1,n}}\frac{N_{k}(\G_{t, i-1})}{S_{t,i-1}(\delta_1)}\right)\\
        &\qquad= (\delta_1 - \hat{\delta}_{1,n})\sum_{t=\tau_n + 1}^{n}\sum_{i=1}^{m}\frac{t^{2}}{S_{t,i-1}(\delta_1)S_{t,i-1}(\hat\delta_{1,n})}
        +\sum_{t=\tau_n + 1}^{n}\sum_{i=1}^{m}\frac{1}{S_{t,i-1}(\delta_1)}\sum_{k=m}^{nm}N_{k}(\G_{t, i-1})\left(1 - \frac{k+\delta_1}{k+\hat{\delta}_{1,n}}\right)\\
        &\qquad= (\delta_1 - \hat{\delta}_{1,n})\sum_{t=\tau_n + 1}^{n}\sum_{i=1}^{m}\bigg[\frac{t^{2}}{S_{t,i-1}(\delta_1) S_{t,i-1}(\hat\delta_{1,n})}
        -\frac{1}{S_{t,i-1}(\delta_1)}\sum_{k=m}^{nm}\frac{N_{k}(\G_{t, i-1})}{k+\hat{\delta}_{1,n}}\bigg]\\
        &\qquad= (\delta_1 - \hat{\delta}_{1,n})\sum_{t=\tau_n + 1}^{n}\sum_{i=1}^{m}\frac{t}{S_{t,i-1}(\delta_1)}\bigg[\frac{t}{S_{t,i-1}(\hat\delta_{1,n})}
        - \sum_{k=m}^{nm}\frac{N_{k}(\G_{t, i-1})}{t(k+\hat{\delta}_{1,n})}\bigg].
\end{align*}
Thus:
\begin{equation*}
    \begin{split}
        \sqrt{n-\tau_n}(\delta_1 - \hat{\delta}_{1,n}) = \frac{\frac{1}{\sqrt{n-\tau_n}}\sum_{k=m}^{nm}\frac{\sum_{t=\tau_n +1}^{n}\sum_{i=1}^{m}\left(\1_{D_{t,i} = k} - \EE_1^n\left[\1_{D_{t,i} = k} \mid \mathcal{F}_{t, i-1}\right]\right)}{k+\hat{\delta}_{1,n}}}{\frac{1}{n-\tau_n}\sum_{t=\tau_n + 1}^{n}\sum_{i=1}^{m}\frac{t}{S_{t,i-1}(\delta_1)}\bigg[\frac{t}{S_{t,i-1}(\hat\delta_{1,n})} - \sum_{k=m}^{\infty}\frac{N_{k}(\G_{t, i-1})}{t(k+\hat{\delta}_{1,n})}\bigg]}.
    \end{split}
\end{equation*}
We will show that the numerator of the previous display is asymptotically normal and the denominator converges to a positive constant. Asymptotic normality of the estimator follows. Let
\begin{align*}
  A(G_n) &= \frac{1}{n-\tau_n}\sum_{t=\tau_n + 1}^{n}\sum_{i=1}^{m}\frac{t}{S_{t,i-1}(\delta_1)}\bigg[\frac{t}{S_{t,i-1}(\hat\delta_{1,n})} - \sum_{k=m}^{\infty}\frac{N_{k}(\G_{t, i-1})}{t(k+\hat{\delta}_{1,n})}\bigg],\\
  B(G_n) &= \frac{1}{\sqrt{n-\tau_n}}\sum_{k=m}^{nm}\frac{\sum_{t=\tau_n +1}^{n}\sum_{i=1}^{m}\left(\1_{D_{t,i} = k} - \EE_1^n\left[\1_{D_{t,i} = k} \mid \mathcal{F}_{t, i-1}\right]\right)}{k+\hat{\delta}_{1,n}},\\
  \tilde{B}(G_n) &= \frac{1}{\sqrt{n-\tau_n}}\sum_{k=m}^{nm}\frac{\sum_{t=\tau_n +1}^{n}\sum_{i=1}^{m}\left(\1_{D_{t,i} = k} - \EE_1^n\left[\1_{D_{t,i} = k} \mid \mathcal{F}_{t, i-1}\right]\right)}{k+\delta_1}.
\end{align*}

\textit{Convergence of $A(G_n)$}. First, note that
\begin{equation*}
    \begin{split}
        \frac{t}{(2m+\hat{\delta}_{1,n})t -2m + i-1} - \sum_{k=m}^{\infty}\frac{N_{k}(\G_{t, i-1})}{t(k+\hat{\delta}_{1,n})} &= \left(\frac{t}{(2m+\hat{\delta}_{1,n})t-2m + i-1} - \frac{1}{2m+\hat{\delta}_{1,n}}\right)\\
        &+ \left(\frac{1}{2m+\hat{\delta}_{1,n}} - \frac{1}{2m+\delta_1}\right) \\
        &-\sum_{k=m}^{\infty}\left(\frac{N_{k}(\G_{t, i-1})}{t(k+\hat{\delta}_{1,n})} - \frac{N_{k}(\G_{t, i-1})}{t(k+\delta_1)}\right)\\
        &-\sum_{k=m}^{nm}\left(\frac{N_{k}(\G_{t, i-1})}{t(k+\delta_1)} - \frac{p_{k}(\delta_0)}{k+\delta_1}\right)\\
        & \frac{1}{2m+\delta_1} - \sum_{k=m}^{nm}\frac{p_{k}(\delta_0)}{k+\delta_1}.
    \end{split}
\end{equation*}
We have :
\begin{equation*}
    \begin{split}
        \frac{\tau_n +1}{(2m+\hat{\delta}_{1,n})n - m-1}\leq \frac{t}{(2m+\hat{\delta}_{1,n})t-2m + i-1}\leq \frac{n}{(2m+\hat{\delta}_{1,n})(\tau_n+1)-2m}.
    \end{split}
\end{equation*}
It follows that:
\begin{equation*}
    \frac{1}{n-\tau_n}\sum_{t=\tau_n +1}^{n}\sum_{i=1}^{m}\frac{t}{S_{t,i-1}(\delta_1)}\left[\frac{t}{(2m+\hat{\delta}_{1,n})t-2m + i-1} - \frac{1}{2m+\hat{\delta}_{1,n}}\right]\xrightarrow{\PP_1^n}0.
\end{equation*}
Similarly:
\begin{equation*}
    \frac{1}{n-\tau_n}\sum_{t=\tau_n +1}^{n}\sum_{i=1}^{m}\frac{t}{S_{t,i-1}(\delta_1)}\left[\frac{1}{2m+\hat{\delta}_{1,n}} - \frac{1}{2m+\delta_1}\right]\xrightarrow{\PP_1^n}0.
\end{equation*}
On the other hand:
\begin{equation*}
    \begin{split}
        \sum_{k=m}^{\infty}\left(\frac{N_{k}(\G_{t, i-1})}{t(k+\hat{\delta}_{1,n})} - \frac{N_{k}(\G_{t, i-1})}{t(k+\delta_1)}\right) &= (\delta_1 - \hat{\delta}_{1,n})\sum_{k=m}^{\infty}\frac{N_{k}(\G_{t, i-1})}{t(k+\delta_1)(k+\hat{\delta}_{1,n})}\\
        &\leq \frac{\delta_1 - \hat{\delta}_{1,n}}{(m+\delta_1)(m+\hat{\delta}_{1,n})}.
    \end{split}
\end{equation*}
It follows that:
\begin{equation*}
    \frac{1}{n-\tau_n}\sum_{t=\tau_n +1}^{n}\sum_{i=1}^{m}\frac{t}{S_{t,i-1}(\delta_1)}\left(\frac{N_{k}(\G_{t, i-1})}{t(k+\hat{\delta}_{1,n})} - \frac{N_{k}(\G_{t, i-1})}{t(k+\delta_1)}\right)\xrightarrow{\PP_1^n}0.
\end{equation*}
The fourth term is smaller than:
\begin{equation*}
    \begin{split}
        &\frac{1}{n-\tau_n}\sum_{t=\tau_n +1}^{n}\sum_{i=1}^{m}\frac{t}{S_{t,i-1}(\delta_1)}\sum_{k=m}^{nm}\frac{1}{k+\delta}\left|\frac{N_{k}(\G_{t, i-1})}{t} - p_{k}(\delta_0)\right|\\
        &\leq \frac{\log(nm+\delta)}{n-\tau_n}\sum_{t=\tau_n +1}^{n}\sum_{i=1}^{m}\frac{t}{S_{t,i-1}(\delta_1)} \max_{m\leq k\leq nm}\left|\frac{N_{k}(\G_{t, i-1})}{t} - p_{k}(\delta_0)\right|
    \end{split}
\end{equation*}
which converges to $0$ in probability by Theorem 1.3 in \cite{DEGH09}. One can then deduce that:
\begin{equation*}
    A(G_n)\xrightarrow{\PP_1^n}\frac{m}{2m+\delta_1}\left[\frac{1}{2m+\delta_1} - \sum_{k=m}^{\infty}\frac{p_{k}(\delta_0)}{k+\delta_1}\right]
\end{equation*}

\textit{Asymptotic normality of $\tilde{B}(G_n)$}. Now we turn our attention to $\tilde{B}$. One has
\begin{equation*}
    \begin{split}
        \tilde{B}(G_n) &= \sum_{t=1}^{n-\tau_n}\sum_{i=1}^{m}\frac{\sum_{k=m}^{nm}\left(\1_{\D_{\V_{t + \tau_n, i}} = k} - \EE_1^n\left[\1_{\D_{\V_{t + \tau_n, i}} = k} \mid \mathcal{F}_{t, i-1}\right]\right)}{\sqrt{n-\tau_n}(k+\delta_1)}.
    \end{split}
\end{equation*} 
Let $Y^{(n)}_{t, i} = \sum_{k=m}^{nm}\frac{\left(\1_{\D_{\V_{t + \tau_n, i}} = k} - \EE_1^n\left[\1_{\D_{\V_{t + \tau_n, i}} = k} \mid \mathcal{F}_{t, i-1}\right]\right)}{\sqrt{n-\tau_n}(k+\delta_1)}$. We now need to show that $\tilde{B}(G_n)$ is asymptotically normal. We will apply proposition 3 of \cite{GV17}. To do this it is enough to prove that
\begin{equation*}
    \begin{split}
        &\sum_{t=1}^{n-\tau_n}\sum_{i=1}^{m} \EE_1^n\left[\left(\sum_{k=m}^{nm}\frac{\left(\1_{\D_{\V_{t + \tau_n, i}} = k} - \EE_1^n\left[\1_{\D_{\V_{t + \tau_n, i}} = k} \mid \mathcal{F}_{t, i-1}\right]\right)}{\sqrt{n-\tau_n}(k+\delta_1)}\right)^{2}\1_{|Y_{t, i}^{(n)}| > \varepsilon}\mid \mathcal{F}_{t, i-1}\right]\xrightarrow{\PP_1^n} 0,\\
        &\sum_{t=1}^{n-\tau_n}\sum_{i=1}^{m} \EE_1^n\left[\left(\sum_{k=m}^{nm}\frac{\left(\1_{\D_{\V_{t + \tau_n, i}} = k} - \EE_1^n\left[\1_{\D_{\V_{t + \tau_n, i}} = k} \mid \mathcal{F}_{t, i-1}\right]\right)}{\sqrt{n-\tau_n}(k+\delta_1)}\right)^{2}\mid \mathcal{F}_{t, i-1}\right]\xrightarrow{\PP_1^n} \nu_1.
    \end{split}
\end{equation*}
The first convergence result is straightforward since for all $t\in\llbracket1, n-\tau_n\rrbracket$ and $i\in\llbracket1, m\rrbracket$, the random variables $Y_{t, i}^{(n)}$ are uniformly bounded by $\frac{2}{\sqrt{n-\tau_n}(m+\delta_1)}$. For the second one, we start by computing the expectations:
\begin{equation*}
    \begin{split}
        &\EE_1^n\left[\left(\sum_{k=m}^{nm}\frac{\left(\1_{\D_{\V_{t + \tau_n, i}} = k} - \EE_1^n\left[\1_{\D_{\V_{t + \tau_n, i}} = k} \mid \mathcal{F}_{t, i-1}\right]\right)}{\sqrt{n-\tau_n}(k+\delta_1)}\right)^{2}\mid \mathcal{F}_{t, i-1}\right]\\
        &= \sum_{k=m}^{nm}\frac{\EE_1^n\left[\left(\1_{\D_{\V_{t + \tau_n, i}} = k} - \EE_1^n\left[\1_{\D_{\V_{t + \tau_n, i}} = k} \mid \mathcal{F}_{t, i-1}\right]\right)^{2}\mid \mathcal{F}_{t, i-1}\right]}{(k+\delta_1)^{2}(n-\tau_n)}\\
        &-\frac{1}{n-\tau_n}\sum_{k\neq k^{\prime}}\frac{\EE_1^n\left[\1_{D_{t,i} = k}\mid \mathcal{F}_{t, i-1}\right]\EE_1^n\left[\1_{D_{t,i} = k^{\prime}}\mid \mathcal{F}_{t, i-1}\right]}{(k^{\prime}+\delta_1)(k+\delta_1)}\\
        &= \sum_{k=m}^{nm}\frac{\EE_1^n\left[\1_{D_{t,i} = k}\mid \mathcal{F}_{t, i-1}\right]}{(k+\delta_1)^{2}(n-\tau_n)} - \frac{1}{n-\tau_n}\left(\sum_{k=m}^{nm}\frac{\EE_1^n\left[\1_{D_{t,i} = k}\mid \mathcal{F}_{t, i-1}\right]}{k+\delta_1}\right)^{2}\\
        &= \sum_{k=m}^{nm}\frac{N_{k}(\G_{t, i-1})}{(n-\tau_n)(k+\delta_1)[S_{t,i-1}(\delta_1)]} - \frac{1}{n-\tau_n}\left(\sum_{k=m}^{nm}\frac{N_{k}(\G_{t, i-1})}{S_{t,i-1}(\delta_1)}\right)^{2}\\
        &= \sum_{k=m}^{nm}\frac{N_{k}(\G_{t, i-1})}{(n-\tau_n)(k+\delta_1)[S_{t,i-1}(\delta_1)]} - \frac{1}{n-\tau_n}\left(\frac{t}{S_{t,i-1}(\delta_1)}\right)^{2}.
    \end{split}
\end{equation*}
Using arguments exactly similar to those of the previous paragraphs, we have that the second sum converges under $\PP_{1}^{n}$ to:
\begin{equation*}
    \nu_1 = \sum_{k=m}^{\infty}\frac{m p_{k}(\delta_0)}{(k+\delta_1)(2m+\delta_1)} - \frac{m}{(2m+\delta_1)^{2}} =  \frac{m}{2m+\delta_1}\left(\sum_{k=m}^{\infty}\frac{p_{k}(\delta_0)}{k+\delta_1}-\frac{1}{2m+\delta_1}\right).
\end{equation*}
By application of Proposition 3 of \cite{GV17}, one has:
\begin{equation*}
    \tilde{B}(G_n)\overset{\PP_1^n}{\rightsquigarrow}\mathcal{N}(0, \nu_1).
\end{equation*}

\textit{Asymptotic normality of $B(G_n)$}. We start by showing that $\tilde{B}(G_n) - B(G_n)$ converges to $0$ in probability.
\begin{equation*}
    \begin{split}
        \tilde{B}(G_n) - B(G_n) &= \sum_{k=m}^{nm}\frac{1}{(k+\delta_1)(k + \hat{\delta}_{1,n})}\left(\frac{(\hat{\delta}_{1,n} - \delta_1)}{\sqrt{n-\tau_n}}\sum_{t = \tau_n +1}^{n}\sum_{i=1}^{m}\left(\1_{D_{t,i} = k} - \EE_{1}\left[\1_{D_{t,i} = k} \mid \mathcal{F}_{t, i-1}\right]\right)\right).
    \end{split}
\end{equation*}
First we show that for all $k\geq m$:
\begin{equation*}
    \frac{(\hat{\delta}_{1,n} - \delta_1)}{\sqrt{n-\tau_n}}\sum_{t = \tau_n +1}^{n}\sum_{i=1}^{m}\left(\1_{D_{t,i} = k} - \EE_1^n\left[\1_{D_{t,i} = k} \mid \mathcal{F}_{t, i-1}\right]\right)\xrightarrow{\PP_1^n} 0.
\end{equation*}
Let $\epsilon > 0$. One has:
\begin{equation*}
    \begin{split}
        &\PP_1^n\left(\left|\frac{(\hat{\delta}_{1,n} - \delta_1)}{\sqrt{n-\tau_n}}\sum_{t = \tau_n +1}^{n}\sum_{i=1}^{m}\left(\1_{D_{t,i} = k} - \EE_1^n\left[\1_{D_{t,i} = k} \mid \mathcal{F}_{t, i-1}\right]\right)\right|\geq \epsilon^{2}\right)\\
        &\leq\PP_1^n\left(\left|\hat{\delta}_{1,n} - \delta_1\right|\geq \frac{\epsilon}{a}\right)\\
        &+ \PP_1^n\left(\left|\frac{\sum_{t = \tau_n +1}^{n}\sum_{i=1}^{m}\1_{D_{t,i} = k} - \EE_1^n\left[\1_{D_{t,i} = k} \mid \mathcal{F}_{t, i-1}\right]}{\sqrt{n-\tau_n}}\right|\geq a\epsilon\right)\\ 
        &\leq \PP_1^n\left(\left|\hat{\delta}_{1,n} - \delta_1\right|\geq \frac{\epsilon}{a}\right) + 2e^{-\frac{2a^{2}\epsilon^{2}}{m}}.
    \end{split}
\end{equation*}
Taking the limit of $n$ to $+\infty$ and then the same for $a$, one obtains the desired convergence. Given that $\hat{\delta}_{1,n}$ is far from $-m$- with probability tending to $1$, one obtains by application of the dominated convergence theorem that:
\begin{equation*}
    \tilde{B}(G_n) - B(G_n)\xrightarrow{\PP_1^n}0.
\end{equation*}
Finally, we apply Slutsky lemma to obtain:
\begin{equation*}
    \sqrt{n-\tau_n}(\delta_1 - \hat{\delta}_{1,n}) \xrightarrow{\PP_1^n} \mathcal{N}(0, \nu_1^{-1})
\end{equation*}
where 
\begin{equation*}
    \nu_1 = \sum_{k=m}^{\infty}\frac{m p_{k}(\delta_0)}{(k+\delta_1)(2m+\delta_1)} - \frac{m}{(2m+\delta_1)^{2}}.
\end{equation*}

\subsubsection{Proof of Proposition \ref{prop:likelihood:uniform:cv}}
\label{sec:proof-prop-uniformcv}

First, one can easily show that:
\begin{equation*}
    \sup_{\delta\geq -m+\epsilon}\left|\frac{1}{n-\tau_n}\sum_{t=\tau_n +1}^{n}\sum_{i=1}^{m}\frac{t}{(2m + \delta)t-2m + i-1} - \frac{m}{2m+\delta}\right|\to 0.
\end{equation*}
For the other sum, we write:
\begin{equation*}
    \begin{split}
        &\left|\frac{1}{n-\tau_n}\sum_{k=m}^{+\infty}\frac{N_{>k}(G_n) - N_{>k}(\G_{\tau_n})}{k+\delta} - \frac{m}{2m+\delta_1}\sum_{k=m}^{+\infty}\frac{k+\delta_1}{k+\delta}p_{k}(\delta_0)\right|\\
        &\leq \sum_{k=m}^{K}\frac{1}{k+\delta}\left|\frac{N_{>k}(G_n) - N_{>k}(\G_{\tau_n})}{n-\tau_n}- \frac{m}{2m+\delta_1}(k+\delta_1)p_{k}(\delta_0)\right|\\
        &+ \sum_{k=K+1}^{+\infty}\frac{1}{k+\delta}\frac{N_{>k}(G_n) - N_{>k}(\G_{\tau_n})}{n-\tau_n} + \frac{m}{2m+\delta_1}\sum_{k=K+1}^{+\infty}\frac{k+\delta_1}{k+\delta}p_{k}(\delta_0)
    \end{split}
\end{equation*}
Taking the supremum over $\delta$ on both sides, one obtains for $K >0$:
\begin{equation*}
    \begin{split}
        &\sup_{\delta\geq -m+\epsilon}\left|\frac{1}{n-\tau_n}\sum_{k=m}^{+\infty}\frac{N_{>k}(G_n) - N_{>k}(\G_{\tau_n})}{k+\delta} - \frac{m}{2m+\delta_1}\sum_{k=m}^{+\infty}\frac{k+\delta_1}{k+\delta}p_{k}(\delta_0)\right|\\
        &\leq \sum_{k=m}^{K}\frac{1}{k-m+\epsilon}\left|\frac{N_{>k}(G_n) - N_{>k}(\G_{\tau_n})}{n-\tau_n}- \frac{m}{2m+\delta_1}(k+\delta_1)p_{k}(\delta_0)\right|\\
        &+ \frac{1}{n-\tau_n}\sum_{t=\tau_n +1}^{n}\sum_{i=1}^{m}\sum_{k=K+1}^{+\infty}\frac{1}{k-m +\epsilon}\1_{D_{t,i} = k} + \frac{m}{2m+\delta_1}\sum_{k=K+1}^{+\infty}\frac{k+\delta_1}{k-m+\epsilon}p_{k}(\delta_0)\\
        &\leq \sum_{k=m}^{K}\frac{1}{k-m+\epsilon}\left|\frac{N_{>k}(G_n) - N_{>k}(\G_{\tau_n})}{n-\tau_n}- \frac{m}{2m+\delta_1}(k+\delta_1)p_{k}(\delta_0)\right|\\
        &+ \frac{m}{K+1+\epsilon-m} + \frac{m}{2m+\delta_1}\sum_{k=K+1}^{+\infty}\frac{k+\delta_1}{k-m+\epsilon}p_{k}(\delta_0)\\
    \end{split}
\end{equation*}
The first term in the upper-bound converges in probability to $0$ and the remaining terms can be made arbitrarily small by choosing a large value for $K$. The convergence in probability result follows:
\begin{equation*}
    \sup_{\delta\geq -m+\epsilon}\left|\frac{\dot\ell_{\tau_n +1: n }(\delta)}{n-\tau_n} - \iota^{\prime}_{1}(\delta)\right|\xrightarrow{\PP^{n}_1}0
\end{equation*}

\subsubsection{Proof of Proposition~\ref{pro:nozeroatboundary}}
\label{sec:proof-proposition-nozeroatboundary}

Based on the expression of the score, we remark that
\begin{align*}
  \dot\ell_{\tau_{n+1}:n}(\delta')%
  &= \sum_{k=m}^{nm}\frac{N_{>k}(G_n) - N_{>k}(G_{\tau_n})}{k+\delta'}%
  - \sum_{t=\tau_{n}+1}^n\sum_{i=1}^m\frac{t}{(2m+\delta')t - 2m+i-1}\\
  &\geq \frac{N_{>m}(G_n) - N_{>m}(G_{\tau_n})}{m + \delta'}%
    - \sum_{t=\tau_n+1}^n\sum_{i=1}^m\frac{t}{m t}\\
  &= \frac{N_{>m}(G_n) - N_{>m}(G_{\tau_n})}{m + \delta'}%
    - \Delta_n.
\end{align*}
Further notice that almost-surely letting $D_{t,i} = \D_{G_{t,i-1}}(V_{t,i})$
\begin{align*}
  N_{>m}(G_n) - N_{>m}(G_{\tau_n})%
  &= \sum_{t=\tau_n+1}^n \sum_{i=1}^m\1(D_{t,i} = m).
\end{align*}
Hence with $\mathcal{F}_{t,i-1}= \sigma(G_{t,i-1})$
\begin{align*}
  \EE_1^n\big(N_{>m}(G_n) - N_{>m}(G_{\tau_n}) \big)%
  &=\sum_{t=\tau_n+1}^n\sum_{i=1}^n\EE_1^n\Big( \PP_1^n(D_{t,i} = m \mid \mathcal{F}_{t,i-1}) \Big)\\
  &=\sum_{t=\tau_n+1}^n\sum_{i=1}^m\EE_1^n\Big(N_m(G_{t,i-1})\frac{m + \delta_1}{S_{t,i-1}(\delta_1)} \Big)\\
  &= (m+\delta_1)\sum_{t=\tau_n+1}^n\sum_{i=1}^m\frac{\EE_1^n(N_m(G_{t,i-1}))}{S_{t,i-1}(\delta_1)}
\end{align*}
Since only $m$ edges are added at each instant $t$, we deduce that $N_m(G_{t,i-1}) \geq N_m(G_t) - m$ for all $t\in \llbracket \tau_n+1,n\rrbracket$ and all $i=1,\dots,m$. Furthermore $\EE(N_m(G_t)) \asymp t p_m(\delta_0)$ (see for instance the computations in \cite{BBCH23}). Hence we deduce that $\EE_1^n\big(N_{>m}(G_n) - N_{>m}(G_{\tau_n}) \big) \geq C\Delta_n$ for a constant $C$ depending only on $(\delta_0,\delta_1)$ and $m$. A standard concentration argument shows that $\dot\ell_{\tau_n+1:n}(\delta_{'}) \geq \Delta_n$ with probability going to one provided $\varepsilon > 0$ is taken small enough.

\subsubsection{Proof of Proposition \ref{prop:consistency}}
\label{sec:proof-prop-consistency}

We first show that:
\begin{equation*}
    \PP_1^n\left(N_{>m}(G_n) - N_{>m}(\G_{\tau_n}) = 0\right)\to 0.
  \end{equation*}
  The starting point is
\begin{equation*}
    \begin{split}
        \PP_1^n\left(N_{>m}(G_n) - N_{>m}(\G_{\tau_n}) = 0\mid \mathcal{F}_{\tau_n}\right) &= \PP_1^n\left(\cap_{t = \tau_n + 1}^{n}\cap_{i=1}^{m}\left\{\D_{G_{t,i-1}}(V_{t,i}) \neq m\right\}\mid \mathcal{F}_{\tau_n}\right)\\ 
        &= \EE_1^n\left[\prod_{t=\tau_n +1}^{n}\prod_{i=1}^{m}\left(1 - \frac{(m+\delta_1)N_{m}(\G_{t, i-1})}{S_{t,i-1}(\delta_1)}\right)\mid \mathcal{F}_{\tau_n}\right]\\
        &= \EE_1^n\left[\exp\left(\sum_{t = \tau_n +1}^{n}\sum_{i=1}^{m}\log\left(1 - \frac{(m+\delta_1)N_{m}(\G_{t, i-1})}{S_{t,i-1}(\delta_1)}\right)\right)\mid \mathcal{F}_{\tau_n}\right]\\
        &\leq \EE_1^n\left[\exp\left(-\sum_{t = \tau_n +1}^{n}\sum_{i=1}^{m}\frac{(m+\delta_1)N_{m}(\G_{t, i-1})}{S_{t,i-1}(\delta_1)}\right)\mid \mathcal{F}_{\tau_n}\right].
    \end{split}
\end{equation*}
It follows that:
\begin{equation*}
     \PP_1^n\left(N_{>m}(G_n) - N_{>m}(\G_{\tau_n}) = 0\right) \leq \EE_1^n\left[\exp\left(-\sum_{t = \tau_n +1}^{n}\sum_{i=1}^{m}\frac{(m+\delta_1)N_{m}(\G_{t, i-1})}{S_{t,i-1}(\delta_1)}\right)\right]\to 0.
\end{equation*}
Note that when $N_{>m}(G_n) - N_{>m}(\G_{\tau_n}) \geq 1$, there exists a deterministic $\eta_0 > 0$ such that $\hat{\delta}_{1,n} > -m + \eta_0$. Finally,
\begin{equation*}
    \begin{split}
        \PP_1^n\left(\left|\iota^{\prime}_{1}(\hat{\delta}_{1,n})\right|\geq \epsilon\right) &\leq \PP_1^n\left(\left|\iota^{\prime}_{1}(\hat{\delta}_{1,n})\right|\geq \epsilon, N_{>m}(G_n) - N_{>m}(\G_{\tau_n})\geq 1\right) + \PP_1^n\left(N_{>m}(G_n) - N_{>m}(\G_{\tau_n}) = 0\right)\\
        &\leq \PP_1^n\left(\sup_{\delta\geq -m + \eta_0}\left|\iota^{\prime}_{1}(\delta)-\frac{\dot\ell_{\tau_n +1\rightarrow n}(\delta)}{n-\tau_n}\right|\geq \epsilon\right) + \PP_1^n\left(N_{>m}(G_n) - N_{>m}(\G_{\tau_n}) = 0\right).
    \end{split}
\end{equation*}
It follows by Proposition \ref{prop:likelihood:uniform:cv} that $\iota_{1}^{\prime}(\hat{\delta}_{1,n})\xrightarrow{\PP_1^n} 0$.

\subsection{Proof of Theorem \ref{thm:labeled:change_detection:unknown_params}}
\label{sec:proof-theor-changepointunknown}

Theorem \ref{thm:labeled:change_detection:unknown_params} is a direct consequence of the following proposition.

\begin{proposition}\label{prop:asymptotic_likelihood:estim}
    Let $\delta_0 \neq \delta_1$. For every increasing sequence $\tau_n < n$ such that $\Delta_n \to \infty$, one has:
\begin{equation*}
    \begin{split}
        &\frac{1}{n-\tau_n}\log\left(\frac{\intd Q^n_{(\tau_n, \hat{\delta}_0, \hat{\delta}_{1,n})}}{\intd Q^n_{(\tau_n, \hat{\delta}_0, \hat{\delta}_0)}}\right)\xrightarrow[n\rightarrow +\infty]{\PP^{n}_0}-\ell_{\infty}^{0} < 0 \\
        &\frac{1}{n-\tau_n}\log\left(\frac{\intd Q^n_{(\tau_n, \hat{\delta}_0, \hat{\delta}_{1,n})}}{\intd Q^n_{(\tau_n, \hat{\delta}_0, \hat{\delta}_0)}}\right)\xrightarrow[n\rightarrow +\infty]{\PP^{n}_1} \ell_{\infty}^{1} > 0
    \end{split}
\end{equation*}
    where $\ell_{\infty}^0$ and $\ell_{\infty}^1$ are defined in Propositions~\ref{prop:asymptotic_likelihood:null} and~\ref{prop:asymptotic_likelihood:alt}.
\end{proposition}

\begin{proof}
    Thanks to Proposition \ref{prop:asymptotic_likelihood:null} and Proposition \ref{prop:asymptotic_likelihood:alt}, one only needs to show that:
    \begin{equation*}
            \frac{1}{n-\tau_n}\log\left(\frac{\intd Q^n_{(\tau_n, \hat{\delta}_0, \hat{\delta}_{1,n})}}{\intd Q^n_{(\tau_n, \hat{\delta}_0, \hat{\delta}_0)}}(G_n)\right) - \frac{1}{n-\tau_n}\log\left(\frac{\intd Q^n_{(\tau_n,\delta_0,\delta_1)}}{\intd Q^n_{(\tau_n,\delta_0,\delta_0)}}(G_n)\right)\xrightarrow{\PP^{n}_{\ell}} 0
    \end{equation*}
    for $\ell\in\left\{0, 1\right\}$. Using the expression of the likelihood ratio in Lemma A.4, one has:
    \begin{equation*}
        \begin{split}
            &\frac{1}{n-\tau_n}\log\left(\frac{\intd Q^n_{(\tau_n, \hat{\delta}_0, \hat{\delta}_{1,n})}}{\intd Q^n_{(\tau_n, \hat{\delta}_0, \hat{\delta}_0)}}(G_n)\right) - \frac{1}{n-\tau_n}\log\left(\frac{\intd Q^n_{(\tau_n,\delta_0,\delta_1)}}{\intd Q^n_{(\tau_n,\delta_0,\delta_0)}}(G_n)\right)\\
            &\sim m\log\left(\frac{2m+\hat{\delta}_0}{2m+\delta_0}\frac{2m+\delta_1}{2m+\hat{\delta}_{1,n}}\right) + \sum_{k=m}^{nm}\frac{N_{>k}(G_n) - N_{>k}(\G_{\tau_n})}{n-\tau_n}\log\left(\frac{k+\hat{\delta}_{1,n}}{k+\delta_1}\frac{k+\delta_0}{k+\hat{\delta}_0}\right).
        \end{split}
    \end{equation*}
    Using Proposition \ref{prop:consistency} and the arguments used in the proof of Proposition \ref{prop:asymptotic_likelihood:null} and Proposition \ref{prop:asymptotic_likelihood:alt}, one can easily deduce that:
    \begin{equation*}
        \frac{1}{n-\tau_n}\log\left(\frac{\intd Q^n_{(\tau_n, \hat{\delta}_0, \hat{\delta}_{1,n})}}{\intd Q^n_{(\tau_n, \hat{\delta}_0, \hat{\delta}_0)}}(G_n)\right) - \frac{1}{n-\tau_n}\log\left(\frac{\intd Q^n_{(\tau_n,\delta_0,\delta_1)}}{\intd Q^n_{(\tau_n,\delta_0,\delta_0)}}(G_n)\right)\xrightarrow{\PP^{n}_{\ell}} 0
    \end{equation*}
    for $\ell\in\left\{0, 1\right\}$.
\end{proof}

\subsection{Proof of Proposition \ref{prop:localization_tau}}
\label{sec:proof-prop-localiz}

In what follows, we introduce the random variables $D_{t,i} = \D_{\G_{t, i-1}}(\V_{t, i})$ and the filtrations $\mathcal{F}_t = \sigma(G_0,\dots,G_t)$ and $\mathcal{F}_{t,i-1}=\sigma(G_{t,0},\dots,G_{t,i-1})$ as in \cite{GV17} to simplify the notations. Let  $\Bar{\tau}_n > \tau_n$, then
\begin{equation*}
    \begin{split}
        \frac{\intd Q^n_{(\Bar{\tau}_n, \delta_0, \delta_1)}}{\intd Q^n_{(\tau_n, \delta_0, \delta_1)}} &=  \prod_{t = \tau_n + 1}^{\Bar{\tau}_n}\prod_{l=1}^{m}\left(\frac{(2m+\delta_1)t-2m +l-1}{(2m+\delta_0)t-2m +l-1}\right)\prod_{k=m}^{nm}\left(\frac{k+\delta_0}{k+\delta_1}\right)^{N_{>k}(\G_{\Bar{\tau}_n}) - N_{>k}(\G_{\tau_n})}\\
        \frac{1}{\Bar{\tau}_{n} - \tau_n}\log\left(\frac{\intd Q^n_{(\Bar{\tau}_n, \delta_0, \delta_1)}}{\intd Q^n_{(\tau_n, \delta_0, \delta_1)}}\right) &= m\log\left(\frac{2m+\delta_1}{2m+\delta_0}\right) + \sum_{k=m}^{nm}\log\left(\frac{k+\delta_0}{k+\delta_1}\right)\frac{N_{>k}(\G_{\Bar{\tau}_n}) - N_{>k}(\G_{\tau_n})}{\Bar{\tau}_n - \tau_n} + \bigO{\frac{1}{n}}.
    \end{split}
\end{equation*}
On the other hand, 
\begin{equation*}
    \begin{split}
        &\sum_{k=m}^{nm}\log\left(\frac{k+\delta_0}{k+\delta_1}\right)\left(N_{>k}(\G_{\Bar{\tau}_n}) - N_{>k}(\G_{\tau_n})\right) = \sum_{k=m}^{nm}\log\left(\frac{k+\delta_0}{k+\delta_1}\right)\sum_{t = \tau_n + 1}^{\Bar{\tau}_n}\sum_{l=1}^{m}\left(\1_{D_{t,l} = k} - \EE_1^n\left[\1_{D_{t,l} = k} \mid \mathcal{F}_{t, l-1}\right]\right)\\
        &+ \sum_{k=m}^{nm}\left(\log\left(\frac{k+\delta_0}{k+\delta_1}\right) - \frac{\delta_0 - \delta_1}{k+\delta_1}\right)\sum_{t = \tau_n + 1}^{\Bar{\tau}_n}\sum_{l=1}^{m} \left(\frac{(k+\delta_1)N_{k}(t, l-1)}{(2m+\delta_1)t} - \frac{(k+\delta_1)p_{k}}{2m+\delta_1} \right)\\
        &+ \sum_{k=m}^{nm}\left(\log\left(\frac{k+\delta_0}{k+\delta_1}\right) - \frac{\delta_0 - \delta_1}{k+\delta_1}\right)\sum_{t = \tau_n + 1}^{\Bar{\tau}_n}\sum_{l=1}^{m} \left(\frac{(k+\delta_1)N_{k}(t, l-1)}{(2m+\delta_1)t -2m + l-1} - \frac{(k+\delta_1)N_{k}(t, l-1)}{(2m+\delta_1)t} \right)\\
        &+ (\delta_0 - \delta_1)\sum_{k=m}^{nm}\sum_{t = \tau_n + 1}^{\Bar{\tau}_n}\sum_{l=1}^{m} \left(\frac{N_{k}(t, l-1)}{(2m+\delta_1)t -2m + l-1} - \frac{p_{k}}{2m+\delta_1} \right)\\
        &+ m(\Bar{\tau}_n - \tau_n)\sum_{k=m}^{nm}\log\left(\frac{k+\delta_0}{k+\delta_1}\right) \frac{(k+\delta_1)p_{k}}{2m+\delta_1}
.    \end{split}
\end{equation*}
Since 
\begin{equation*}
    \begin{split}
        &\left|\sum_{k=m}^{nm}\left(\log\left(\frac{k+\delta_0}{k+\delta_1}\right) - \frac{\delta_0 - \delta_1}{k+\delta_1}\right)\sum_{t = \tau_n + 1}^{\Bar{\tau}_n}\sum_{l=1}^{m} \left(\frac{(k+\delta_1)N_{k}(t, l-1)}{(2m+\delta_1)t -2m + l-1} - \frac{(k+\delta_1)N_{k}(t, l-1)}{(2m+\delta_1)t} \right)\right|\\
        &\leq \sum_{k=m}^{nm}\left|\log\left(\frac{k+\delta_0}{k+\delta_1}\right) - \frac{\delta_0 - \delta_1}{k+\delta_1}\right|(k+\delta_1)N_{k}(t, l-1)\sum_{t = \tau_n + 1}^{\Bar{\tau}_n}\sum_{l=1}^{m} \frac{2m-l+1}{(2m+\delta_1)t \left((2m+\delta_1)t-2m +l-1\right)}\\
        &\lesssim \sum_{k=m}^{nm}\left|\log\left(\frac{k+\delta_0}{k+\delta_1}\right) - \frac{\delta_0 - \delta_1}{k+\delta_1}\right|\frac{n(\Bar{\tau}_n - \tau_n)}{\tau_n ^{2}}\\
        &= \bigO{\frac{\Bar{\tau}_n - \tau_n}{n}}
    \end{split}
\end{equation*}
and
\begin{equation*}
    \begin{split}
        \sum_{k=m}^{nm}\sum_{t = \tau_n + 1}^{\Bar{\tau}_n}\sum_{l=1}^{m} \left(\frac{N_{k}(t, l-1)}{(2m+\delta_1)t -2m + l-1} - \frac{p_{k}}{2m+\delta_1} \right) &= \sum_{t = \tau_n + 1}^{\Bar{\tau}_n}\sum_{l=1}^{m} \left(\frac{t}{(2m+\delta_1)t -2m + l-1} - \frac{1}{2m+\delta_1}\right)\\
        &= \bigO{\frac{\Bar{\tau}_n - \tau_n}{n}}
    \end{split}
\end{equation*}
and

\begin{equation*}
        m(\Bar{\tau}_n - \tau_n)\sum_{k=m}^{nm}\log\left(\frac{k+\delta_0}{k+\delta_1}\right) \frac{(k+\delta_1)p_{k}}{2m+\delta_1} = m(\Bar{\tau}_n - \tau_n)\sum_{k=m}^{+\infty}\log\left(\frac{k+\delta_0}{k+\delta_1}\right) \frac{(k+\delta_1)p_{k}}{2m+\delta_1} + \bigO{\frac{\Bar{\tau}_n - \tau_n}{n}},
\end{equation*}
one obtains
\begin{equation*}
    \begin{split}
        &\frac{1}{\Bar{\tau}_{n} - \tau_n}\log\left(\frac{\intd Q^n_{(\Bar{\tau}_n, \delta_0, \delta_1)}}{\intd Q^n_{(\tau_n, \delta_0, \delta_1)}}\right) = \frac{1}{\Bar{\tau}_n - \tau_n}\sum_{k=m}^{nm}\log\left(\frac{k+\delta_0}{k+\delta_1}\right)\sum_{t = \tau_n + 1}^{\Bar{\tau}_n}\sum_{l=1}^{m}\left(\1_{D_{t,l} = k} - \EE_1^n\left[\1_{D_{t,l} = k} \mid \mathcal{F}_{t, l-1}\right]\right)\\
        &+ \sum_{k=m}^{nm}\left(\log\left(\frac{k+\delta_0}{k+\delta_1}\right) - \frac{\delta_0 - \delta_1}{k+\delta_1}\right)\frac{1}{\Bar{\tau}_n - \tau_n}\sum_{t = \tau_n + 1}^{\Bar{\tau}_n}\sum_{l=1}^{m} \left( \frac{(k+\delta_1)N_{k}(t, l-1)}{(2m+\delta_1)t} - \frac{(k+\delta_1)p_{k}}{2m+\delta_1} \right)\\
        &+ m\log\left(\frac{2m+\delta_1}{2m+\delta_0}\right) + m\sum_{k=m}^{+\infty}\log\left(\frac{k+\delta_0}{k+\delta_1}\right) \frac{(k+\delta_1)p_{k}}{2m+\delta_1} + \bigO{\frac{1}{n}}.\\
    \end{split}
\end{equation*}
Let 
\begin{equation*}
    \ell_{\infty} = m\log\left(\frac{2m+\delta_1}{2m+\delta_0}\right) +\sum_{k=m}^{+\infty}\frac{m(k+\delta_1)}{2m+\delta_1}p_{k}\log\left(\frac{k+\delta_0}{k+\delta_1}\right) < 0.
\end{equation*}
One has:
\begin{equation*}
    \begin{split}
        &\frac{1}{\Bar{\tau}_{n} - \tau_n}\log\left(\frac{\intd Q^n_{(\Bar{\tau}_n, \delta_0, \delta_1)}}{\intd Q^n_{(\tau_n, \delta_0, \delta_1)}}\right) - \ell_{\infty} = \frac{1}{\Bar{\tau}_n - \tau_n}\sum_{k=m}^{nm}\log\left(\frac{k+\delta_0}{k+\delta_1}\right)\sum_{t = \tau_n + 1}^{\Bar{\tau}_n}\sum_{l=1}^{m}\left(\1_{D_{t,l} = k} - \EE_1^n\left[\1_{D_{t,l} = k} \mid \mathcal{F}_{t, l-1}\right]\right)\\
        &+ \sum_{k=m}^{nm}\left(\log\left(\frac{k+\delta_0}{k+\delta_1}\right) - \frac{\delta_0 - \delta_1}{k+\delta_1}\right)\frac{1}{\Bar{\tau}_n - \tau_n}\sum_{t = \tau_n + 1}^{\Bar{\tau}_n}\sum_{l=1}^{m} \left( \frac{(k+\delta_1)N_{k}(t, l-1)}{(2m+\delta_1)t} - \frac{(k+\delta_1)p_{k}}{2m+\delta_1} \right) + \bigO{\frac{1}{n}}\\
        &\leq \log\left(\frac{nm+\delta_0}{m+\delta_1}\right) \sup_{m\leq k \leq nm}\left|\frac{1}{\Bar{\tau}_n  - \tau_n}\sum_{t = \tau_n + 1}^{\Bar{\tau}_n}\sum_{l=1}^{m}\left(\1_{D_{t,l} = k} - \EE_1^n\left[\1_{D_{t,l} = k} \mid \mathcal{F}_{t, l-1}\right]\right)\right|\\
        &+ \sup_{t, l, k}\left|\frac{N_{k}(t, l-1)}{t} - p_{k}\right| \sum_{k=m}^{nm}\left|\log\left(\frac{k+\delta_0}{k+\delta_1}\right) - \frac{\delta_0 - \delta_1}{k+\delta_1}\right|\frac{k+\delta_1}{2m+\delta_1}+ \bigO{\frac{1}{n}}
    \end{split}
\end{equation*}
Since $\left(\exists C = C(m, \delta_0, \delta_1) > 0 \right)$ such that $\sum_{k=m}^{nm}\left|\log\left(\frac{k+\delta_0}{k+\delta_1}\right) - \frac{\delta_0 - \delta_1}{k+\delta_1}\right|\frac{k+\delta_1}{2m+\delta_1} \leq C \log(n)$ and $\log\left(\frac{nm+\delta_0}{m+\delta_1}\right) \leq C\log(n)$ one gets:
\begin{equation*}
    \begin{split}
        \frac{1}{\Bar{\tau}_{n} - \tau_n}\log\left(\frac{\intd Q^n_{(\Bar{\tau}_n, \delta_0, \delta_1)}}{\intd Q^n_{(\tau_n, \delta_0, \delta_1)}}\right) &\leq \ell_{\infty}  + C\log\left(n\right)\sup_{t, l, k}\left|\frac{N_{k}(t, l-1)}{t} - p_{k}\right| + \bigO{\frac{1}{n}}\\
        & + C\log\left(n\right) \sup_{m\leq k \leq nm}\left|\frac{1}{\Bar{\tau}_n  - \tau_n}\sum_{t = \tau_n + 1}^{\Bar{\tau}_n}\sum_{l=1}^{m}\left(\1_{D_{t,l} = k} - \EE_1^n\left[\1_{D_{t,l} = k} \mid \mathcal{F}_{t, l-1}\right]\right)\right|
    \end{split}
\end{equation*}
where $\sup_{t, l, k}$ corresponds to the supremum over $t\in\llbracket \tau_n + 1, \Bar{\tau}_n\rrbracket$, $l\in\llbracket 1, m\rrbracket$ and $k\in\llbracket m, nm\rrbracket$. It follows that there exists $C^{\prime} > 0$ such that:
\begin{equation*}
    \begin{split}
       &\PP_1^n\left(\sup_{\Bar{\tau}_n - \tau_n\geq \kappa_n}\frac{\intd Q^n_{(\Bar{\tau}_n, \delta_0, \delta_1)}}{\intd Q^n_{(\tau_n, \delta_0, \delta_1)}} \geq 1\right)\leq \sum_{\Bar{\tau}_n - \tau_n\geq \kappa_n} \PP_1^n\left(\frac{1}{\Bar{\tau}_n - \tau_n}\log\left(\frac{\intd Q^n_{(\Bar{\tau}_n, \delta_0, \delta_1)}}{\intd Q^n_{(\tau_n, \delta_0, \delta_1)}}\right) \geq 0\right)\\
        &\leq \sum_{\Bar{\tau}_n - \tau_n\geq \kappa_n}\PP_1^n\left(\sup_{m\leq k \leq nm} \left|\frac{1}{\Bar{\tau}_{n} - \tau_n}\sum_{t = \tau_n + 1}^{\Bar{\tau}_{n}}\sum_{l=1}^{m}\left(\1_{D_{t,l} = k} - \EE_{1}\left[\1_{D_{t,l} = k} \mid \mathcal{F}_{t, l-1}\right]\right)\right| \geq \frac{-\ell_{\infty} + \bigO{\frac{1}{n}}}{2C\log\left(n\right)}\right)\\
        &+ \sum_{\Bar{\tau}_n - \tau_n\geq \kappa_n}\PP_1^n\left(\sup_{t, l, k}\left|\frac{N_{k}(t, l-1)}{t} - p_k \right| \geq \frac{-\ell_{\infty} + \bigO{\frac{1}{n}}}{2C\log(n)}\right)\\
        &\lesssim n \left( n\exp\left(-C^{\prime}\frac{\Bar{\tau}_n - \tau_n}{\log(n)^{2}}\right) + \smallO{\frac{1}{n}}\right)\\
        &\lesssim n \left( n\exp\left(-C^{\prime}\frac{\kappa_n}{\log(n)^{2}}\right) + \smallO{\frac{1}{n}}\right)
    \end{split}
\end{equation*}
where we have used Hoeffding-Azuma inequality for the first term and Theorem 8.3 of \cite{hofstad_2016} (or Proposition 2.1 of \cite{DEGH09}) for the second term. Note that these results were stated only in the case of no change. However, they should  remain valid for our model. Using an exactly similar argument for $\Bar{\tau}_n - \tau_n < -\kappa_n$, we finally obtain:

\begin{equation*}
    \begin{split}
        \Tilde{\PP}^{n}_{1}\left(\sup_{\left|\Bar{\tau}_n - \tau_n\right| \geq \kappa_n}\frac{\intd Q^n_{(\Bar{\tau}_n, \delta_0, \delta_1)}}{\intd Q^n_{(\tau_n, \delta_0, \delta_1)}} \geq 1\right) = \bigO{n^{2}\exp\left(-C^{\prime}\frac{\kappa_n}{\log(n)^{2}}\right) + \smallO{1}}
    \end{split}
\end{equation*}
Taking $\kappa_n \asymp \log\left(n\right)^{3}$, one obtains:
\begin{equation*}
    \Tilde{\PP}^{n}_{1}\left(\sup_{\left|\Bar{\tau}_n - \tau_n\right| \gtrsim \kappa_n}\frac{\intd Q^n_{(\Bar{\tau}_n, \delta_0, \delta_1)}}{\intd Q^n_{(\tau_n, \delta_0, \delta_1)}} \geq 1\right)\to 0.
\end{equation*}
One obtains a localization error smaller than $\log(n)^{3}$.


  \putbib%
\end{bibunit}


\end{document}